\documentclass[11pt,a4paper]{article}
\usepackage{amsmath}
\usepackage{amsfonts}
\usepackage{amssymb}
\usepackage{graphicx}
 \usepackage{amsthm}

\usepackage{subcaption}

\usepackage{url}

\usepackage{hyperref}
\usepackage{fancyhdr}
\usepackage[authoryear, round]{natbib}

\usepackage{color}
\begin{document}

\newtheorem{thm}{Theorem}[section]
\newtheorem{theorem}{Theorem}[thm]
\newtheorem{propn}[thm]{Proposition}
\newtheorem{proposition}[thm]{Proposition}
\newtheorem{lemma}[thm]{Lemma}
\newtheorem{eg}[thm]{Example}
\newtheorem{defn}[thm]{Definition}
\newtheorem{definition}[thm]{Definition}
\newtheorem{remark}[thm]{Remark}
\newtheorem{notn}[thm]{Notation}
\newtheorem{condition}[thm]{Condition}
\newtheorem{corollary}[thm]{Corollary}
\newtheorem{conjecture}[thm]{Conjecture}
\newtheorem{assumption}[thm]{Assumption}
\newtheorem{condn}[thm]{Condition}
\newtheorem{example}[thm]{Example}
\newcommand\numberthis{\addtocounter{equation}{1}\tag{\theequation}}

\newcommand{\IP}{\mathbb P}
\newcommand{\IQ}{\mathbb Q}
\newcommand{\IE}{\mathbb E}
\newcommand{\IR}{\mathbb R}
\newcommand{\IZ}{\mathbb Z}
\newcommand{\IN}{\mathbb N}
\newcommand{\IT}{\mathbb T}
\newcommand{\IC}{\mathbb C}

\newcommand{\rmP}{\mathrm{P}}
\newcommand{\rmE}{\mathrm{E}}
\newcommand{\rmT}{\mathrm{T}}
\newcommand{\rmt}{\mathrm{t}}

\newcommand{\tX}{\tilde{X}}

\newcommand{\hX}{\hat{X}}
\newcommand{\hS}{\hat{S}}
\newcommand{\hq}{\hat{q}}
\newcommand{\hQ}{\hat{Q}}

\newcommand{\cP}{\mathcal{P}}
\newcommand{\cO}{\mathcal{O}}
\newcommand{\cA}{\mathcal{A}}
\newcommand{\cL}{\mathcal{L}}
\newcommand{\cG}{\mathcal{G}}
\newcommand{\cC}{\mathcal{C}}
\newcommand{\cV}{\mathcal{V}}
\newcommand{\cE}{\mathcal{E}}
\newcommand{\cT}{\mathcal{T}}
\newcommand{\cD}{\mathcal{D}}
\newcommand{\cZ}{\mathcal{Z}}
\newcommand{\cB}{\mathcal{B}}
\newcommand{\cK}{\mathcal{K}}
\newcommand{\cN}{\mathcal{N}}
\newcommand{\cX}{\mathcal{X}}
\newcommand{\cF}{\mathcal{F}}
\newcommand{\dd}{\mathrm{d}}
\newcommand{\ind}{\mathbf{1}}
\newcommand{\cJ}{\mathcal{J}}

\newcommand{\vs}{\mathbf{s}}
\DeclareRobustCommand{\VAN}[3]{#3}

\pagestyle{fancy}
\lhead{}
\chead{SLFV in a fluctuating environment and SuperBrownian Motion}
\rhead{}

\numberwithin{equation}{section}

\title{\large{\bf Rare mutations in the spatial Lambda-Fleming-Viot model in a fluctuating environment and SuperBrownian Motion}}
                                           
\author{ \begin{small}
\begin{tabular}{ll}                              
Jonathan Chetwynd-Diggle  \thanks{Supported by EPSRC grant
number EP/L015811/1}&
Aleksander Klimek \thanks{Supported by EPSRC grant
number EP/L015811/1} \thanks{Funding for this work has been provided by the Alexander von Humboldt Foundation in the framework of the Sofja Kovalevskaja Award endowed by the German Federal Ministry of Education and Research}
\\   
Department of Mathematics
& Max
Planck
Institute
\\       
 Oxford University & for
Mathematics
in
the
Sciences\\                   
Radcliffe Observatory Quarter& Inselstrasse 22\\                                                         
Oxford OX2 6GG & 04103
Leipzig \\
 UK & DE
 \\  chetwynd-diggle@maths.ox.ac.uk &  klimek@mis.mpg.de   \\
\end{tabular}
\end{small}}


\maketitle

\begin{abstract}

We investigate the behaviour of an establishing mutation which is subject to rapidly fluctuating selection under the Lambda-Fleming-Viot model and show that under a suitable scaling it converges to the Feller diffusion in a random environment. 
We then extend to a population that is distributed across a spatial continuum. In this setting the scaling limit is the SuperBrownian motion in a random environment.
The scaling results for the behaviour of the rare allele are achieved via particle representations which belong to the family of `lookdown constructions'.
This generalises the results obtained for the neutral version of the model by \cite{chetwynd-diggle/etheridge:2017}, which was proved using a duality argument.
To our knowledge this is the first instance of the application of the lookdown approach in which other techniques seem unavailable.
\vspace{.1in}


\vspace{.1in}

\noindent {\bf MSC 20}10 {\bf Subject Classification:}  Primary:
60F05, 
60G57, 
60J25, 
60J68, 
60J80, 
\\
Secondary:  
60G51, 
60G55, 
60J75, 
92D10, 
92D15  

\end{abstract}

\allowdisplaybreaks



\maketitle
\tableofcontents

\section{Introduction}

We address a question of some biological interest: how the frequency of a rare mutation evolves in a spatially distributed population if the direction of selection on that mutation fluctuates in time?
This type of question is particularly relevant in the context of the `Court Jester hypothesis' (see \cite{barnosky:2001}, \cite{benton:2009}), which states that long term improvements in fitness may not occur,
since populations must constantly evolve to keep pace with changes in the environment.

The simplest mathematical framework in which this question can be addressed is given by observation of a single genetic locus. 
For simplicity, we consider a population with two genetic types, the `common type', $\kappa_{c}$, (often referred to as a wild type in biological literature) and the `rare' type, $\kappa_{r}$.
We assume that the rare type forms only a small fraction of the total population. This simple setting may be interpreted as the model for a new mutation before it establishes itself within the population.

Consider for a moment a simple model without selection (for example, a Moran model or a Wright-Fisher model).
In the absence of spatial structure, the absolute number of rare individuals will evolve approximately according to a branching process, which, under appropriate scaling, converges to the Feller diffusion.
We may ask whether a similar phenomenon occurs in the presence of selection, especially when the direction of selection fluctuates rapidly in time. 
One expects that in the latter case the evolution approximately follows a branching process in a random environment, and, under suitable scaling, converges to the Feller diffusion in a random environment. 

It seems natural to try to establish an analogous result for a spatially distributed population. 
It is well known that there are serious difficulties when trying to construct models which incorporate genetic drift in higher dimensional spatial continua, see  \cite{barton/etheridge/veber:2013} for a review. 
A framework which allows us to overcome these difficulties has been found in the spatial Lambda-Fleming-Viot process, introduced in \cite{etheridge:2008} and \cite{barton/etheridge/veber:2010}. 
Diffusion approximations of this model lead to a limit of the Fisher-KPP type, (see e.g. \cite{etheridge/veber/yu:2014}, \cite{forien/penington:2017}) which is consistent with the behaviour of its non-spatial counterpart. 

Therefore, the spatial Lambda-Fleming-Viot model provides a reasonable framework to study the behaviour of the establishment of a mutation. 
Recent work by \cite{chetwynd-diggle/etheridge:2017} for the model without selective advantage or disadvantage for the rare mutation shows convergence.
The limiting object is a superBrownian motion, a measure-valued process introduced independently by \cite{watanabe:1968} and \cite{dawson:1975}, which is the spatial counterpart of the Feller diffusion. 
There is some evidence that superBrownian motion (sometimes referred to as a Dawson-Watanabe superprocess) is a universal scaling limit of critical interacting particle systems, see e.g. \cite{cox/durrett/perkins:2000}, \cite{bramson/cox/legall:2001} and \cite{vanderhofstad/holmes/perkins:2017} and references therein.

Recently \cite{biswas/etheridge/klimek:2018} studied a diffusion approximation of the spatial Lambda-Fleming-Viot with selection in a fluctuating environment.
In contrast to their work we are interested in the behaviour of an establishing mutation within a large population rather than two established populations of comparable size. 
In an analogy to the non-spatial case, we show that the limiting process is the superBrownian motion in a random environment, introduced and studied in \cite{mytnik:1996a}.
The work of \cite{nakashima:2015} shows that superBrownian motion in a random environment is the scaling limit of a model of branching random walks on a lattice in random environment, introduced by  \cite{birkner/geiger/kersting:2005}. We conjecture that superBrownian motion in a random environment is a universal scaling limit for the critical interacting particle systems in random environments. 

The proof of the scaling result in \cite{chetwynd-diggle/etheridge:2017} is based on a duality method. 
However, as discussed in detail in Section~5 of \cite{biswas/etheridge/klimek:2018}, a useful dual process seems to not be available in our setting. The techniques of \cite{biswas/etheridge/klimek:2018} are also not available to us as we are considering a rare mutation.
Therefore, we use a different approach, based on a particle representation which belongs to the family of lookdown constructions.
We build on the work of \cite{kurtz/rodrigues:2011} and
provide a new lookdown construction of superBrownian motion in a random environment.
We then use ideas from \cite{etheridge/kurtz:2014} and a slightly modified version of their construction of the spatial Lambda-Fleming-Viot lookdown which can incorporate a random environment.
 
For other techniques discussed above many difficulties arise with the introduction of spatial continua. However, the vast majority of our work in the non-spatial result transfers to the spatial result without difficulty, see Remark~\ref{sparsity remark} for an explanation.
For this reason, the majority of this paper is devoted to the rigorous derivation of the non-spatial result. 
Our proof technique is based on four main ingredients: lookdown representation, an averaging trick due to \cite{kurtz:1973}, a perturbation result due to \cite{kurtz:1992} (which we recall as Theorem~\ref{kurtzaveraging}) and the Markov Mapping Theorem of \cite{kurtz:1998} (which we recall in Appendix~\ref{Markov Mapping Theorem}).
Previous results using the lookdown approach have shown either the existence of processes under weak conditions, or convergence of processes which had previously been shown to converge through other means. This is, to our knowledge, the first proof of convergence using the lookdown approach in which other techniques are not available.

Lookdown constructions were introduced in \cite{donnelly/kurtz:1996, donnelly/kurtz:1999}.
This approach has proved to be particularly fruitful in applications to population models. 
In  this setting, each individual in the population is assigned  a `level' (taking values in either the integers, as in the original paper of \cite{donnelly/kurtz:1996}, or the reals, as introduced in \cite{kurtz:2000}). 
Levels typically carry information about genealogical relations between individuals.  
The name `lookdown' is used as individuals usually determine their parents by `looking down' at the sub-population with levels lower than their own. 

From a practical perspective, one of the most useful properties of lookdown constructions is that when passing from individual based models to their high density limits, or continuous approximation, the genealogies are preserved.
The importance of this can be seen in the examples of systems of individual based models approximated by the same diffusion processes with very different genealogies obtained in \cite{taylor:2009}. For further examples in the context of Lambda-Fleming-Viot models we refer the reader to \cite{miller:2015}.
For an excellent explanation of the general principle of lookdown constructions we refer to \cite{kurtz/rodrigues:2011}, particularly their death process example in Section~2.1.

In the context of the Spatial Lambda-Fleming-Viot we would like to point to two different approaches. 
The first one is a the construction developed in \cite{etheridge/kurtz:2014}, Section~4.1.3. 
This construction forms the basis for our construction of the SLFV with selection in a fluctuating environment. 
We recall a special case of this construction in Appenidix~\ref{Section Spatial Lambda-Fleming-Viot model}.
The second construction, which was the first lookdown construction for the SLFV, was presented in \cite{veber/wakolbinger:2015}, and was developed using a different approach in \cite{etheridge/kurtz:2014}, Section~4.1.1. 
The later construction is much closer in the flavour to the original ideas of \cite{donnelly/kurtz:1996}.

\subsection{Statement of main results}\label{Statement of main results}

We suppose that the population, which is distributed across $\IR^d$, is
subdivided into two genetic types. We will denote the space of types by $\mathcal{K} = \{\kappa_{c},\kappa_{r}\}$.
Formally, the state of the
population at time $t$ is described by a measure $M_t \in \mathcal{M}$. Where $\mathcal{M}$ is the space of measures whose first marginal is Lebesgue measure
on $\IR^d\times \mathcal{K}$. We note $\mathcal{M}$ is compact when equipped with the topology of weak convergence. 
 
At any fixed time there is a density
$w(t,\cdot) : \IR^d\rightarrow [0,1]$ such that 
$$M_t(dx,d\kappa)=\left(w(t,x)\delta_{\kappa_{r}}(d\kappa)+
(1-w(t,x))\delta_{\kappa_c}(d\kappa)\right) \mathrm{d}x.$$
We interpret $w(t,x)$ as the proportion of population of type $\kappa_{r}$ at location $x$ at time $t$. 
It is defined only up to Lebesgue null set. 
For what follows, it is convenient to fix a representative of $M_{0}$ and update it according to a procedure described below. 
We consider two types of events - neutral and selective events. 
Selective events are influenced by the state of the environment.

We begin with a description of the environment, which is used for all models in this section. Our environment is modelled through a simple random field.
\begin{definition}
\label{Branching -  Environment definition}
Let $\Pi^{env}$ be a Poisson process 
with intensity $E$, dictating the times of the changes in the environment.
Let $q(x,y)$ be a covariance function which belongs to $C_{0}\left(\mathbb{R}^d \times \mathbb{R}^d\right)$ (continuous functions vanishing at infinity) and
let $\{\xi^{(m)}(\cdot)\}_{m\geq 0}$ be a  family of identically distributed 
random fields on $\mathbb{R}^{d}$ such that 
\begin{align*}
\mathbb{P}\left[\xi^{(m)}(x) = -1  \right] = &  \frac{1}{2}=
\mathbb{P}\left[\xi^{(m)}(x) =+1 \right],\\
\mathbb{E}\left[ \xi^{(m)}(x)\xi^{(m)}(y) \right] = & q(x,y).
\end{align*}
Set $\tau_0=0$ and write $\{\tau_m\}_{m\geq 1}$ for the points in $\Pi^{env}$ and define
$$
\zeta(t,\cdot):=\sum_{m=0}^\infty\xi^{(m)}(\cdot)
\mathbf{1}_{[\tau_m,\tau_{m+1})}(t).
$$
\end{definition}
For the construction of the random field in Definition~\ref{Branching -  Environment definition} we refer to  \cite{ma:2009}, especially Example~1. 
We observe that the generator of the process describing the evolution of the environment, $A^{env}$, is given by
\begin{align}
\label{chapter4 environmnet generator}
A^{env}f(\zeta) = 
\mathbb{E}_{\pi}[f(\zeta)] - f(\zeta)
,
\end{align}
where $\pi$ is the stationary distribution of the random field  $\zeta$.

The following version of the Spatial Lambda-Fleming-Viot is a slight modification of the process discussed in \cite{biswas/etheridge/klimek:2018}.

\begin{definition}[Spatial Lambda-Fleming-Viot process with fluctuating selection (SLFVFS)]
\label{SLFVSFE_def}
Let $\mu$ be a $\sigma$-finite measure on $(0, \infty)$ and for each $r\in (0,\infty)$, 
let $\nu_{r}$ be a probability measure on $[0,1]$, such that the 
mapping $r \rightarrow \nu_r$ is measurable and 
\begin{align*}
\int_{(0,\infty)}r^{d}\int_{[0,1]}u\; \nu_{r}(\mathrm{d}u)\mu(\mathrm{d}r) < \infty.
\end{align*} 
Further, fix ${\vs}\in [0,1]$ and 
let $\Pi^{neu}$, $\Pi^{fsel}$,  be independent Poisson point processes on 
$\IR_+\times \IR^d \times (0,\infty) \times [0,1]$ with intensity measures 
$(1-{\vs})\mathrm{d}t \otimes \mathrm{d}x \otimes \mu(\mathrm{d}r)\nu_{r}(\mathrm{d}u)$
and 
${\vs}\mathrm{d}t \otimes \mathrm{d}x \otimes \mu(\mathrm{d}r)\nu_{r}(\mathrm{d}u)$ 
respectively. Let $\Pi^{env}$ be a Poisson process of Definition~\ref{chapter4 environmnet generator}, evolving independently of  $\Pi^{neu}, \Pi^{sel}$.
Let $\sigma(\kappa,\xi): \big(\mathcal{K}\times \{-1,1\}\big) \to \mathbb{R}$ be a function which satisfies the symmetry condition
\begin{align}
\label{sysmetryeqcondition}
\mathbb{E}_{\pi}\left[  
 \frac{\sigma(\kappa_{r},\zeta)}{\sigma(\kappa_{c},\zeta)} - 1
 \right]
 =
 0,
\end{align} 

The {\em spatial Lambda-Fleming-Viot process with fluctuating selection (SLFVFS)}
with driving noises $\Pi^{neu}$, $\Pi^{fsel}$, $\Pi^{env}$, is the 
$\mathcal{M}_{\lambda}$-valued process $M_t$ with
dynamics described as follows. Let $w(t_{-},\cdot)$ be a representative of the 
density of $M_{t-}$ immediately before an event $(t,x,r,u)$ from $\Pi^{neu}$
or $\Pi^{fsel}$. Then the measure $M_t$ immediately after the event has 
density $w(t, \cdot)$ determined by:
\begin{enumerate}
\item If $(t,x,r,u) \in \Pi^{neu}$, a neutral event occurs at time 
$t$ within the closed ball $B(x,r)$. Then
\begin{enumerate}
\item Choose a parental location $l$ according to the uniform distribution on
$B(x,r)$. 
\item Choose the parental type $\kappa \in \{\kappa_r,\kappa_c\}$ according to distribution
\begin{align*}
\mathbb{P}\left[\kappa = \kappa_{r} \right] = w (t_{-},l) , 
\quad \mathbb{P}\left[\kappa =\kappa_c \right] = 1 -  w (t_{-},l).
\end{align*}
\item A proportion $u$ of the population within
$B(x,r)$ dies and is replaced by offspring with type $\kappa$.
Therefore, for each point $y \in B(x,r)$, 
\begin{align*}
w(t,y) = w(t_{-},y)(1 - u) + u\ind_{\{\kappa = \kappa_{r}\}}. 
\end{align*}
\end{enumerate}

\item If $(t,x,r,u) \in \Pi^{fsel}$, a   selective event  occurs at 
time $t$ within the closed ball $B(x,r)$. Then
\begin{enumerate}
\item Choose a parental location $l$ according to the uniform distribution on
$B(x,r)$. 
\item Choose the parental type $\kappa \in \{\kappa_r,\kappa_c\}$ according to 
\begin{align*}
\mathbb{P}\left[\kappa = \kappa_r \right] = \frac{\sigma(\kappa_{r},\zeta)w(t_{-},l_{i})}{\sigma(\kappa_{c},\zeta)w(t_{-},l_{i}) +\sigma(\kappa_{r},\zeta) (1 - w(t_{-},l_{i})}, \\
\mathbb{P}\left[\kappa = \kappa_c \right] =  \frac{\sigma(\kappa_{r},\zeta)(1 - w(t_{-},l_{i}))}{\sigma(\kappa_{c},\zeta)w(t_{-},l_{i}) + \sigma(\kappa_{r},\zeta)(1 - w(t_{-},l_{i})}.
\end{align*}
\item A proportion $u$ of the population within
$B(x,r)$ dies and is replaced by offspring with type $\kappa$.
Therefore, for each point $y \in B(x,r)$, 
\begin{align*}
w(t,y) = w(t_{-},y)(1 - u) + u\ind_{\{\kappa = \kappa_{r}\}}. 
\end{align*}
\end{enumerate}
\end{enumerate}
\end{definition}
The existence of the process follows from the methods of \cite{etheridge/veber/yu:2014} or results of \cite{etheridge/kurtz:2014}, Section~4.1.2.
We note the symmetry condition is not required for existence of the model. This condition can be relaxed if $\sigma$ depends on $N$, however for simplicity, we assume fixed $\sigma$ and so require condition \eqref{sysmetryeqcondition} to be satisfied. Furthermore, we will fix both the impact and radius of events in our models. We assume, with a slight abuse of notation that $\mu = \delta_r$ and $\nu_r = \delta_u$.

We discuss the lookdown representation of the non-spatial version of the model in Section~\ref{Scaling limits of the LFV - dynamics of the rare type}, while Section~\ref{Subsection projected model} explains how the lookdown relates to the underlying process.
The spatial version of the lookdown representation for this process is described in Section~\ref{Scaling limits of the SLFV - dynamics of the rare type}.

We shall define the limiting process, which is a variant of superBrownian motion in a random environment with a drift, in terms of the generator. 
Let $\mathcal{M}_{F}(\mathbb{R}^{d})$ we denote the space of finite measures on $\mathbb{R}^{d}$, again equipped with a topology of weak convergence.

\begin{definition}[SuperBrownian motion in a random environment]
\label{Definition SuperBrownian motion drift}
Let $q(x,y) \in C_{0}\left(\mathbb{R}^d \times \mathbb{R}^d\right)$ be a covariance function. The superBrownian motion in a random environment with a diffusion parameter $m$, a growth parameter $b$ and a quadratic variation parameter $(a,c)$ is the (unique) process $\mu$, taking values in $\mathcal{M}_{f}(\mathbb{R}^{d})$,  characterised by the generator (specified  for $f\in \bar{C}^{2}(R_{+}), \phi \in\mathcal{D}(\Delta)$)
\begin{multline}
\label{SBMRE with drift generator}
\mathcal{L}f(\langle \phi, X_{t}\rangle) =
f^{\prime}(\mu(\phi))\big[\frac{m}{2}\langle \Delta\phi, X_{t}\rangle + b\langle \phi, X_{t}\rangle\big]
\\ 
+ 
\frac{1}{2}f^{\prime\prime}( \langle \phi, X_{s}\rangle)\left( a \langle \phi^{2}, X_{s}\rangle 
+
c
\int_{\mathbb{R}^{d}\times \mathbb{R}^{d}}q(x,y)\phi(x)\phi(y)X_{s}(\mathrm{d}x)X_{s}(\mathrm{d}y)\mathrm{d}s \right)
.
\end{multline}
\end{definition}
This model is discussed in more detail in Section~\ref{Section SuperBrownian motion in a random environment}. In particular, we provide a new lookdown construction of the model, which we derive from the lookdown construction of branching Brownian motion in a random environment.

Our results describe the scaling limit of a sequence of processes.
At the $N$th stage of our scaling, the local population density will
be $K=K(N)$.  We shall denote the representative of the density of the  scaled SLFVSRE by $w^N$ 
and the population of rare individuals by $X^N=Kw^N$, which is defined 
Lebesgue almost everywhere. We shall think of $X^N$ as a measure-valued 
process and abuse notation by writing, for any Borel measurable $\phi$, 
\begin{align*}
\left\langle X_t^N,\phi\right\rangle
 =
K\int_{\IR^d}\phi(x)w_t^N(x)\mathrm{d} x
=
\int_{\mathbb{R}^d}\phi(x)
X_t^N(x)\mathrm{d} x.
\end{align*}
%
The scaling for the neutral part of the model is nearly the same as in \cite{chetwynd-diggle/etheridge:2017}.
The modifications required by the presence of selection and fluctuations in the environment are inspired by \cite{biswas/etheridge/klimek:2018}.
For the scaled process,
time is sped up by a factor $N$, space is
shrunk by $M(N)$, and the impact of each event is reduced by a factor
$J(N)$. 
The rate of environmental changes is multiplied by $\widehat{S}(N)^2$ and
the proportion of selective events is multiplied by $\widehat{S}(N)/S(N)$. 
Scaling of the selective events is motivated by our inclination
 to model short burst of strong selection.
 In the model weak selection limit the rate of selective events is scaled by $1/S(N)$. 
 The additional portion of selective events prevents the action of selection to average out in the diffusive limit.
The Central Limit Theorem suggests the relation between the rate of additional selective events and rate of changes of the environment.
See also \cite{biswas/etheridge/klimek:2018}, Section~3.2 and Section~4.2.
Let  $C(d) := \int_{\cB_1(0)} x^2 \mathrm{d} x$.
We are now in position to state the main result.
\begin{theorem}
\label{Space fluctuating selection theorem}
Suppose that
$X_0^N$ is absolutely continuous with respect to Lebesgue measure with support $\mathrm{supp}(X_0^N)\subseteq D$, where $D$ is a 
compact subset of $\IR^d$ independent of $N$, and  
$X_0^N$ converges weakly to $X_0$.
Moreover, suppose that, as $N$ tends to infinity,
\begin{align*}
\frac{C_{d}ur^{d+2}N}{JM^2} \to C_1 ; \; J,K,M,S,\widehat{S} \to \infty; \;  \frac{K}{JM^{d}} \to 0;  \;  \frac{u^2 V_R NK}{J^2 M^{d}} \to a;
\\
 \frac{N^2}{KJ^{2} M^{d}} \to 0 ; \;
\mathbb{E}_{\pi}\left[ \left\lbrace\frac{suNV_{R}}{SJ} \left(\frac{\sigma(\kappa_r, \zeta)}{\sigma(\kappa_c,\zeta)} - 1 \right)\right\rbrace^{2}\right] \to b^{2}  ; \; \frac{\widehat{S}}{S} \to 0
. 
\end{align*} 
If there exists an $n$ such that, as $N$ tends to infinity,
$N(K/J)^{n} \to 0$ then the sequence $X_{N}(t)$ converges weakly  to superBrownian motion in a random environment with initial condition $X_{0}$, diffusion parameter $C_1$, growth parameter $b^{2}$, quadratic variation parameter $(a,b^{2})$.
\end{theorem}

\begin{remark}
\label{sparsity remark}
We note that the neutral case of our model is analysed in \cite{chetwynd-diggle/etheridge:2017}. Their scaling requires
\begin{align*}
 \frac{N}{JM^2} \to C_1; \; J,K,M \to \infty; \; \frac{NK}{J^2M^d} \to C_2
,
\end{align*}
along with a `sparsity' condition which we discuss below. Their paper discusses the heuristic reason for their scaling which also gives a solid justification for our choice of scaling.

We note that our `sparsity' condition is the requirement $K/JM^d \to 0 $, which is stronger than the one present in \cite{chetwynd-diggle/etheridge:2017}.
Our set of conditions implies $M^2/J \to 0$ in comparison to theirs 
\begin{align*}
\frac{M}{J} \to 0 \mbox{ if }d=1;\; 
\frac{\log M}{J} \to 0 \mbox{ if }d=2;
\;
\frac{1}{J} \to 0 \mbox{ if }d\geq3.
\end{align*}
It is due to the fact that the lookdown construction `sees' the Hausdorff dimension of the support of superBrownian motion in a way that previous work did not, since our set of test functions does not smooth out the support of the process, in sharp contrast to their approach. 
We require the intensity of levels within the ball of radius $r/M$ to tend to infinity in order to allow the limit to have Hausdorff dimension two.
This observation explains why the proof of our results in the spatial case does not differ significantly from the proof in the non-spatial case. 
\end{remark}

\subsection{Structure of the paper}
The rest of the paper is structured as follows. In Section~\ref{Scaling limits of the LFV - dynamics of the rare type} we discuss the Lambda-Fleming-Viot model, studying its scaling limits and discussing the lookdown representation. 
In particular, Section~\ref{Subsection Convergence of the non-spatial model} contains the bulk of our proof.
In Section~\ref{Section SuperBrownian motion in a random environment}, we discuss the lookdown construction for a version of Branching Brownian motion in a random environment and the lookdown representation of the superBrownian motion in a random environment.
In Section~\ref{Scaling limits of the SLFV - dynamics of the rare type}, we discuss how to extend the result of Section~\ref{Scaling limits of the LFV - dynamics of the rare type} to the spatial setup.
Appendix~\ref{Section Poisson random measures} contains some information of Poisson random measures, which are used extensively throughout the paper.
Appendix~\ref{kurtzrodmartingalelemma} briefly recalls a Lemma~A.13 from \cite{kurtz/rodrigues:2011} which ensures our projected process is a solution to the correct martingale problem.
Appendix~\ref{Markov Mapping Theorem} discusses the Markov Mapping Theorem.
Appendix~\ref{Proofs SBMRE BBMRE} contains some of the proofs of Theorems of Section~\ref{Section SuperBrownian motion in a random environment}. 
In Appendix~\ref{Section Spatial Lambda-Fleming-Viot model}, we recall the original construction of \cite{etheridge/kurtz:2014}.
\section{Scaling limits of the LFV - dynamics of the rare type}
\label{Scaling limits of the LFV - dynamics of the rare type}

In this section we are interested in describing the evolution of a subpopulation with a rare mutation within a population which evolves according to the $\Lambda$-Fleming-Viot model with selection in  a fluctuating environment (LFVSFE).
We will again consider a type space with two types, rare and common.
Which of those two types has higher fitness changes with the environment. 
We show that the evolution of the rare subpopulation follows a Feller diffusion in a random environment.

We provide a new description of the LFVSFE in terms of a lookdown construction. 
This construction is inspired by the lookdown construction of the neutral model in Section~4.1.3 of \cite{etheridge/kurtz:2014}.

Let us begin with a description of the model which gives some insight into its construction before defining the model precisely in Definition~\ref{Lookdown representation of LFVSRE}.
Each of the individuals in the population is assigned a \emph{genetic type} $\kappa$ from the set $\cK$ and a \emph{level} $l \in \mathbb{R}^{+} \cup \{0\}$. 
We restrict our attention to $\mathcal{K} = \{\kappa_{c}, \kappa_{r} \}$, which we refer to as the `common' and `rare' type, respectively.
Since in this section we consider a model without spatial structure, the state of the population can be represented as a collection of points, $\eta = \{ (l, \kappa)\}$, or as a measure which can be written, with a slight abuse of notation, as
\begin{align*}
\eta = \sum_{(l, \kappa) \in \eta  }\delta_{(l, \kappa) }, \quad \text{where } (l, \kappa)  \in \left(\mathbb{R}^{+} \cup \{0\}, \cK \right).
\end{align*}
We assume that the  process with levels  $\eta$ is always a conditionally Poisson system with Cox measure $m_{leb} \times \Xi$, where $m_{leb}$ is Lebesgue measure on $\IR^+$. This is because, in our lookdown model, when our initial condition has this form then our process will have this form for all subsequent times.
This is a key feature in how the lookdown construction relates to its underlying model.

For lookdown representations we consider  test functions of the form
\begin{align*}
f(\eta) = \prod_{l \in \eta}g(l),
\end{align*}
with the additional requirement that there exists a $\lambda_{g} > 0$ such that $g(l) = 1$ for all $l > \lambda_{g}$.
This set of test functions will be used throughout this paper for computations involving lookdown representations.

We are interested in the situation 
when the type which is selectively advantageous
depends on the environment. 
We write $\zeta$ for the random process which models the state of the environment. 
 Therefore the full state of the model at time $t$ is given by a pair $(\eta_{t},\zeta_{t})$.

We now proceed to carefully state our non-spatial model, the scaling and the non-spatial results. The following statements are almost identical to those found in Section~\ref{Statement of main results} but we include them for completeness. This further demonstrates the similarity between the spatial and non-spatial methods when using the lookdown construction and motivates why our spatial proof mainly lies within Proposition~\ref{Proposition Neutral Convergence In Space} and Proposition~\ref{selectivepropositionAspace}.

The evolution of the population is determined by reproduction events
of two types 
 - neutral and selective.
We assume that a proportion, $s$, of events are selective and favour one of the two types. 
Events are driven by independent Poisson processes $\Pi^{neu}$ and $\Pi^{sel}$.
For simplicity we assume that the impact of the events is fixed and equal to $u$.

Both neutral and selective events are composed of two elements - discrete births and thinning. 
The birth phase of the events differs between neutral and selective events, while the thinning phase is the same.

Whenever $t \in \Pi^{neu}$, a birth event produces offspring, with levels distributed according to an independent Poisson point process with intensity $u$.
Let $v^{*}$ be the smallest of the new levels. 
Let $(l^{*}_{neu}, \kappa^{*})$ denote the element of $\eta$ with the smallest level greater than $v^{*}$, that is 
\begin{align*}
l^{*}_{neu}  = \min\{ l  : (l, \kappa) \in \eta, l > v^{*}   \}.
\end{align*}
The individual $(l^{*}_{neu}, \kappa^{*})$ is chosen as the parent of the event and removed from the population.
All new individuals are assigned type $\kappa^*$, the type of the parent.
The levels of all old individuals in the population are changed. 
If the level of the individual was smaller than $v^{*}$ it remains unaffected by the birth.
If the level of the individual was larger than  $v^{*}$, it is moved to $l - l^{*}_{neu} + v^{*}$.

If the event is selective, the situation is more complicated.
We introduce an additional function $\sigma(\kappa,\zeta)$, which influences the likelihood of an individual of  type $\kappa$ being parent, given the state of the environment, $\zeta$.
When the environment is in state  $\zeta$, the higher the value of $\sigma(\kappa,\zeta)$, the more likely an individual of type $\kappa$   is to be selected as a parent during selective events in environment $\zeta$.

Whenever $t \in \Pi^{sel}$, a birth event produces offspring, with levels distributed according to an independent Poisson  process with intensity $u$.
As before, let $v^{*}$ be the smallest of the new levels. 
Let $(l^{*}_{sel}, \kappa^{*})$ denote the element of $\eta$ which obtains the minimum
\begin{align}
\min \left\lbrace \frac{l_{i} - v^{*}}{\sigma(\kappa_{i},\zeta)} : (l_i, \kappa_i) \in \eta, l_i > v^*  \right\rbrace
.
\end{align}
The individual $(l^{*}_{sel}, \kappa^{*})$ is chosen as the parent of the event and removed from the population.
All new individuals are assigned the parent's type as in neutral events.
The levels of all old individuals in the population are then changed. 
If the level of the individual was smaller than $v^{*}$ it remains unaffected by the birth.
If the level of the individual was larger than  $v^{*}$, it is moved to $\sigma(\kappa,\zeta)(l - l^{*}_{neu} + v^{*})/\sigma(\kappa^{*},\zeta)$.

For both neutral and selective events once the parent has been selected, and the old levels moved to their new locations, thinning takes place. 
Thinning does not affect the new individuals. 
The thinning takes the level of each individual which is neither a child or the parent of the event present within the population and multiplies it by $1/(1-u)$. 
The combined effect of this movement is defined explicitly through the function $\cJ_{neu}$, defined in \eqref{Jcaldeffinition}, for neutral events and $\cJ_{sel}$, defined in \eqref{Jcal sel deffinition}, for selective events.
We note that instead of removing the parent from the population you can consider the parent to be the lowest offspring instead.
The movement of the levels is chosen in this way to maintain levels with a conditionally Poisson system after any event.

We now define the main process of interest.

\begin{definition}[Lookdown representation of LFVSFE]
\label{Lookdown representation of LFVSRE}
Fix $s \in (0,1)$. Let $\Pi^{neu}, \Pi^{sel}$ be a pair of  independent Poisson point  processes with intensity measures $(1-s)\mathrm{d}t \otimes \nu(\mathrm{d}u)$ and  $s\mathrm{d}t \otimes \nu(\mathrm{d}u)$ respectively on $\IR^+ \times (0,1)$.
Moreover, let $\Pi^{env}$ be a Poisson process with rate $E$, independent of $\Pi^{neu}, \Pi^{sel}$.
Let $\sigma: \mathcal{K} \times \{-1, 1\} \rightarrow \mathbb{R}$ be a function. 

The lookdown representation of LFVSRE is the process taking values in purely atomic measures on $\mathbb{R}\times\mathcal{K} \times \{-1,1\}$ with  dynamics described as follows.
\begin{enumerate}
\item If $(t,u) \in \Pi^{neu}$  
	\begin{enumerate}
	\item a group of new individuals with levels $(v_{1}, v_{2}, \dots)$ is added to the 			population. Their levels are distributed according to a Poisson process with intensity $u$.
	\item Let $v^{*} = \min\{ v_{1}, v_{2}, \dots \}$.
	The type of the new individuals is chosen to be the same as the type of the individual $(\kappa^{*},l^{*})$ whose level is the lowest above $v^{*}$, that is 
	\begin{align}\label{l* sel definition}
l_{neu}^{*} = \min \{ l : (l,\kappa) \in \eta, l>v^{*} \}.
\end{align}
	\item As a result of an event the levels with position $l$ before an event will have new position given by
	\begin{align}
\label{Jcaldeffinition}
\cJ_{neu}(l ,l_{neu}^{*},v^*)=
\begin{cases}
\frac{1}{1 - u} (l - (l_{neu}^{*} - v^*))
& \text{ if } l > l_{neu}^{*}
,
\\
\frac{1}{1 - u} l
& \text{ if } l <l _{neu}^{*}
,
\\
v^*
& \text{ if } l = l_{neu}^{*}
.
\end{cases}
\end{align}
	\end{enumerate}

\item If $(t,u) \in \Pi^{sel}$  
	\begin{enumerate}
	\item a group of new individuals with levels $(v_{1}, v_{2}, \dots)$ is added to the 			population. Their levels are distributed according to a Poisson process with intensity $u$.
	\item Let $v^{*} = \min\{ v_{1}, v_{2}, \dots \}$.
	The type of the new individuals is chosen to be the same as the type of the individual $(l^{*}, \kappa^{*})$ whose level minimizes
	\begin{align}
\left\lbrace \frac{l_{i} - v^{*}}{\sigma(\kappa_{i},\zeta)} : (l_i, \kappa_i) \in \eta, l_i > v^*  \right\rbrace
.
\end{align}
	\item  As a result of an event the levels with position $l$ and type $\kappa$ before an event will maintain their type but will have new position given by
	\begin{align}
\label{Jcal sel deffinition}
\cJ_{sel}((l, \kappa),(l_{sel}^{*},\kappa^*),\zeta,v^*)
=
\begin{cases}
v^* & \text{ if } l = l_{sel}^{*}
,
\\
\frac{1}{1 - u}
\left(
l - (l^* - v^*)\frac{\sigma(\kappa,\zeta)}{\sigma(\kappa^*,\zeta)}
\right)
& \text{ if } l \neq l_{sel}^{*}, l > v^*
,
\\
\frac{1}{1 - u}
l
& \text{ if } l < v^*
.
\end{cases}
\end{align}
\end{enumerate}
\item If $t \in \Pi^{env}$, the environmental variable $\zeta_{t}$ is resampled uniformly from $\{-1, 1\}$.
\end{enumerate}
\end{definition}

As in the spatial case our results require a symmetry condition on the values of $\sigma(\kappa, \zeta)$. 
To be more precise, we require that 
\begin{align}
\label{0d symmetry condition}
\mathbb{E}_{\pi}\left[  
 \frac{\sigma(\kappa_{r},\zeta)}{\sigma(\kappa_{c},\zeta)} - 1
 \right]
 =
 0,
\end{align} 
where $\pi$ is the stationary distribution of $\zeta$.

We note that the generator of the process from Definition~\ref{Lookdown representation of LFVSRE} is given by
\begin{align}
\label{fluctuating generator unscalled}
A f(\eta,\zeta) = A_{neu}f(\eta,\zeta) + A_{sel}f(\eta,\zeta) + A_{env}f(\eta,\zeta),
\end{align}
where 
\begin{align*}
A_{neu}f(\eta,\zeta) &
\\
= (1-s)&
\int_0^\infty
\Bigg[
u e^{- uv^*} g(\kappa^*, v^*)
e^{- u \int_{v^*}^{\infty} (1 - g(\kappa^*, v)
) \mathrm{d} v} 
\prod_{ (\kappa , l) \in \eta , l \neq l^* }
g \left( \kappa , \mathcal{J}_{neu}(l,l_{neu}^*,v^*) \right)
\Bigg] \mathrm{d} v^*
 - f(\eta),
 \\
A_{sel}f(\eta,\zeta) &
= s \int_0^\infty
\Bigg[
u e^{- uv^*} g(\kappa^*, v^*)
e^{- u \int_{v^*}^{\infty} (1 - g(\kappa^*, v)
) \mathrm{d} v} 
\\
&\phantom{u e^{- uv^*} g(\kappa^*, v^*)}
\times
\prod_{ (\kappa , l) \in \eta , l \neq l^*_{sel} }
g \left( \kappa , \cJ_{sel}((l, \kappa),(l_{sel}^{*},\kappa^*),\zeta,v^*) \right)
\Bigg] \mathrm{d} v^*
 - f(\eta),
\\
A_{env}f(\eta,\zeta) &= \mathbb{E}_{\pi}[f(\eta,\zeta)]- f(\eta,\zeta).
\end{align*}
Observe that the only differences between neutral and selective events are the choice of the parent $l^*$ and the movement of levels.

\subsection{Scaling}

We are interested in the evolution of the subpopulation of a rare type within the population evolving according to LFVSFE.
To quantify the rarity, we consider a population with total density $K$ (which is equal to one for the usual Lambda-Fleming-Viot model), and will let $K$ tend to infinity. 
We wish the rare type to make up about $\cO(1/K)$ of the population at each level of the scaling.
This is represented in the look-down process by the intensity of levels. 
The levels of individuals are initially Poisson distributed with intensity $K$ with individuals given the rare type with probability $\cO(1/K)$ and given the common type otherwise.

In order to recover the correct scaling limit we need to readjust our parameters. 
Let $w^{N}$ denote the proportion of individuals of the rare type at the $N^{th}$ stage of scaling.
Let $X^{N} = Kw^{N}$ denote the size of the population of the rare type.
Our scaling limit describes the behaviour of $X^{N}$.
The right scaling of the other parameters is suggested by \cite{chetwynd-diggle/etheridge:2017} and \cite{biswas/etheridge/klimek:2018} as discussed in Remark~\ref{sparsity remark}. 
We speed up the reproduction rate by $N$ and increase the total population size to $K(N)$, but scale down both the impact and the selection coefficient. 
The impact of an event at the $N^{th}$ stage of the approximation will be given by $u/J(N)$. The selection coefficient in the presence of fluctuations will be $s \widehat{S}(N)/S(N)$ and $s/S(N)$ in the absence of fluctuations.
In order to simplify the notation we drop the explicit dependence of scaling parameters on $N$ in what follows.
For our results to hold, certain relations between the parameters need to be satisfied. 
In  Theorem~\ref{NS fluctuating selection theorem}, Theorem~\ref{NS neutral theorem} and Theorem~\ref{NS selection theorem} we specify scaling limits for the model with fluctuating selection, the model where the direction of selection does not change and the neutral model respectively.
In particular, 
our results require a specific relation between the rate of the changes in the environment and the rate of selection. 
We therefore assume that the rate of the environmental  events is  $\widehat{S}^{2}$.

The generator of the scaled process can then be written as 
\begin{align}
\label{fluctuating generator scaled}
A^N f(\eta,\zeta) = A^N_{neu}f(\eta,\zeta) + \widehat{S}A^N_{sel}f(\eta,\zeta) + \widehat{S}^{2}A^N_{env}f(\eta,\zeta),
\end{align}
where
\begin{align*}
A^N_{neu} f(\eta, \zeta)
&=
N
\Bigg(
\int_0^\infty
\Bigg[
\frac{uK}{J} e^{- \frac{uK}{J}v^*} g(\kappa^*, v^*)
e^{- \frac{uK}{J} \int_{v^*}^{\infty} (1 - g(\kappa^*, v)
) \mathrm{d} v} 
\\
&\phantom{A^N_{neu} f(\eta)A^N_{neu} f(\eta)}
\times
\prod_{ (\kappa , l) \in \eta , l \neq l^* }
g \left( \kappa , \mathcal{J}_{neu}(l,l^*,v^*) \right)
\Bigg] \mathrm{d} v^*
 - f(\eta) \Bigg)
,
\\
A^N_{sel}f(\eta, \zeta)
&=
\frac{sN}{S}
\Bigg(
\int_0^\infty
\Bigg[
\frac{uK}{J} e^{- \frac{uK}{J}v^*} g(\kappa^*, v^*)
e^{- \frac{uK}{J} \int_{v^*}^{\infty} (1 - g(\kappa^*, v)
) \mathrm{d} v}
\\
&\phantom{A^N_{neu} f(\eta)A^N_{neu} f(\eta)}\times
\prod_{ (\kappa , l) \in \eta , l \neq l^*_{sel} }
g \left( \kappa , \mathcal{J}_{sel}(l,l^*_{sel},v^*) \right)
\Bigg] \mathrm{d} v^*
 - f(\eta) \Bigg)
,
\\
A^N_{env} f(\eta, \zeta)
&= 
\IE_{\pi}[
f(\eta, \zeta)
]
-
f(\eta, \zeta)
.
\end{align*}
\subsection{Main result of this section}
\label{Subsection Statement of main results}
We now state the main results of this section. 
We recall the definitions of some of the classical models in terms of their generators, in order to state the results formally.
\begin{definition}[Feller diffusion]
Let  $a,b>0$.
The Feller diffusion is the process taking values in $\mathbb{R}$ with generator $C_{fd}$, defined for every $f \in C^{\infty}(\mathbb{R})$, given by
\begin{align*}
C_{fd}f(y) = ayf^{\prime\prime}(y) + byf^{\prime}(y).
\end{align*}
Its lookdown representation is characterised by the process with generator $A_{fd}$ given by
\begin{align}
\label{Kurtz-Rodrigues limiting generator no fluctuations}
A_{fd}f(\eta) = f(\eta)\sum_{i}2a\int_{l_{i}}^{\infty}\big(g(v)  - 1\big)\mathrm{d}v + f(\eta)\sum_{i}\big(al_{i}^{2}  - bl_{i}\big)\frac{g^{\prime}(l_{i})}{g(l_{i})}.
\end{align}
\end{definition}

\begin{definition}[Feller diffusion in random environment]
Let  $a,b>0$.
The Feller diffusion in a random environment is the process taking values in $\mathbb{R}$ with generator $C_{fdr}$, defined for every $f \in C^{\infty}(\mathbb{R})$, of the form
\begin{align*}
C_{fdr}f(y) = \big(ay + b^{2}y^{2} \big)f^{\prime\prime}(y) + b^{2}yf^{\prime}(y).
\end{align*}
Its lookdown representation is characterised by the process with generator $A$ given by
\begin{multline}
\label{Kurtz-Rodrigues limiting generator}
Af(\eta) = f(\eta)\sum_{i}2a\int_{l_{i}}^{\infty}\big(g(v)  - 1\big)\mathrm{d}v + f(\eta)\sum_{i}\big(al_{i}^{2}  - b^{2}l_{i}\big)\frac{g^{\prime}(l_{i})}{g(l_{i})}
\\
+
b^{2}f(\eta)\sum_{j} \left( \sum_{i \ne j}l_{j}l_{i}\frac{g^{\prime}(l_{i})g^{\prime}(l_{j})}{g(l_{i})g(l_{j})} + l_{j}^{2}\frac{g^{\prime\prime}(l_{i})}{g(l_{i})}\right).
\end{multline}
\end{definition}
For a detailed discussion of the lookdown constructions for the Feller diffusion and the Feller diffusion in a random environment we refer to \cite{kurtz/rodrigues:2011}, Section~2.

\begin{remark}
We observe that our choice of test functions guarantees that since 
\begin{align*}
 \sum_{i}\left(
g^{\prime}(l_{i})\prod_{j \ne i}g(l_{j})
\right)
= f(l) \sum_{i}\left(
\frac{g^{\prime}(l_{i})}{g(l_{i})}
\right),
\end{align*}
all terms appearing in \eqref{Kurtz-Rodrigues limiting generator no fluctuations} and \eqref{Kurtz-Rodrigues limiting generator} are well-defined, even if $g(l_i) = 0$.
\end{remark}

\begin{theorem}
\label{NS fluctuating selection theorem}
Let $X^{N}(t)$ denote the total intensity of individuals of rare type at time $t$.
Suppose that $X^{N}_{0}$ converges to $X_{0}$.
Moreover, suppose that, as $N$ tends to infinity, 
\begin{align}
\label{Scaling, full model, no space}
J,K,S,\widehat{S} \to \infty; \quad  \frac{K}{J} \to 0;  \quad  \frac{u^2 NK}{J^2} \to 2a; \quad  \frac{N^2}{KJ^2} \to 0; \nonumber\\
\mathbb{E}_{\pi}\left[ \left\lbrace\frac{suN}{SJ} \left(\frac{\sigma(\kappa_r, \zeta)}{\sigma(\kappa_c,\zeta)} - 1 \right)\right\rbrace^{2}\right] \to b^{2}; \quad \frac{\widehat{S}}{K} \to 0; \quad \frac{\widehat{S}}{S} \to 0. 
\end{align} 
In addition, assume that there exists an $m$ such that, as $N \to \infty$, $NK^m/J^m \to 0$.
Then the sequence $X_{N}(t)$ converges weakly  to the Feller diffusion in a random environment with initial condition $X_{0}$ and parameters $2a$, $b^{2}$.
\end{theorem}

\begin{example}
Fix $\epsilon \in (0,1/4)$, $\beta \in (0,1/4 - \epsilon)$ and $\gamma \in (0,\beta)$.
The conditions of Theorem~\ref{NS fluctuating selection theorem} are satisfied if
\begin{align*}
J = N^{\frac{3}{4} + \epsilon} \quad S = N^{\beta} \quad K = N^{\frac{1}{2} + 2\epsilon} \quad \widehat{S}  = N^{\gamma}.
\end{align*}
\end{example}

As a by-product of our technique, we prove an analogous result for the neutral model and the model with selection. 

\begin{theorem}
\label{NS neutral theorem}
Let $X^{N}(t)$ denote the total intensity of individuals of rare type at time $t$.
Suppose that $X^{N}_{0}$ converges to $X_{0}$ and the intensity of selective events is zero.
Moreover, suppose that, as $N$ tends to infinity,
\begin{align}
\label{Scaling, neutral, no space}
J,K \to \infty; \quad \frac{K}{J} \to 0;  \quad   \frac{N^2}{KJ^2} \to 0; \quad \frac{u^2 NK}{J^2} \to 2a.
\end{align} 
In addition, assume that there exists an $m$ such that, as $N \to \infty$, $NK^m/J^m \to 0$.
Then the sequence $X_{N}(t)$ converges weakly to the critical Feller diffusion ($b = 0$) with initial condition $X_{0}$ and variance parameter $2a$.
\end{theorem}

\begin{theorem}
\label{NS selection theorem}
Let $X^{N}(t)$ denote the total intensity of individuals of rare type at time $t$.
Suppose that $X^{N}_{0}$ converges to $X_{0}$, $\sigma$ does not depend on the environment and $\hat{S}= 1$.
Moreover, suppose that, as $N$ tends to infinity,
\begin{align}
\label{Scaling, model with selection, no space}
J,K,S \to \infty; \quad  \frac{K}{J} \to 0; \quad   \frac{N^2}{KJ^2} \to 0; \quad  \frac{u^2 NK}{J^2} \to 2a; \quad \frac{suN}{SJ} \left(\frac{\sigma(\kappa_r)}{\sigma(\kappa_c)} - 1 \right) \to b
\end{align} 
In addition, assume that there exists an $m$ such that, as $N \to \infty$, $NK^m/J^m \to 0$.
Then the sequence $X_{N}(t)$ converges weakly to the Feller diffusion  with initial condition $X_{0}$ and parameters $2a$ and $b$.
\end{theorem}

\begin{remark}
We stress  that our proofs do not guarantee convergence of the lookdown representations, but only convergence of the  projected models.
\end{remark}

\subsubsection{Projected model}
\label{Subsection projected model}
We follow the work in \cite{etheridge/kurtz:2014} to show the connection between our lookdown construction and the standard LFV.

The neutral generator is simply the non-spatial counterpart of the lookdown construction from \cite{etheridge/kurtz:2014} and so we do not treat it here. However, in order to clarify how our form of selection acts from the perspective of the underlying process we investigate further.
This section follows \cite{etheridge/kurtz:2014} and all the techniques in this section are taken from there, but formulated in our notation.
The part of the generator of our process which describes selective events is of the form
\begin{multline*}
A_{sel}f(\eta)
=
s
\Bigg(
\int_0^\infty
\Bigg[
u e^{- uv^*} g( v^*,\kappa^*)
e^{- u \int_{v^*}^{\infty} (1 - g(v,\kappa^*)
) \mathrm{d} v}
\\
\times
\prod_{ (l, \kappa ) \in \eta , l \neq l^*_{sel} }
g \left( \mathcal{J}_{sel}(l,l^*_{sel},v^*), \kappa \right)
\Bigg] \mathrm{d} v^*
 - f(\eta) \Bigg)
.
\end{multline*}
We define $h(\kappa) = \int_0^\infty (1 - g(l, \kappa) ) \mathrm{d} l$ and note that integration by parts gives
\begin{align*}
\int_0^\infty u e^{-u v^*} g( v^*, \kappa^*) e^{-\frac{uK}{J} \int_{v^*}^{\infty} (1 - g( v , \kappa^* ) ) \mathrm{d}v} \mathrm{d} v^*
= 
e^{-u
h(\kappa^*)}
.
\end{align*}
We also note that if $\eta$ is formed from a Poisson point process with intensity measure $m_{leb} \otimes \Xi(\mathrm{d} \kappa)$ then
\begin{align*}
\{ \mathcal{J}_{sel}(l,l^*_{sel},v^*)
:
l \neq l^*_{sel} \}
,
\end{align*}
is a Poisson point process with intensity measure 
$m_{leb} \otimes(1 - u)\Xi(\mathrm{d}\kappa)$.
Consider the distribution of $\kappa^*$ conditioned on $\Xi(\mathrm{d} \kappa)$. We note that 
\begin{align*}
\IP[ \kappa^* = \kappa_r] = \frac{\sigma(\kappa_r) \Xi(\{\kappa_c \})}
{
\sigma(\kappa_r) \Xi(\{\kappa_c \})
+
\sigma(\kappa_c) \Xi(\{\kappa_r \})
}
,
\end{align*}
and so when we average our selective generator we get
\begin{multline*}
\alpha
A^N_{sel}f(\Xi)
=
s
e^{-\int_{\mathcal{K}} h(\kappa) \Xi(\mathrm{d}\kappa)}
\\
\times
\Bigg(
\Bigg[
\frac{\sigma(\kappa_r) \Xi(\{\kappa_c \})}
{
\sigma(\kappa_r) \Xi(\{\kappa_c \})
+
\sigma(\kappa_c) \Xi(\{\kappa_r \})
}
e^{-u h(\kappa_r)}
\\
+
\frac{\sigma(\kappa_c) \Xi(\{\kappa_r \})}
{
\sigma(\kappa_r) \Xi(\{\kappa_c \})
+
\sigma(\kappa_c) \Xi(\{\kappa_r \})
}
e^{-u h(\kappa_c)}
\Bigg]
e^{u \int_{\mathcal{K}} h(\kappa) \Xi(\mathrm{d}\kappa)}
 - 1 \Bigg)
.
\end{multline*}
We finally note that our selection can now be seen as a simple weighted choice of parent in the (non-spatial) Lambda Fleming-Viot process. The same calculation can be done for the spatial case.

\begin{remark}
\label{remark how our selection compares to the others}
We notice that our way of modelling selection differs from the approach of \cite{etheridge/veber/yu:2014} and \cite{biswas/etheridge/klimek:2018}. 
However, a quick generator calculation shows that this type of selection leads to the same diffusion approximation as in \cite{etheridge/veber/yu:2014} and \cite{biswas/etheridge/klimek:2018}. 
\end{remark}

\subsection{Convergence of the non-spatial model}
\label{Subsection Convergence of the non-spatial model}

We are  interested in convergence of the model with selection in a fluctuating environment. 
Recall that we  study the behaviour of an establishing mutation under this model.
We take that into account by considering test functions which are unaffected by the individuals of the common type, $\kappa_{c}$.
 
We recall the generator of the rescaled process takes the form
\begin{align}
\label{fluctuating generator}
A^N f(\eta,\zeta) = A^N_{neu}f(\eta,\zeta) + \widehat{S}A^N_{sel}f(\eta,\zeta) + \widehat{S}^{2}A^N_{env}f(\eta,\zeta),
\end{align}
where $A^N_{neu}$ is the part of the generator which describes  neutral events, $A^N_{sel}$ is the part of the generator which describes  selective events, and $A^N_{env}$ describes the evolution of the environment.

In Sections~\ref{Subsection Neutral model proof} and \ref{Subsection Selective model proof}, we show that the terms $A^{N}_{neu}$ and $A^{N}_{sel}$ 
converge to well-defined limits.
However, since $\widehat{S} \to \infty$ as $N$ tends to infinity, it may seem that as $N$ tends to infinity, \eqref{fluctuating generator} will not converge to a non-trivial limit. 
However, the naive limiting procedure does not take into account the changes in the direction of selection.
In order to identify the correct limit, we apply a `separation of timescales' trick due to \cite{kurtz:1973}.

By the calculations in the proof of  Theorem~\ref{NS selection theorem}
the operator $A_{sel}^N$ can be written as
%
%
\begin{align}
\label{Asel for f1}
A_{sel}^Nf(\eta)
=
f(\eta)
\frac{suN}{JS}
\left(
\frac{\sigma(\kappa_r,\zeta)}{\sigma(\kappa_c,\zeta)}
- 1
\right)
\sum_{j}l_{j} 
\frac{g^{\prime}(\kappa_j , l_j,\xi )}{g(\kappa_j ,l_j,\xi)}
+ \cO \left(\frac{1}{S}+ \frac{1}{K}
\right)
\end{align}
We consider a test function $\widetilde{f}$ of the form
\begin{align*}
\widetilde{f}(\eta,\zeta) = f(\eta) + 
\frac{1}{\hat{S}}
f(\eta)
\frac{suN}{JS}
\left(
\frac{\sigma(\kappa_r,\zeta)}{\sigma(\kappa_c,\zeta)}
- 1
\right)
\sum_{j}
\frac{g^{\prime}(\kappa_j , l_j )}{g(\kappa_j ,l_j)}l_{j} 
=: f(\eta) 
+ \frac{1}{\hat{S}}f_{1}(\eta,\zeta)
.
\end{align*}
Observe that since $\hat{S} \to \infty$, we can prove the test function $\widetilde{f}$ will tend to $f(\eta)$ as $N \to \infty$.
We apply the generator \eqref{fluctuating generator} to $\widetilde{f}$.  This leads to
\begin{multline*}
A^{N}\widetilde{f}(\eta,\zeta)
=
A_{neu}^{N}f(\eta) + \widehat{S}A_{sel}^{N}f(\eta)
+
\frac{1}{\hat{S}}A_{neu}^{N}
f_{1}(\eta,\zeta)
+ A_{sel}^{N}\left( 
f_N(\eta,\xi)
  \right) 
- \widehat{S}f_N(\eta, \xi) 
\\
 = A_{neu}^{N}f(\eta)+  A_{sel}^{N}\left( 
f_N(\eta, \xi) 
 \right)
 + \mathcal{O}\left(
 \frac{1}{\widehat{S}}
 +
 \frac{\widehat{S}}{S} + \frac{\widehat{S}}{K}
 \right)
\end{multline*}
where we have used \eqref{Asel for f1}, that $f$ does not
depend on $\xi$ (to see $A_{env}f(\eta) = \IE_\pi[f(\eta)]-f(\eta) = 0$) and \eqref{0d symmetry condition},
(to see that, as $\IE_\pi[f_N(\eta,\zeta)]=0$, $A_{env}f_N(\eta,\zeta) = - f_N(\eta, \zeta)$).
Therefore, at least heuristically, the identification of $A_{sel}^{N}\left( 
f_N(\eta, \xi) 
 \right)$
 should lead to the correct limit. 
 
To make this argument rigorous, 
we shall use a theorem due to \cite{kurtz:1992} which we recall here.
Let us introduce some notation.
For a metric space $E$, let 
$l_m(E)$
be the space of measures on $[0,\infty) \times E$ such that 
$\mu \in l_m(E)$ if and only if $\mu([0,t) \times E) = t$.
\begin{theorem}[\cite{kurtz:1992}, Theorem~2.1]
\label{kurtzaveraging}
Let $E_1$, $E_2$ be complete separable metric spaces, 
and set $E=E_1\times E_2$. For each $n$, let $\{(X_n,Y_n)\}$ be 
a stochastic process with sample paths in $D_E([0,\infty) )$ adapted to 
a filtration $\{{\cal F}_t^n\}$. Assume that 
$\{X_n\}$ satisfies the compact containment condition, that is, for 
each $\epsilon>0$ and $T>0$, there exists a compact $K\subset E$ such
that
\begin{equation}
\label{tightness}
\inf_n\IP[X_n(t)\in K, t\leq T]\geq 1-\epsilon,
\end{equation}
and assume that $\{Y_n(t) : t\geq 0, n=1,2,\ldots\}$ is 
relatively compact (as a collection of $E_2$-valued random 
variables). Suppose that there is an operator 
$A:{\cal D}(A)\subset \overline{C}(E_1)\rightarrow C(E_1\times E_2)$ such
that for $f\in {\cal D}(A)$ there is a process $\epsilon_n^f$ for 
which 
\begin{equation}
\label{kurtz rearrangement}
f(X_n(t))-\int_0^t Af(X_n(s),Y_n(s))\mathrm{d}s +\epsilon_n^f(t)
\end{equation}
is an $\{{\cal F}_t^n\}$-martingale. Let ${\cal D}(A)$ be dense
in $\overline{C}(E_1)$ in the topology of uniform 
convergence on compact sets. Suppose that for each $f\in{\cal D}(A)$ and
each $T>0$, there exists $p>1$ such that
\begin{equation}
\label{kurtzbound}
\sup_n\IE\left[\int_0^T|Af(X_n(t),Y_n(t)|^{p}\mathrm{d}t\right]<\infty
\end{equation}
and 
\begin{equation}
\label{epsilontozero}
\lim_{n\rightarrow\infty}\IE\left[\sup_{t\leq T}|\epsilon_n^f(t)|\right]=0.
\end{equation}
Let $\Gamma_n$ be the $l_m(E_2)$-valued random variable given by
\begin{equation}
\nonumber
\Gamma_n\left([0,t]\times B\right)=\int_0^t\ind_B(Y_n(s))\mathrm{d}s.
\end{equation}
Then $\{(X_n,\Gamma_n)\}$ is relatively compact in 
$D_{E_1}[0,\infty)\times l_m(E_2)$, and for any limit point
$(X,\Gamma)$ there exists a filration $\{{\cal G}_t\}$ such that
\begin{equation}
f(X(t))-\int_0^t\int_{E_2}Af(X(s),y)\Gamma(\mathrm{d}s\times \mathrm{d}y)
\end{equation}
is a $\{{\cal G}_t\}$-martingale for each $f\in {\cal D}(A)$.
\end{theorem} 
We observe that 
\begin{multline*}
f(\eta_t) - \int_0^t Af(\eta_t,\zeta_t)\mathrm{d}s +  (\widehat{f}(\eta_t,\zeta_t) - f(\eta_t) ) +  \int_0^t Af(\eta_s,\zeta_s) -  A^{N}\widehat{f}(\eta_s,\zeta_s) \mathrm{d}s
\\
= \widetilde{f}(\eta_t,\zeta_t)- \int_0^t A^{N}\widetilde{f}(\eta_s,\zeta_s)\mathrm{d}s
,
\end{multline*}
where $A$ is given by \eqref{Kurtz-Rodrigues limiting generator}. 
Since $\widetilde{f} - \int_{0}^{t} A^{N}\widetilde{f}(s)\mathrm{d}s$ is a martingale, we have written our problem in the form \eqref{kurtz rearrangement} with
\begin{align}
\label{epsilon definition}
\epsilon_N^f(t) = 
(\widetilde{f}(\eta_t,\zeta_t) - f(\eta_t) ) +  \int_0^t Af(\eta_s,\zeta_s) -  A^{N}\widetilde{f}(\eta_s,\zeta_s) \mathrm{d}s
\end{align}

To check that the assumptions of Theorem~\ref{kurtzaveraging} are satisfied, we
 work with both the lookdown representation and the projected model. 
The projected model allows us to check the compact containment condition \eqref{tightness} and prove the $L^p$ estimate \eqref{kurtzbound}.
Both are achieved via an intensity estimate given by the following Lemma.
\begin{lemma}
\label{Intensity estimate 0d}
Let $X^{N} = Kw^{N}$ denote the total intensity of individuals of the rare type.  Assume that $\mathbb{E}[X^{N}(0)] < \infty$. Then for any $T > 0$ 
\begin{align}
\label{First part of intensity estimate 0d}
\sup_{t \leq T}\sup_{N}\mathbb{E}[X^{N}(t)]  < &  \infty
,
\\
\label{Second part of intensity estimate 0d}
\lim_{H \to \infty}\sup_{N}\mathbb{P}\left[\sup_{t \leq T}X^{N}(t) > H \right]  = &  0
.
\end{align} 
\end{lemma}
We discuss the proof in Section~\ref{subsection intensity estimate 0d}.

The part of the argument which allows us to identify the correct limit and shows that  condition \eqref{epsilontozero} is satisfied,  that is
\begin{equation}
\lim_{n\rightarrow\infty}\IE\left[\sup_{t\leq T}|\epsilon_N^f(t)|\right]=0,
\end{equation}
is more involved, and requires the use of the lookdown representation. 
We will first look at the behaviour of
\begin{align*}
\int_0^t A^Nf(\eta_s,\zeta_s) - A f(\eta_s,\zeta_s) \mathrm{d} s
.
\end{align*}
The terms involving $A^{N}_{neu}$ and $A^{N}_{sel}$ are tackled separately.

Let $\eta^r_t$ be the process obtained from $\eta_t$ by only considering the individuals of the rare type, $\eta^r_t := \{l : (l, \kappa_r) \in \eta_t \}$.
\begin{proposition}
\label{Proposition Neutral Convergence}
Under the conditions of Theorem~\ref{NS neutral theorem}, 
\begin{multline*}
\IE \Bigg[
\sup_{t \leq T}
\Bigg|
\int_0^t
A_{neu}^{N}f(\eta_s)
\\
-
\left(
f(\eta_s^r) 
\sum_{l_i(t) \in \eta_t^N} al_{i}^{2} \frac{g'(l_{i}(t))}{g(l_{i}(t))}
+
2a
f(\eta_s^r) \sum_{l_i(t) \in \eta_s^r}
\int_{l_i(t)}^{\infty} \big( 1 - g(\kappa_{i} , v) \big) \mathrm{d}v
\right)
\mathrm{d} s
\Bigg|
\Bigg]
\to 0
.
\end{multline*}
\end{proposition}
This proposition will be proved via Taylor's formula through Proposition~\ref{probconvergenceofpaths} and Proposition~\ref{probconvergenceofbirths} in Section~\ref{Subsection Neutral model proof}.
An analogous proposition applies to the terms involving $A^{N}_{sel}$.
 \begin{proposition} \label{selectivepropositionA}
Under the conditions of Theorem~\ref{NS selection theorem}
for any $T \in \IR$
\begin{align*}
\IE\Bigg[ \sup_{t \leq T} \Bigg| \int_0^t A^N_{sel}f(\eta^N_s) - 
f(\eta^N_s) \left(
\sum_{l_i(t) \in \eta_t^N} - bl_{i} \frac{g'(l_{i})}{g(l_{i})}
\right)
\mathrm{d} s
\Bigg| \Bigg]
\to 0
.
\end{align*}
\end{proposition}
We discuss the proof of this proposition, along with calculations which allow us to fully justify the separation of timescales procedure in Section~\ref{Subsection Selective model proof}. 
These three propositions will allow us to conclude Theorem~\ref{NS neutral theorem} and Theorem~\ref{NS selection theorem}. We then proceed to use these results
 to prove Theorem~\ref{NS fluctuating selection theorem} in Section~\ref{Subsection fluctuating model proof}.
We notice that if the random perturbation $Y$ appearing in the statement of Theorem~\ref{kurtzaveraging} is trivial, the statement itself reduces to the usual condition for relative compactness of the sequence of stochastic processes, see, for example, \cite{ethier/kurtz:1986}, Theorem~3.9.1 and Theorem~3.9.4.
Therefore as a by-product of our construction we give a proof of Theorem~\ref{NS neutral theorem} and Theorem~\ref{NS selection theorem}.
The proof of Theorem~\ref{NS neutral theorem} is the most technical of this section.

\begin{remark}
\label{NS_almost_full_lookdown_proof}
Our proof guarantees the relative compactness of the sequences of scaled lookdown representations and that limit points must satisfy a martingale problem. 
However, we do not have a proof of uniqueness of the martingale problem characterizing the limiting equation. 
We will use the Markov Mapping Theorem to deduce the relative compactness of the sequence of projected models and a martingale problem characterising limit points.
Lemma~A.13 from \cite{kurtz/rodrigues:2011} guarantees that every projection given by the Markov Map solves the projected martingale problem.
Since this projected martingale problem has unique solutions, the relative compactness is enough to guarantee convergence of the sequence of projected models.
\end{remark}



\begin{remark} 
\label{rarety remark}
We are interested in the behaviour of a rare subpopulation. 
As we have discussed earlier, we would like it to form $\mathcal{O}(1/K)$ of the population. 
For technical reasons, instead of the process $X^{N}(t)$, it is sometimes convenient to consider a stopped process $X^{N}(t \wedge \tau^{N})$, where 
\begin{align*}
\tau^N := \inf \{ t > 0 : X^N_t > Z^N \}.
\end{align*}
We require the sequence of real numbers $Z^{N}$ to be finite for each $N$ and to tend to infinity as $N$ tends to infinity. 
This requirement coupled with Lemma~\ref{Intensity estimate 0d} guarantees that the convergence of the stopped processes translates directly into convergence of the unstopped processes. 
We shall see that technical assumptions will require that $Z^{N} \rightarrow \infty$ sufficiently slowly. However, these assumptions will not change our proof.
\end{remark}

\subsubsection{Intensity estimate}
\label{subsection intensity estimate 0d}
This subsection is devoted to the proof of our intensity estimate. 

\begin{proof}[Proof of Lemma~\ref{Intensity estimate 0d}] 
\hspace{1pt} The generator of the projected process can be written as 
\begin{align*}
\mathcal{L}f(w,\zeta) &= \left\lbrace
wf((1-u)w + u,\zeta) + (1-p)f((1-u)w,\zeta) - f(w,\zeta)
\right\rbrace \\
&+ s\left[
\frac{\sigma(\kappa_{r},\zeta)w}{\sigma(\kappa_{r},\zeta)w + \sigma(\kappa_{s},\zeta)(1-w)}pf((1-u)w + u,\zeta)
\right. \\
&+\left.\frac{\sigma(\kappa_{c},\zeta)(1-w)}{\sigma(\kappa_{r},\zeta)w + \sigma(\kappa_{s},\zeta)(1-w)}wf((1-u)w + u,\zeta) - f(w,\zeta)
\right]\\
&+ \mathcal{L}^{env}f(p,\zeta).
\end{align*}
Recall that we are interested in an estimate for  the intensity of the process describing the evolution of the rare individuals.
We therefore substitute $X = Kw$.
Under the assumptions of Theorem~\ref{NS fluctuating selection theorem} this leads to the generator
\begin{align*}
\mathcal{L}f(X,\zeta) &= N\left\lbrace\frac{X}{K}f\left(\left(1-\frac{u}{J}\right)X + K\frac{u}{J},\zeta\right) + \left(1-\frac{X}{K}\right)f\left(\left(1-\frac{u}{J}\right)X,\zeta\right) - f(X,\zeta)\right\rbrace
\\
&+ N\widehat{S}\frac{s}{S}\left[
\frac{\sigma(\kappa_{r},\zeta)\frac{X}{K}}{\sigma(\kappa_{r},\zeta)\frac{X}{K} + \sigma(\kappa_{s},\zeta)\left(1-\frac{X}{K}\right)}f\left(\left(1-\frac{u}{J}\right)X + K\frac{u}{J},\zeta\right)
\right. \\
&+\left.\frac{\sigma(\kappa_{c},\zeta)\left(1-\frac{X}{K}\right)}{\sigma(\kappa_{r},\zeta)\frac{X}{K} + \sigma(\kappa_{s},\zeta)\left(1-\frac{X}{K}\right)}f\left(\left(1-\frac{u}{J}\right)X,\zeta\right) - f(X,\zeta)
\right]\\
&+ \mathcal{L}^{env}f(X,\zeta).
\end{align*}
This means that for any $f \in C^{\infty}(\mathbb{R})$,
\begin{align*}
f(X^{N}(T),\zeta_{T}) = f(X_{0},\zeta_{0}) + \int_{0}^{T}\mathcal{L}(X^{N}(s),\zeta_{s})\mathrm{d}s + M(T),
\end{align*}
where $M$ is a martingale. 
Substituting $f(x,\zeta) = x$ leads to 
\begin{align*}
X(T) = &
 X(0) + N\widehat{S}\frac{s}{S}\frac{u}{J}\int_{0}^{T}
 \left[
\frac{\sigma(\kappa_{r},\zeta)}{\sigma(\kappa_{r},\zeta)\frac{X}{K} + \sigma(\kappa_{c},\zeta)\left(1-\frac{X}{K}\right)}
\frac{X}{K}\left(
K - X
\right)
\right.
\\
&
\left.
+\frac{ - \sigma(\kappa_{c},\zeta)}{\sigma(\kappa_{r},\zeta)\frac{X}{K} + \sigma(\kappa_{c},\zeta)\left(1-\frac{X}{K}\right)}
\left(1-\frac{X}{K}\right)X
\right]\mathrm{d}s + M(T)
\\
= &
 X(0) + N\widehat{S}\frac{s}{S}\frac{u}{J}\int_{0}^{T}
 \left[
\frac{\sigma(\kappa_{r},\zeta)}{\sigma(\kappa_{r},\zeta)\frac{X}{K} + \sigma(\kappa_{c},\zeta)\left(1-\frac{X}{K}\right)}
X
\right.
\\
& \left.
-
\left(
\frac{\sigma(\kappa_{r},\zeta) \frac{X}{K}}{\sigma(\kappa_{r},\zeta)\frac{X}{K} + \sigma(\kappa_{c},\zeta)\left(1-\frac{X}{K}\right)}
+
\frac{ \sigma(\kappa_{c},\zeta) \left(1-\frac{X}{K}\right)}{\sigma(\kappa_{r},\zeta)\frac{X}{K} + \sigma(\kappa_{c},\zeta)\left(1-\frac{X}{K}\right)}
\right)
X
\right]\mathrm{d}s + M(T).
\end{align*}
Since 
\begin{multline}
\label{0d intensity estimate fraction expansion}
\frac{\sigma(\kappa_{r},\zeta)}{\sigma(\kappa_{r},\zeta)\frac{X}{K} + \sigma(\kappa_{c},\zeta)\left(1-\frac{X}{K}\right)}
 = 
\frac{\sigma(\kappa_{r},\zeta)}{\sigma(\kappa_{c},\zeta)}
+
\frac{X}{K}\frac{ \sigma(\kappa_{c},\zeta)\sigma(\kappa_{r},\zeta) - \sigma^{2}(\kappa_{r},\zeta)}{\sigma(\kappa_{c},\zeta)\left(\sigma(\kappa_{r},\zeta)\frac{X}{K} + \sigma(\kappa_{c},\zeta)\left(1-\frac{X}{K}\right)\right)}
,
\end{multline}
this expression can be written as 
\begin{multline}
\label{IE 0d martingale representation}
X(T) = 
 X(0) + N\widehat{S}\frac{s}{S}\frac{u}{J}
\left\lbrace 
 \int_{0}^{T}
 \left[
\left(\frac{\sigma(\kappa_{r},\zeta)}{\sigma(\kappa_{c},\zeta)} - 1\right)X
\right]
\mathrm{d}s \right.\\
+
\left.
\int_{0}^{T}
 \left[
\frac{X^{2}}{K}\frac{ \sigma(\kappa_{c},\zeta)\sigma(\kappa_{r},\zeta) - \sigma^{2}(\kappa_{r},\zeta)}{\sigma(\kappa_{c},\zeta)\left(\sigma(\kappa_{r},\zeta)\frac{X}{K} + \sigma(\kappa_{c},\zeta)\left(1-\frac{X}{K}\right)\right)}
\right]
\mathrm{d}s
\right\rbrace
+ M(T).
\end{multline}
Consider the stopping time $\tau = \inf \{ t \geq 0: X^{N} > H \}$ and the stopped process $\widehat{X}^{N}(t) = X^{N}(t \wedge \tau)$.
Since \eqref{IE 0d martingale representation} holds at a bounded stopping time, taking the expectation and using the symmetry condition \eqref{0d symmetry condition} leads to  
\begin{multline*}
\mathbb{E}[\widehat{X}^{N}(T)] = \mathbb{E}[\widehat{X}(0)] 
+
 N\widehat{S}\frac{s}{S}\frac{u}{J}
 \mathbb{E}
\left[ 
 \int_{0}^{T}
 \left\lbrace
\left(\frac{\sigma(\kappa_{c},\zeta)}{\sigma(\kappa_{r},\zeta)} - 1\right)\widehat{X}^{N}(s)
\right\rbrace
\mathrm{d}s \right.\\
+
\left.
\int_{0}^{T}
 \left\lbrace
\frac{(\widehat{X}^{N}(s))^{2}}{K}\frac{\sigma^{2}(\kappa_{r},\zeta) + \sigma(\kappa_{c},\zeta)\sigma(\kappa_{r},\zeta)}{\sigma(\kappa_{c},\zeta)\left(\sigma(\kappa_{r},\zeta)\frac{\widehat{X}^{N}(s)}{K} + \sigma(\kappa_{c},\zeta)\left(1-\frac{\widehat{X}^{N}(s)}{K}\right)\right)}
\right\rbrace
\mathrm{d}s
\right]
\\
\leq
\mathbb{E}[\widehat{X}(0)] 
+
N\frac{H}{K}\frac{s\widehat{S}}{S}\frac{u}{J} C_{\sigma}
 \int_{0}^{T}\mathbb{E}[\widehat{X}^{N}(s)] \mathrm{d}s
 ,
\end{multline*}
where $C_{\sigma}$ is a constant depending on $\sigma$.
By Gr\"onwall's inequality 
\begin{align*}
\mathbb{E}[\widehat{X}^{N}(T)] \leq \mathbb{E}[\widehat{X}^{N}(0)]\exp
\left( 
N\frac{H}{K}\frac{s\widehat{S}}{S}\frac{u}{J} C_{\sigma}T
\right).
\end{align*}
If one takes $H = Z^N$ this, combined with the conditions from Theorem~\ref{NS fluctuating selection theorem}, concludes the proof of \eqref{First part of intensity estimate 0d}.
To see that \eqref{Second part of intensity estimate 0d} holds, it is enough to observe that by Markov's inequality 
\begin{align}\label{stoppingtimemarkov}
\IP\left[ \sup_{0 \leq t \leq T} X_t \geq H\right]
=
\IP[ \widehat{X}_{T} \geq H]
\leq
\frac{\IE[ \widehat{X}_{T}]}{H}
,
\end{align}
and for any 
$T$, $\varepsilon > 0$ we can choose $H_1$, $N_1$ such that for 
$H \geq H_1$, $N \geq N_1$ and $t\leq T$
the right hand side of~(\ref{stoppingtimemarkov})
is less than $\varepsilon$.
\end{proof}

We also take note of a simple corollary which we will use later in our main proof.
\begin{corollary}
\label{compactcontainmentcorollary}
The lookdown process satisfies the compact containment condition, \eqref{tightness} in Theorem~\ref{kurtzaveraging}.
\end{corollary}
\begin{proof}
We note that the characterisation of convergence given in Theorem~\ref{KR convergence of random measures theorem} ensures that a sequence of lookdown processes satisfies the compact containment condition if and only if the projected processes also satisfy compact containment condition.
\end{proof}

\begin{remark}
We do not give details but a simpler calculation also proves Lemma~\ref{Intensity estimate 0d} under the assumptions of Theorem~\ref{NS neutral theorem} and Theorem~\ref{NS selection theorem}.
\end{remark}

\subsubsection{Neutral model - proof of Theorem~\ref{NS neutral theorem}}
\label{Subsection Neutral model proof}
Even though this subsection contains the proof of the result for  the least complicated model, the proof itself is the most involved one. 
The relative simplicity of the proof for the more complicated model demonstrates the power of this lookdown method. 
Proofs of Theorem~\ref{NS selection theorem} and Theorem~\ref{NS fluctuating selection theorem} heavily rely on technical observations from this subsection.
We recall that the generator of the neutral part of the process takes the form 
\begin{multline}
\label{noselecgenerator}
A^N_{neu} f(\eta)
=
N
\Bigg(
\int_0^\infty
\Bigg[
\frac{uK}{J} e^{- \frac{uK}{J}v^*} g(\kappa^*, v^*)
e^{- \frac{uK}{J} \int_{v^*}^{\infty} (1 - g(\kappa^*, v)
) \mathrm{d} v} 
\\
\times
\prod_{ (\kappa , l) \in \eta , l \neq l^* }
g \left( \kappa , \mathcal{J}_{neu}(l,l^*,v^*) \right)
\Bigg] \mathrm{d} v^*
 - f(\eta) \Bigg)
,
\end{multline}
where $(\kappa^{*},v^{*})$ denotes the parent and $\cJ_{neu}$ is defined as in \eqref{Jcaldeffinition}, that is 
\begin{align*}
\cJ_{neu}(l ,l_{neu}^{*},v^*)=
\begin{cases}
\frac{1}{1 - u} (l - (l_{neu}^{*} - v^*))
& \text{ if } l > l_{neu}^{*}
,
\\
\frac{1}{1 - u} l
& \text{ if } l <l _{neu}^{*}
,
\\
v^*
& \text{ if } l = l_{neu}^{*}
.
\end{cases}
\end{align*}
Before stating the propositions and lemmas which prove Theorem~\ref{NS neutral theorem}, let us rewrite \eqref{noselecgenerator} in a more convenient form.

We start by observing that a Taylor expansion of $g$ gives
\begin{multline}
\label{taylorforproduct}
g(\kappa^*, v^*)
\prod_{ (\kappa , l) \in \eta , l \neq l^* }
g \left( \kappa , \mathcal{J}(l,l^*,v^*) \right)
-
f(\eta)
\\
=
f(\eta)
\sum_{l} \frac{g^{\prime}(\kappa, l)}{g(\kappa, l)} (\mathcal{J}(l,l^*,v^*) - l)
+ 
\sum_{l, \hat{l}}
\frac{g^{\prime}(\hat{\kappa}, \hat{l})}{g(\hat{\kappa}, \hat{l})}
\frac{g^{\prime}(\kappa, l)}{g(\kappa, l)}
\cO \big( (\mathcal{J}(l,l^*,v^*) - l)(\mathcal{J}(\hat{l},l^*,v^*) - \hat{l}) \big)
\\
+
\sum_l
\frac{g^{\prime\prime}(\kappa, l)}{g^{2}(\kappa, l)}
\cO \big( (\mathcal{J}^2(l,l^*,v^*) - l)
\big)
,
\end{multline}
and a Taylor expansion of the exponential function about $0$ leads to 
\begin{align}
\label{exponentialinttaylor}
e^{- \frac{uK}{J} \int_{v^*}^{\infty} (1 - g(\kappa^*, v)
) \mathrm{d} v}
=
1 - \frac{uK}{J} \int_{v^*}^{\infty} (1 - g(\kappa^*, v)
) \mathrm{d} v + \cO
\left( \frac{uK}{J}
\int_{v^*}^{\infty} (1 - g(\kappa^*, v)
) \mathrm{d} v
 \right)^2
.
\end{align}
Applying \eqref{exponentialinttaylor} we may rewrite \eqref{noselecgenerator} as 
\begin{align*}
A_{neu}^{N} = A_{neu,1}^{N} + A_{neu,2}^{N} + \mathcal{O}\left( \frac{K}{J} A_{neu, 2}^{N} \right),
\end{align*}
where
\begin{align}
\label{eq2.1limitgen}
 A_{neu,1}^{N} = & N
\Bigg(
\int_0^\infty
\frac{uK}{J} e^{- \frac{uK}{J}v^*} 
\bigg[
g(\kappa^*, v^*)
\prod_{ (\kappa , l) \in \eta , l \neq l^* }
g \left( \kappa , \mathcal{J}(l,l^*,v^*) \right)
 - f(\eta)
\bigg] 
\mathrm{d} v^*
\Bigg)
,
\\
A_{neu, 2}^{N} = & N
\Bigg(
- \int_0^\infty
\frac{u^2K^2}{J^2} e^{- \frac{uK}{J}v^*}
\int_{v^*}^{\infty} (1 - g(\kappa^*, v)
) \mathrm{d} v
\;
g(\kappa^*, v^*)
\nonumber
\\
\label{eq2.2limitgen}
&\phantom{
A_{neu, 2}^{N} = N
\Bigg(
- \int_0^\infty
\frac{u^2K^2}{J^2}
e^{- \frac{uK}{J}v^*}
}
\times \prod_{ (\kappa , l) \in \eta , l \neq l^* }
g \left( \kappa , \mathcal{J}(l,l^*,v^*) \right)
\mathrm{d} v^*
\Bigg)
.
\end{align}



We begin with  statements of the propositions that identify limits of \eqref{eq2.1limitgen} and \eqref{eq2.2limitgen} separately. The proofs appear later in this section. 
\begin{proposition}
\label{probconvergenceofpaths}
Under the conditions of Theorem~\ref{NS neutral theorem}, 
\begin{align*}
\IE \Bigg[
\sup_{t \leq T}
\left|
\int_0^t
A_{neu,1}^{N}f(\eta_s)
-
f(\eta_s^r) \left(
\sum_{l_i(t) \in \eta_t^r} al_{i}^{2} \frac{g'(l_{i}(t))}{g(l_{i}(t))}
\right)
\mathrm{d} s
\right|
\Bigg]
\to 0
.
\end{align*}
\end{proposition}
\begin{proposition}
\label{probconvergenceofbirths}
Under the conditions of Theorem~\ref{NS neutral theorem}, 
\begin{align*}
\IE \left[ \sup_{t \leq T} \left|\int_0^t  
\left\lbrace A_{neu, 2}^{N}f(\eta_s)
- 
2a
f(\eta_s^r) \sum_{l_i(t) \in \eta_s^r}
\int_{l_i(t)}^{\infty} \big( 1 - g(\kappa_{i} , v) \big) \mathrm{d}v \right\rbrace \mathrm{d}s \right| \right]
\to 0
.
\end{align*}
\end{proposition}


The proof of Proposition~\ref{probconvergenceofpaths} will use the following elementary lemma:
\begin{lemma}
\label{Lemma enough to bound variance lemma}
Let $\Gamma$ be a stochastic process. Assume that the second moment of $\Gamma$ is bounded uniformly for all times up to time $T$ by $\epsilon$, that is  
$
\mathbb{E}[(\Gamma(s))^{2}] \leq \epsilon
,
$
for $0\leq s \leq T$.
Then 
\begin{align*}
\mathbb{E}\left[\sup_{t<T} \left| 
\int_{0}^{t}\Gamma(s) \mathrm{d}s 
\right| \right]
< \sqrt{T\epsilon}
.
\end{align*}
\end{lemma}
\begin{proof}
By Jensen's inequality 
\begin{align*}
\IE \left[ \sup_{t \leq T}
\left( \int_0^t \Gamma(s) \mathrm{d} s\right)^2 \right]
\leq  \; \IE \left[ \sup_{t \leq T} \; t\int_0^t \Gamma^2(s) \mathrm{d} s \right]
=
\; T \IE \left[
\int_0^T \Gamma^2(s) \mathrm{d}s
\right]
.
\end{align*}
The inequality follows by assumption and a final application of Jensen's inequality.
\end{proof}

\begin{proof}[Proof of Proposition~\ref{probconvergenceofpaths}]
By Lemma~\ref{Lemma enough to bound variance lemma}, it suffices to prove that, conditioned on the postion of rare levels,
\begin{multline*}
\IE \Bigg[
\Bigg\lbrace
N
\Bigg(
\int_0^\infty
\frac{uK}{J} e^{- \frac{uK}{J}v^*} 
\bigg\{
g(\kappa^*, v^*)
\\
\times
\prod_{ (\kappa(t) , l(t)) \in \eta , l(t) \neq l^*(t) }
g \left( \kappa , \mathcal{J}(l(t),l^*(t),v^*) \right)
 - f(\eta_s)
\bigg\}
\mathrm{d} v^*\Bigg)
\\
-
f(\eta_s^r) \left(
\sum_{i} al_{i}^{2} \frac{g'(l_{i}(t))}{g(l_{i}(t))}
\right)
\Bigg\rbrace^{2}
\Bigg]
\to 0
,
\end{multline*}
uniformly for $0\leq s\leq T$.

For ease of notation we let $l_0 = 0$. 
We will use the ordering $l_i$ and \eqref{taylorforproduct} to observe that $A_{neu,1}^{N}f(\eta)$
can be approximated by
\begin{multline}\label{no space first approximation of level part of generator}
N f(\eta)
\Bigg(
\sum_i
\int_{l_{i-1}}^{l_i}
\frac{uK}{J} e^{- \frac{uK}{J}v^*} 
\bigg[
\frac{g^\prime (\kappa_i, l_i)}{g(\kappa_i, l_i)} (v^* - l_i)
+
\sum_{l \neq l_i}
\frac{g^\prime (\kappa, l)}{g(\kappa, l)}
\left(
\frac{l\frac{u}{J}}{1 - \frac{u}{J}}
- \ind_{l > l_i} \frac{l_i - v^*}{1 - \frac{u}{J}}
\right)
\bigg] 
\mathrm{d} v^*
\Bigg)
\\
=
Nf(\eta)
\sum_i
\Bigg[
\frac{g^{\prime}(\kappa_i, l_i)}{g(\kappa_i,l_i)}
\left(
-(l_{i} - l_{i-1}) e^{- \frac{uK}{J} l_{i-1}}
+ \frac{J}{uK}
\left(
e^{- \frac{uK}{J} l_{i-1}} - e^{- \frac{uK}{J} l_{i}}
\right)
\right)
\\
+
\sum_{j \neq i}
\left(
\frac{g^{\prime}(\kappa_j , l_j )}{g(\kappa_j ,l_j)}
\frac{l_j \frac{u}{J}}{1 - \frac{u}{J}}
\left(
e^{- \frac{uK}{J} l_{i-1}} - e^{- \frac{uK}{J} l_{i}}
\right)\right)
\\
+
\sum_{j \neq i}\sum_{j<i} 
\frac{g^{\prime}(\kappa_j , l_j )}{g(\kappa_j ,l_j)}
\frac{1}{1 - \frac{u}{J}}
\left(
-(l_i - l_{i-1}) e^{- \frac{uK}{J} l_{i-1}}
+
\frac{J}{uK}
\left(
e^{- \frac{uK}{J} l_{i-1}} - e^{- \frac{uK}{J} l_{i}}
\right)
\right)
\Bigg]
,
\end{multline}
where the second line follows from  integration.
At this stage we note that 
we can drop the factor $1 - u/J$ at the cost of an error of order $NK/J^{3}$, which tends to zero as $N$ tends to infinity.
We also swap the order of summation to rewrite \eqref{no space first approximation of level part of generator} as
\begin{multline*}
Nf(\eta)
\sum_j
\frac{g^{\prime}(\kappa_j , l_j )}{g(\kappa_j ,l_j)}
\Bigg[
\sum_{i \leq j}
\Bigg(
-(l_i - l_{i-1}) e^{- \frac{uK}{J} l_{i-1}}
+
\frac{J}{uK}
\left(
e^{- \frac{uK}{J} l_{i-1}} - e^{- \frac{uK}{J} l_{i}}
\right)
\Bigg)
\\
+
l_j \frac{u}{J}
\sum_{i \neq j}
\left(
e^{- \frac{uK}{J} l_{i-1}} - e^{- \frac{uK}{J} l_{i}}
\right)
\Bigg]
.
\end{multline*}
We observe that since the second and the third terms in the inner sum are telescoping sums and $l_0 = 0$ we simplify our expression to
\begin{multline}
\label{telescopingsumsstatement}
f(\eta)
\sum_j
\frac{g^{\prime}(\kappa_j , l_j )}{g(\kappa_j ,l_j)} 
\big(P_1 + P_2 + P_3\big)
:=
\\
Nf(\eta)
\sum_j
\frac{g^{\prime}(\kappa_j , l_j )}{g(\kappa_j ,l_j)}
\Bigg[
\frac{J}{uK}
\left(
1 - e^{- \frac{uK}{J} l_{j}}
\right)
+
l_j \frac{u}{J}
\left(
1 - e^{ -\frac{uK}{J} l_{j-1}} + e^{ - \frac{uK}{J} l_j}
\right)
\\
+
\sum_{i \leq j}
\left(
-(l_i - l_{i-1}) e^{- \frac{uK}{J} l_{i-1}}
\right)
\Bigg]
.
\end{multline}
Now we treat each term in the new sum separately. Using the Taylor expansion of the exponential function about $0$ we observe that
\begin{align*}
P_1 &=
\frac{NJ}{uK}
\left(
1 - e^{- \frac{uK}{J} l_{j}}
\right)
=
N\sum_{1 \leq k \leq n+1}
\frac{(-1)^{k-1}}{k!}
l_j^k \left(\frac{uK}{J} \right)^{k-1}
+ \cO\left(\frac{N(uK)^{n+1}}{J^{n+1}} \right)
,
\\
P_2 &=
l_j \frac{uN}{J}
\left(
1 - e^{ -\frac{uK}{J} l_{j-1}} + e^{ - \frac{uK}{J} l_j}
\right)
=\frac{uN l_j}{J} + \cO\left( \frac{NK}{J^2} (l_j - l_{j-1} ) \right)
,
\\
P_3 &=
N\sum_{i \leq j}
\left(
-(l_i - l_{i-1}) e^{- \frac{uK}{J} l_{i-1}}
\right)
=
N\sum_{1 \leq k \leq n + 1}
\frac{(-1)^{k}}{(k-1)!} \left(\frac{uK}{J}  \right)^{k-1}
\sum_{i \leq j}
(l_i - l_{i-1}) l_{i-1}^{k-1}
\\
& \hspace{330pt}
+
\cO\left(\frac{N(uK)^{n+1}}{J^{n+1}} \right)
.
\end{align*}
We focus our attention on $P_{3}$.
We investigate the terms corresponding to different values of $k$ separately. 
We observe that the first three terms involve
\begin{align}
\sum_{i \leq j}
(l_i - l_{i-1})
= & l_j
, \label{approximation k=0}
\\
\sum_{i \leq j}
(l_i - l_{i-1}) l_{i-1}
= &
\frac{1}{2}
\left(
\sum_{i \leq j} (l_i^2 - l_{i-1}^2 )
-\sum_{i \leq j} (l_i  - l_{i-1})^2
\right) 
\nonumber \\
= &
\frac{1}{2}
\left(
l_j^2
-\sum_{i \leq j} (l_i  - l_{i-1})^2
\right)
\label{approximation k=1}
\\
\sum_{i \leq j} \left(
(l_{i} - l_{i-1})l_{i-1}^2
\right)
= &
\frac{1}{3} \sum_{i \leq j} \left( l_i^3 - l_{i-1}^3 \right)
-
\frac{1}{3} \sum_{i \leq j} \left( l_i - l_{i-1} \right)^3
-
\sum_{i \leq j} \left( l_i - l_{i-1} \right)^2 l_{i-1}
\nonumber \\
= &
\frac{1}{3} l_j^3
-
\frac{1}{3} \sum_{i \leq j} \left( l_i - l_{i-1} \right)^3
-
\sum_{i \leq j} \left( l_i - l_{i-1} \right)^2 l_{i-1}
,\label{approximation k=3}
\end{align}
respectively.
We are therefore interested in $\sum_i (l_i - l_{i-1})^2$ conditioned on the locations of the rare levels, that is, conditioned on $\eta^{r}$.


Recall  that conditioned on the locations of the levels of the rare type, the levels of the common type will be Poisson distributed with intensity $K - Z^N$.
Conditioned on the number of levels of the rare type between $0$ and $l_{j}$ these  levels will be independent uniformly distributed on $[0,l_{j}]$.
We denote $n$ independent uniformly distributed random variables on $[0,1]$ by $u_1, \dots, u_n$ and define the order statistics by letting $u_{(i)} = u_k$ if and only if $\#\{ \hat{k} : u_{\hat{k}} \leq u_k\} = i$ for $i \in \{1, \dots n \}$.
We note that $n$ points uniformly distributed on $[0,1]$ can be identified with $n+1$ points uniformly distributed on the unit circle with one of these points chosen at random to be a reference point corresponding to both $0$ and $1$. This then leads us to see that $u_{(i)} - u_{(i-1)}$ is equal in distribution to $u_{(1)}$ for $i \in \{1, \dots, n+1 \}$, where, by convention, $u_{(n+1)}:=1$ and $u_{(0)} := 0$.
From this one can see that
\begin{align*}
\IE \left[ \sum_{i=1}^j (l_i - l_{i-1})^2 | l_0= 0, j = n+1, l_j \right]
=
l_j^2 \IE \left[ \sum_{i=1}^{n+1} (u_{(i)} - u_{(i-1)})^2 \right]
= l_j^2 (n+1)\frac{2 }{(n+1)(n+2)}
.
\end{align*}
We then use that the number of levels of the rare type within $[0,l_j]$ will be Poisson distributed to see
\begin{align*}
\IE \left[ \sum_{i=1}^j (l_i - l_{i-1})^2 | l_0= 0, l_j \right] = &
\sum_{i=0}^\infty    \frac{2 l_j^2 }{n+2} \frac{\left(l_{j} (K - Z^N) \right)^n \exp(-\left(l_{j} (K - Z^N) \right))}{n!}
\\
= &
\frac{2l_{j}}{(K-Z^N)} + \cO\left(\exp \big(-l_{j}(K- Z^N) \big) \right)
.
\end{align*}
Identical calculations show
\begin{align*}
\IE \left[ \left( \sum_{i=1}^j (l_i - l_{i-1})^2 \right)^2 | l_0= 0, l_j  \right] = & \frac{4 l_{j}}{(K-Z^N)^3} \big( 3 + x(K-Z^N) \big)
\\
&+ \cO\left((K - Z^N)\exp \big(-l_{j}(K- Z^N) \big) \right)
,
\numberthis
\label{levelmovementsimpleterm}
\\
\text{Var} \left(  \sum_{i=1}^j (l_i - l_{i-1})^2 | l_0= 0, l_j \right) =  & \cO \left( \frac{l_{j}}{(K-Z^N)^3} \right)
.
\end{align*}
We note that we are considering $j$ to be a random variable throughout this corresponding to the level of a given rare individual.

From this calculation since we multiply \eqref{approximation k=1} by $NK/J$ in we see that we require $\frac{N^2}{KJ^2} \to 0$. Therefore, we may approximate
$\left(\frac{uK}{J}  \right)
\sum_{i \leq j}
(l_i - l_{i-1})^2$
by
$ \frac{ul_j}{J} $.
For $k = 3$ we again condition on the number of levels of the rare type beneath $l_j$ and see that
\begin{align*}
\IE \left[ \sum_{i=1}^j (l_i - l_{i-1})^2 l_{i-1} | l_0= 0, j = n+1, l_j \right]
=
(n+1)
\IE \left[ (l_{I} - l_{I-1})^2 l_{I-1} | l_0= 0, j = n+1, l_j \right]
,
\end{align*}
where $I$ is chosen uniformly at random from $(1, \cdots, n+1)$.
We then again consider the levels as $n+1$ points chosen uniformly at random from a circle to see that this will be
$
\frac{l_j^3}{n+2}
$.
This then gives us
\begin{align}
\label{levelmovementsquareterm}
\IE \left[ \sum_{i \leq j} \left( l_i - l_{i-1} \right)^2 l_{i-1} 
| l_0= 0, l_j 
\right]
=  \frac{l_j^2}{K - Z^N} + \cO\left(\exp(- l_j (K-Z^N))\right)
.
\end{align}
We see that for $k \geq 3$
\begin{align}\label{approximation k geq 3}
\IE
\left[
\sum_{i \leq j} \left(
(l_{i} - l_{i-1})l_{i-1}^k
\right)
-
\frac{1}{k+1} l_j^{k+1}
\right]
= &
\cO\left( \frac{1}{K} \right)
,
\\
\text{Var}
\left(
\sum_{i \leq j} \left(
(l_{i} - l_{i-1})l_{i-1}^k
\right)
-
\frac{1}{k+1} l_j^{k+1}
\right) 
= &
\cO \left(
\frac{1}{K^3} \right)
,
\end{align}
which will suffice as each of these terms will be multiplied by $\frac{NK^{k-1}}{J^{k-1}}$ in \eqref{telescopingsumsstatement}.
We  then note that $\frac{NK}{J^2}$ is bounded and $\frac{K}{J} \to 0$.
We combine \eqref{approximation k=0}, \eqref{approximation k=1}, \eqref{approximation k=3}, \eqref{approximation k geq 3} to see that, conditioned on $l_j$,
\begin{multline*}
\IE \left[
\sum_{i \leq j}
\left(
-(l_i - l_{i-1}) e^{- \frac{uK}{J} l_{i-1}}
\right)
\right]
\\
=
\sum_{0 \leq k \leq n} \frac{(-1)^{k+1}}{(k+1)!} \left(\frac{uK}{J}\right)^k l_j^{k+1}
- \frac{u}{J} l_j
+ \frac{1}{2} \frac{u^2K}{J^2} l_j^2
+ \cO\left( \frac{K^2}{J^3} \right)
+
\cO \left(\frac{K}{J}  \right)^{n+1}
.
\end{multline*}

Finally, we observe that all these approximations and cancellations allow us to approximate  \eqref{telescopingsumsstatement} by
\begin{align*}
f(\eta) \sum_j \frac{g^\prime(\kappa_j, l_j)}{g(\kappa_j,l_j)}
\Bigg(
\frac{1}{2} \frac{u^2 N K}{J^2} l_j^2
+ \cO\left( \frac{N K^2}{J^3} \right)
+
\cO \left(\frac{N K^{n+1}}{J^{n+1}}  \right)
+ \cO\left( \frac{N K}{J^2} (l_j - l_{j-1} ) \right)
\Bigg)
,
\end{align*}
plus a random, mean zero, correction term with variance $\cO(N^2K^2/(J^2(K-Z)^3))$.
\end{proof}

The proof of Proposition~\ref{probconvergenceofbirths} is more involved than that of Proposition~\ref{probconvergenceofpaths}.
Once again we begin with an elementary lemma.
\begin{lemma}
\label{Lemma partition and supremum}
Let $\Gamma(s)$ be a stochastic process. 
For any partition $0 = t_0 < t_1< \dots < t_m = T$,
\begin{align*}
\IE \left[ \sup_{t\leq T} \left|\int_0^t  \Gamma(s)  \mathrm{d} s \right|
\right]
 \leq
 \sum_{j=1}^m \sqrt{ \IE \left[\left(\int_{t_{j-1}}^{t_j} \Gamma(s)   \mathrm{d} s \right)^2 \right]}
+
\IE \left[
\sup_j \int_{t_j-1}^{t_j} | \Gamma(s) | \mathrm{d} s
\right]
.
\end{align*}
\end{lemma}
\begin{proof}
We observe that 
\begin{align*}
\IE \left[ \sup_{t\leq T} \left|\int_0^t  \Gamma(s)  \mathrm{d} s \right|
\right]
& \leq
\IE \left[ \sum_{j=1}^m |\int_{t_{j-1}}^{t_j} \Gamma(s)   \mathrm{d} s | 
+
\sup_j \int_{t_j-1}^{t_j} | \Gamma(s) | \mathrm{d} s
\right]
\\
& \leq
 \sum_{j=1}^m \IE \left[\left|\int_{t_{j-1}}^{t_j} \Gamma(s)   \mathrm{d} s \right| \right]
+
 \IE \left[
\sup_j \int_{t_j-1}^{t_j} | \Gamma(s) | \mathrm{d} s
\right]
\\
& \leq
 \sum_{j=1}^m \sqrt{ \IE \left[\left(\int_{t_{j-1}}^{t_j} \Gamma(s)   \mathrm{d} s \right)^2 \right]}
+
\IE \left[
\sup_j \int_{t_j-1}^{t_j} | \Gamma(s) | \mathrm{d} s
\right]
.
\end{align*}
which concludes the proof.
\end{proof}
%
%

We require a few more computations to transform $A_{neu,2}^{N}$ into a more convenient form.
First of all, we observe that as $(1 - g)$ is bounded, the following approximation is valid:
\begin{multline}
\label{Calculation for the birth term part zero}
\int_{l_{i-1}}^{l_{i}} 
\int_{v^*}^{\infty}  \big( 1 - g(\kappa_{i} , v)  \big) \mathrm{d}v \; \mathrm{d} v^*
= 
\int_{l_{i-1}}^{\infty}
\int_{l+{i-1}}^{l_i \wedge v}  \big( 1 - g(\kappa_{i} , v) \big)  \mathrm{d}v^* \; \mathrm{d} v
\\
= 
\int_{l_{i}}^{\infty}
\int_{l+{i-1}}^{l_i } \big(  1 - g(\kappa_{i} , v) \big) \mathrm{d}v^* \; \mathrm{d} v
+
\int_{l_{i-1}}^{l_i}
\int_{l+{i-1}}^{ v} \big( 1 - g(\kappa_{i} , v) \big) \mathrm{d}v^* \; \mathrm{d} v
\\
= 
(l_i - l_{i-1})
\int_{l_{i}}^{\infty}
 \big( 1 - g(\kappa_{i} , v) \big) \mathrm{d} v
+
\cO\left(  ( l_i - l_{i-1})^2
\right).
\end{multline}
For convenience, we order the individuals present in the system according to their level, that is  we consider  $\eta = \{ (\kappa_i,l_i)\}_{i \geq 1}$ where $l_{i} < l_{i+1} $.
   We observe that  the ordering leads  to the following simplification:
\begin{align*}
A_{neu,2}^{N}f(\eta)
=
- \frac{u^2NK}{J}
\sum_i 
\left(
\int_{l_{i-1}}^{l_{i}} \frac{uK}{J} e^{-\frac{uK}{J} v^*} f(\eta)
\int_{v^*}^{\infty} \big( 1 - g(\kappa_{i} , v) \big) \mathrm{d}v \; \mathrm{d} v^*
\right)
.
\end{align*}

Since $\frac{K}{J} \to 0$, we may use \eqref{Calculation for the birth term part zero} to further simplify  $A_{neu,2}^{N}$ to 
%
\begin{align*}
A_{neu,2}^{N}f(\eta)
=
- \frac{u^2NK^2}{J^2}
f(\eta)
\sum_i
(l_i - l_{i-1})
\left(
\int_{l_i}^{\infty} \big( 1 - g(\kappa_{i} , v) \big) \mathrm{d}v \; \mathrm{d} v^*
\right)
\left(1 +\cO\left(\frac{K}{J} + \frac{1}{K^{2}}\right) \right)
.
\end{align*}

By this calculation and Lemma~\ref{Lemma partition and supremum}, it is clear that in order to prove Proposition~\ref{probconvergenceofbirths}, it is enough to show that for any partition $0 = t_0 < t_1< \dots < t_m = T$,
\begin{align}
\label{splittingthing}
\sum_{j=1}^m \sqrt{ \IE \left[\left(\int_{t_{j-1}}^{t_j} \Gamma^{N}(s)   \mathrm{d} s \right)^2 \right]}
+
\IE \left[
\sup_j \int_{t_j-1}^{t_j} | \Gamma^{N}(s) | \mathrm{d} s
\right] \to 0,
\end{align}
where
\begin{align}
\Gamma^{N}(s) = 
\sum_i
\left(
\frac{u^2NK^2}{J^2}
(l_i(t) - l_{i-1}(t))
-
2 a 
\right)
f(\eta_t)
\int_{l_i(t)}^{\infty} \big( 1 - g(\kappa_{i} , v) \big) \mathrm{d}v \; \mathrm{d} v^*
\mathrm{d} t.
\end{align}
We first turn our attention to the parts without a supremum which we calculate directly.

\begin{lemma} \label{lemmaforbirthprocess}
Conditioned on the locations of the rare levels,
\begin{multline}
\label{2ndmomentforbirthprocess}
\IE \left[
\left(
\int_0^T
\sum_i
\left(
\frac{u^2NK^2}{J^2}
(l_i(t) - l_{i-1}(t))
-
2 a 
\right)
f(\eta_t)
\int_{l_i(t)}^{\infty} \big( 1 - g(\kappa_{i} , v) \big) \mathrm{d}v \; \mathrm{d} v^*
\mathrm{d} t
\right)^2
\right]
\\
=
\cO \left( \frac{LT}{N}\right)
+ o( T^2 )
,
\end{multline}
for any $L(N)$ such that $\frac{L}{J} \to \infty$.
\end{lemma}
\begin{proof}
We recall that there exists $\lambda_g$ such that for all $v \geq \lambda_g$, $g(v) = 1$.
Therefore
\begin{align*}
f(\eta_t)
\int_{l_i(t)}^{\infty} \big( 1 - g(\kappa_{i} , v) \big) \mathrm{d}v
\end{align*}
 is bounded for each
 test function.
As $\eta_t$ is constant between events we condition on exactly $n$ reproductive events occurring within $[0,T]$.
We denote the time of events by $t_i$ for $i \in \{1, \dots , n \}$. Observe that $t_i$ are independent random variables uniformly distributed on $[0,T]$.
For ease of notation we let $t_0 = 0$ and $t_{n+1} = T$.
We will use $\hat{l}_k$ to denote a rare type individual's level and $\tilde{l}_k$ to be the highest level across the population (that is, from individuals of both rare and common type) below $\hat{l}_k$ and consider
\begin{multline}
\label{conditionedsquareexpectation1}
\IE \left[
\sum_{i=0}^n
\sum_{j = 0}^n
(t_{i+1} - t_{i})(t_{j+1} - t_{j})
\right.
\\
\left.
\left(
\frac{u^2NK^2}{J^2}
(\hat{l}_k(t_i) - \tilde{l}_{k}(t_i))
-
2 a 
\right)
\left(
\frac{uNK^2}{J^2}
(\hat{l}_k(t_j) - \tilde{l}_{k}(t_j))
-
2 a 
\right)
\right]
.
\end{multline}
We note that the times of events are independent of the level process and so we use
\begin{align*}
\IE[(t_{i+1} - t_{i})(t_{j+1} - t_{j})]
=
\begin{cases}
\frac{T^2}{(n+1)(n+2)} & \text{ if } i \neq j,
\\
\frac{2T^2}{(n+1)(n+2)} & \text{ if } i = j.
\end{cases}
\end{align*}
We also use the Cauchy--Schwarz inequality to see
\begin{align*}
\IE[(\hat{l}_k(t_i) - \tilde{l}_{k}(t_i))
(\hat{l}_k(t_j) - \tilde{l}_{k}(t_j))
] \leq \frac{2}{(K-Z^{N})^2}
,
\end{align*}
where we note that we are still looking at the stopped process.
 
We denote the intensity, after $n$ events, of the levels of all individuals which were offspring in one of those $n$ events by $I_n$. It is then easy to see $I_{n+1} = I_{n}(1-\frac{u}{J}) + \frac{uK}{J}$ and $I_0 = 0$. This recurrence equation  has a solution given by
\begin{align*}
I_n = K \left(  1 - \left(1 - \frac{u}{J} \right)^n  \right)
.
\end{align*}
We now introduce some notation in order to simplify the calculations presented in this section.
For $i < j$ we define $\tilde{l}^{new}_k(t_j)$ as the highest level below $\hat{l}_k$  of an individual (of either rare or common type) which has been born since $t_i$. Then we can see that
\begin{align*}
(\hat{l}_k(t_j) - \tilde{l}_{k}(t_j)) \leq (\hat{l}_k(t_j) - \tilde{l}^{new}_{k}(t_j))
,
\end{align*}
and so
\begin{align*}
\IE[(\hat{l}_k(t_i) - \tilde{l}_{k}(t_i))
(\hat{l}_k(t_j) - \tilde{l}_{k}(t_j))] \leq
\frac{1}{K - Z^{N}} \frac{1}{I_{j-i}}
.
\end{align*}
We then also see that
\begin{align*}
\IE[(\hat{l}_k(t_i) - \tilde{l}_{k}(t_i))
(\hat{l}_k(t_j) - \tilde{l}_{k}(t_j))]
\geq &
\IE[(\hat{l}_k(t_i) - \tilde{l}_{k}(t_i))
(\hat{l}_k(t_j) - \tilde{l}_{k}(t_j)) \ind_{ \tilde{l}_{k}(t_j) =   \tilde{l}^{new}_{k}(t_j) }  ]
\\
= &
\IE[(\hat{l}_k(t_i) - \tilde{l}_{k}(t_i))
(\hat{l}_k(t_j) - \tilde{l}^{new}_{k}(t_j))]
\\
&
-
\IE[(\hat{l}_k(t_i) - \tilde{l}_{k}(t_i))
(\hat{l}_k(t_j) - \tilde{l}^{new}_{k}(t_j)) \ind_{ \tilde{l}_{k}(t_j)  \neq   \tilde{l}^{new}_{k}(t_j) }  ]
\\
\geq
\frac{1}{I_{j-i}}&\frac{1}{K - Z^{N}}- \sqrt{\frac{2}{(K-Z^{N})^2}\frac{2}{I_{j-i}^2} \IP[ \tilde{l}_{k}(t_j)  \neq   \tilde{l}^{new}_{k}(t_j) ] }
,
\end{align*}
where we have used the Cauchy-Schwartz inequality for the final line.

We note that 
\begin{align*}
\IP[ \tilde{l}_{k}(t_j)  \neq   \tilde{l}^{new}_{k}(t_j) ] \leq
\frac{K  \left( 1 - \frac{u}{J} \right)^{j-i}}{K  \left( 1 - \frac{u}{J} \right)^{j-i} + I_{j-i}} 
.
\end{align*}
We now consider a function $L(N)$. By splitting \eqref{conditionedsquareexpectation1} into parts with $|i - j| \leq L$ and $|i - j| > L$, we bound it above by
\begin{multline*}
\frac{4 (n+1) L T^2}{(n+1)(n+2)} \left( \frac{N^2 K^4}{J^4}  \frac{2}{(K - Z^{N})^2} + 4a^2 \right)
\\
+
\frac{C (n+1) (n+1 - L) T^2}{(n+1)(n+2)}
 \Bigg\lbrace \frac{N^2K^4}{J^4} \frac{1}{K-Z^{N}} \frac{1}{I_L} - \left(\frac{N K^2}{J^2}\frac{1}{K-Z^{N}} + \frac{N K^2}{J^2} \frac{1}{I_L} \right) 2a + 4a^2 
\\
+ \cO\left(\frac{N^2 K^4}{J^4} \frac{1}{K-Z^{N}}\frac{1}{I_L} \sqrt{\frac{K  \left( 1 - \frac{u}{J} \right)^{L}}{K  \left( 1 - \frac{u}{J} \right)^{L} + I_{L}}    } \right)
\Bigg\rbrace
.
\end{multline*}
We observe that
\begin{align*}
\sum_{n=1}^\infty \frac{n+1}{(n+1)(n+2)} \frac{(NT)^{n} e^{-NT}}{n!}   = & \frac{1}{NT} - \frac{1}{(NT)^2} + \frac{ e^{-NT}}{(NT)^2} - \frac{e^{-NT}}{2}
,
\\
\sum_{n=1}^\infty \frac{(n+1)^2}{(n+1)(n+2)} \frac{(NT)^{n} e^{-NT}}{n!}   \leq & 1
.
\end{align*}
We may combine the two results above to see
\begin{multline*}
\IE \left[
\left(
\int_0^T
\sum_i
\left(
\frac{u^2NK^2}{J^2}
(l_i(t) - l_{i-1}(t))
-
2 a 
\right)
f(\eta_t)
\int_{l_i(t)}^{\infty} \big( 1 - g(\kappa_{i} , v) \big) \mathrm{d}v \; \mathrm{d} v^*
\mathrm{d} t
\right)^2
\right]
\\
=
\cO
\Bigg(
\frac{LT}{N} \left( \frac{N^2 K^4}{J^4}  \frac{2}{(K - Z)^2} + 4a^2 \right)
\\
+ T^2
\Bigg\lbrace \frac{N^2K^4}{J^4} \frac{1}{K-Z} \frac{1}{I_L} - \left(\frac{N K^2}{J^2}\frac{1}{K-Z} + \frac{N K^2}{J^2} \frac{1}{I_L} \right) 2a
\\+ 4a^2 
+ \cO\left(\frac{N^2 K^4}{J^4} \frac{1}{K-Z}\frac{1}{I_L} \sqrt{\frac{K  \left( 1 - \frac{u}{J} \right)^{L}}{K  \left( 1 - \frac{u}{J} \right)^{L} + I_{L}}    } \right)
\Bigg\rbrace
\Bigg)
.
\end{multline*}
Therefore, in order to conclude that \eqref{2ndmomentforbirthprocess} converges to $0$, it is enough to find $L$ which satisfies 
\begin{align*}
\frac{L}{N}  \to  0
; \quad
\frac{N K^2}{J^2} \frac{1}{I_L} \to  2a
; \quad
\frac{K \left( 1 - \frac{u}{J} \right)^{L}}{K  \left( 1 - \frac{u}{J} \right)^{L} + I_{L}} \to  0
.
\end{align*}
Since $\frac{N}{J} \to \infty$, this is achieved by any $L(N)$ such that $\frac{L}{J} \to \infty$ and $\frac{L}{N} \to 0$. For example, we may choose $L(N) = \sqrt{N J}$.

\end{proof}

We now turn our attention back to \eqref{splittingthing}. We consider $t_j = j \delta(N)$ and see that to ensure the first part of \eqref{splittingthing} converges to zero, we need choices of $L$ and $\delta$ such that
\begin{align}
\label{sumbirthconditions}
\frac{L}{N} \to 0; 
\quad
\frac{L}{J} \to \infty;
\quad
\frac{L}{N {\delta}} \to 0
.
\end{align}

To show that the second term of  \eqref{splittingthing} converges to zero,  it is enough to show boundedness of 
$\IE[\sup_{t \leq T} \Gamma^N(t)]$, provided that there exists a $\delta$ satisfying \eqref{sumbirthconditions} such that  $\delta \to 0$.
The latter is satisfied by taking $\delta = \sqrt{L/N}$.


To show $\IE[\sup_{t \leq T} \Gamma^N(t)]$ is bounded we consider an auxiliary process $\beta^{N}$, defined by
\begin{align}
\label{0d auxiliary submartingale}
\beta^N_t = \frac{NK^2}{J^2}\sum_{l(t) \in \eta_{t}^{r}} \left(l(t) - \hat{l}(t)\right) h_{\lambda}(l(t))
,
\end{align}
where $h$ is a decreasing, positive cut-off function (that is, we assume that there exists a $\lambda$ such that $h(v) = 0$ for $v > \lambda$) and $\hat{l}(t) \in \eta_t$ is the first level below $l(t)$.
To show the required bound on the expectation of the supremum of the integral of $X$, (and therefore conclude the proof), we show that $\beta^{N}$ is dominated by a bounded submartingale. 
We note that under the conditions of Theorem~\ref{NS fluctuating selection theorem} $\frac{NK^2}{J^2} = \cO(K)$.

\begin{lemma}
\label{neutral intensity lemma}
Define $\beta^{N}$ as in \eqref{0d auxiliary submartingale}. Then $\beta^{N}$ is dominated by a bounded submartingale.
\end{lemma}
\begin{proof}
We recall that the offspring in an event can be ordered $(v_1, v_2, \dots )$ with $v_1 = v^*$. Therefore, as the parent is thought of as moving to $v^*$ we will refer to $(v_i)_{i \geq 2}$ as the `children' after an event.
Recall that within the period of time $[0,t]$ with probability $\cO(t^2)$ we will see two or more events. Therefore
\begin{multline*}
\frac{\IE\left[
\sum_{l(t) \in \eta_{t}^{r}} \left(l(t) - \hat{l}(t)\right) h(l(t))
-
\sum_{l(0) \in \eta_{0}^{r}} \left(l(0) - \hat{l}(0)\right) h(l(0))
\right]}
{t}
\\
=
N
\int_0^\infty \frac{uK}{J} e^{-\frac{uK}{J}v^*}
\Bigg\{
\sum_{l(0) \in \eta_{t}^{r}}
\Bigg\{ \left[\left(\cJ(l(0)) - \cJ(\hat{l}(0))\right) - \left(l(0) - \hat{l}(0)\right)\right] h(l(0))
\\
+
\left(\cJ(\hat{l}(0)) - \hat{l}(t)\right) h(l(0))
\\
+
\left(l(t) - \hat{l}(0) \right)\left(\cJ(l(0)) - l(0) \right) \cO( || {h}^\prime ||_{\infty} )
\Bigg\}
+
\sum_{i = 2}^\infty
\left(v_i - \hat{v}_i\right) h(v_i)\ind_{\{v_i \text{ is rare}\} }
\Bigg\} \mathrm{d} v^*
+
\cO(t)
,
\end{multline*}
where we are using $\hat{v}_i$ to denote the highest level below the $i$th `child', $v_i$.
We now require that $h(l) = 0$ for all $l \geq \lambda_h$.
We note that in the proof of Proposition~\ref{probconvergenceofpaths} we have shown that the parts involving $\cJ(l(0)) - l(0) $ converge in $L^2$ to $a l(0)^2$ and so we can see that this is approximated by
\begin{multline*}
\sum_{l(0) \in \eta_{t}^{r}} 
\left\{
\left( a l^2(0) - a \hat{l}^2(0) \right) h(l(0))
+
\left( l(0) - \hat{l}(0) \right) a l^2(0) \cO(||h^\prime||_\infty )
\right\}
\\
+
N
\int_0^\infty \frac{uK}{J} e^{-\frac{uK}{J}v^*}
\Bigg\{
\sum_{l(0) \in \eta_{t}^{r}}
\left(\cJ(\hat{l}(0)) - \hat{l}(t)\right) h(l(0))
+
\sum_{i = 2}^\infty
\left(v_i - \hat{v}_i\right) h(v_i)\ind_{\{v_i \text{ is rare}\} }
\Bigg\} \mathrm{d} v^*
.
\end{multline*}
We note that since the ordering of the $l(t) \in \eta_t$ which are not the parent is unaffected by an event, the only possibilities that will force $\cJ(\hat{l}(0)) < \hat{l}(t)$ is when a `child' has been born between the new locations of levels or when $\hat{l}(0)$ is chosen as the parent. However we use the last sum appearing in the previous equation and the fact that $h$ is decreasing to see that the generator applied to $\beta^N_t$, for large enough $N$, is bounded below by
\begin{align*}
-C_h \beta^N_t + 
NS \frac{NK^2}{J^2}
\sum_{l_i \in \eta_{t}^{r}}
\int_{l_{i-2}}^{\frac{l_{i-2}}{1 - \frac{u}{J}}} \frac{uK}{J} e^{-\frac{uK}{J} v^*} \left(
v^* - \frac{l_{i-2}}{1- \frac{u}{J}} \right) \mathrm{d} v^*
\\
\geq
-C_h \beta^N_t - 
\sum_{l_i \in\eta_{t}^{r}}
\cO\left(
\frac{NK^2}{J^3} l_{i-2}^2
\right)
,
\end{align*}
where $C_h = \lambda_h ||h^\prime||_\infty$.
We therefore see that
\begin{align*}
\beta^N_t + \int_0^t C_h \beta^N_s + C \frac{NK^2Z}{J^3} \mathrm{d} s
,
\end{align*}
is a sub-martingale. Combined with the conditions of Theorem~\ref{NS fluctuating selection theorem} this concludes the proof of the lemma.
\end{proof} 

To complete the proof of Proposition~\ref{probconvergenceofbirths} we see that, by using Lemma~\ref{lemmaforbirthprocess},
\begin{align*}
\IE\left[ \left( \beta^N_T + \int_0^T C_h \beta^N_s + C \frac{NK^2Z}{J^3} \mathrm{d} s\right)^2 \right]
,
\end{align*}
is bounded. 
Therefore, by Jensen's inequality,
\begin{align*}
\IE[\sup_{t \leq T} \beta^N_t ] \leq 
\sqrt{\IE\left[\left(\sup_{t \leq T} \beta^N_t \right)^2\right]}
\end{align*}
which is also bounded by Doob's martingale inequality.  Proposition~\ref{probconvergenceofbirths} therefore follows by an application of Lemma~\ref{Lemma partition and supremum}.


\subsubsection{Selective model - proof of Theorem~\ref{NS selection theorem}}
\label{Subsection Selective model proof}

The generator of the model presented in Theorem~\ref{NS selection theorem} is of the form
\begin{align*}
A^{N}f(\eta) = A_{neu}^{N}f(\eta) + A_{sel}^{N}f(\eta).
\end{align*}
The analysis for the neutral part of the generator, $A^N_{neu}$, is the same as in Section~\ref{Subsection Neutral model proof}, with the exception of a slight modification of Lemma~\ref{neutral intensity lemma}, which we discuss in Lemma~\ref{selective intensity lemma}.
We therefore turn our attention to $A^N_{sel}$, which is the part of the generator describing the selection.
It can be written as 
\begin{multline}\label{selgenerator}
A^N_{sel}f(\eta)
=
\frac{sN}{S}
\Bigg(
\int_0^\infty
\Bigg[
\frac{uK}{J} e^{- \frac{uK}{J}v^*} g(\kappa^*, v^*)
e^{- \frac{uK}{J} \int_{v^*}^{\infty} (1 - g(\kappa^*, v)
) \mathrm{d} v}
\\
\times
\prod_{ (\kappa , l) \in \eta , l \neq l^*_{sel} }
g \left( \kappa , \mathcal{J}_{sel}(l,l^*_{sel},v^*) \right)
\Bigg] \mathrm{d} v^*
 - f(\eta) \Bigg)
,
\end{multline}
where we recall that the movement of the levels $\cJ_{sel}$ is specified by 
\begin{align*}
\cJ_{sel}((l,\kappa),(l^*_{sel},\kappa^*),v^*)
=
\begin{cases}
v^* & \text{ if } l = l^*_{sel}
,
\\
\frac{1}{1 - \frac{u}{J}}
\left(
l - (l^*_{sel} - v^*)\frac{\sigma(\kappa)}{\sigma(\kappa^*)}
\right)
& \text{ if } l \neq l^*_{sel}, l > v^*
,
\\
\frac{1}{1 - \frac{u}{J}}
l
& \text{ if } l < v^*
.
\end{cases}
\end{align*}

As in Section~\ref{Subsection Neutral model proof}, we transform $A^{N}_{sel}$ into a more convenient form. 
We use the Taylor approximation \eqref{exponentialinttaylor} to approximate the generator by
\begin{align*}
A_{sel}^{N} = A_{sel,1}^{N} + A_{sel,2}^{N} + \mathcal{O}\left( \frac{K}{J} A_{sel, 2}^{N} \right),
\end{align*}
where
\begin{align}
\label{selective generator, levels, no space}
 &A_{sel,1}^{N}f(\eta) \nonumber \\
 &=  \frac{sN}{S}
\Bigg(
\int_0^\infty
\frac{uK}{J} e^{- \frac{uK}{J}v^*} 
\bigg[
g(\kappa^*, v^*)
\prod_{ (\kappa , l) \in \eta , l \neq l^*_{sel} }
g \left( \kappa , \mathcal{J}_{sel}(l,l^*_{sel},v^*) \right)
 - f(\eta)
\bigg] 
\mathrm{d} v^*
\Bigg)
,
\\
&A_{sel,2}^{N} f(\eta) =  \frac{N}{S}
\Bigg(
- \int_0^\infty
\frac{u^2K^2}{J^2} e^{- \frac{uK}{J}v^*}
\int_{v^*}^{\infty} (1 - g(\kappa^*, v)
) \mathrm{d} v
\;
g(\kappa^*, v^*)
\nonumber
\\
\label{selective generator, births, no space}
&\phantom{\frac{N}{S}
\Bigg(
\frac{u^2K^2}{J^2} e^{- \frac{uK}{J}v^*}
\int_{v^*}^{\infty} (1 - g(\kappa^*, v)
) \mathrm{d} v}
\times
\prod_{ (\kappa , l) \in \eta , l \neq l^*_{sel} }
g \left( \kappa , \mathcal{J}_{sel}(l,l^*_{sel},v^*) \right)
\mathrm{d} v^*
\Bigg)
.
\end{align}
Once again, we treat $A_{sel,1}^{N}$ and $A_{sel,2}^{N}$ separately. We order the levels  $\eta = \{ (\kappa_i,l_i)\}_{i \geq 1}$ present in the system by requiring that for each $i$,  $l_{i} < l_{i+1} $.
Arguing in the same way as for the neutral case, we may conclude that the term involving $A_{sel,2}^{N}$ tends to $0$ as $N \to \infty$, as the requirements $NK/J^{2} \to C$ combined with $1/S \to 0$ imply that $NK/J^{2}S \to 0$.
We turn our attention to $A_{sel,1}^{N}$.
Recalling the difference between $l^*_{neu}$ and $l^*_{sel}$,
we note that for large $N$ we will very rarely see $l^*_{sel} \neq l^*_{neu}$ and so we consider
\begin{multline}
\label{selective generator, births, no space, neutral parent}
A_{sel,3}^{N}f(\eta)
\\
= 
\frac{sN}{S}
\Bigg(
\int_0^\infty
\frac{uK}{J} e^{- \frac{uK}{J}v^*} 
\bigg[
g(\kappa^*, v^*)
\prod_{ (\kappa , l) \in \eta , l \neq l^*_{neu} }
g \left( \kappa , \mathcal{J}_{sel}(l,l^*_{neu},v^*) \right)
 - f(\eta)
\bigg] 
\mathrm{d} v^*
\Bigg)
.
\end{multline}
We will prove Proposition~\ref{selectivepropositionA} by showing that  $A^N_{sel, 3}$ satisfies a suitable estimate  and then showing that $A^N_{sel, 1} - A^N_{sel, 3}$ converges to $0$ by virtue of Lemma~\ref{unfavouredselectivegenerator}.
\begin{lemma}
\label{Lemma selective estimate neutral parent}
Under the conditions of Theorem~\ref{NS selection theorem}, 
\begin{multline*}
\IE\Bigg[
\Bigg|
\frac{sN}{S}
\Bigg(
\int_0^\infty
\frac{uK}{J} e^{- \frac{uK}{J}v^*} 
\bigg[
g(\kappa^*, v^*)
\prod_{ (\kappa , l) \in \eta_t , l \neq l^*_{neu} }
g \left( \kappa , \mathcal{J}_{sel}(l,l^*_{neu},v^*) \right)
 - f(\eta)
\bigg] 
\mathrm{d} v^*
\Bigg)
\\
-
b
\left(
f(\eta^r_t)
\sum_{j}
\frac{g^{\prime}(\kappa_i , l_i )}{g(\kappa_i ,l_i)}
\frac{suN}{JS}l_{i}
\right)
\Bigg|
\Bigg]
\to 0,
\end{multline*}
\end{lemma}
\begin{proof}
Just as in Proposition~\ref{probconvergenceofpaths}, we approximate \eqref{selective generator, births, no space, neutral parent} using \eqref{taylorforproduct} by 
\begin{multline*}
\frac{sN}{S} f(\eta)
\Bigg(
\sum_i
\int_{l_{i-1}}^{l_i}
\frac{uK}{J} e^{- \frac{uK}{J}v^*} 
\bigg[
\frac{g^\prime (\kappa_i, l_i)}{g(\kappa_i, l_i)} (v^* - l_i)
\\
+
\sum_{l \neq l_i}
\frac{g^\prime (\kappa, l)}{g(\kappa, l)}
\left(
\frac{l\frac{u}{J}}{1 - \frac{u}{J}}
- \ind_{l > l_i}
\frac{\sigma(\kappa)}{\sigma(\kappa^{*})}
\frac{l_i - v^*}{1 - \frac{u}{J}}
\right)
\bigg] 
\mathrm{d} v^*
\Bigg)
\\
=
Nf(\eta)
\sum_i
\Bigg[
\frac{g^{\prime}(\kappa_i, l_i)}{g(\kappa_i,l_i)}
\left(
-(l_{i} - l_{i-1}) e^{- \frac{uK}{J} l_{i-1}}
+ \frac{J}{uK}
\left(
e^{- \frac{uK}{J} l_{i-1}} - e^{- \frac{uK}{J} l_{i}}
\right)
\right)
\\
+
\sum_{j \neq i}
\left(
\frac{g^{\prime}(\kappa_j , l_j )}{g(\kappa_j ,l_j)}
\frac{l_j \frac{u}{J}}{1 - \frac{u}{J}}
\left(
e^{- \frac{uK}{J} l_{i-1}} - e^{- \frac{uK}{J} l_{i}}
\right)\right)
\\
+
\sum_{j \neq i}\sum_{j<i} 
\frac{\sigma(\kappa_j)}{\sigma(\kappa_i)}
\frac{g^{\prime}(\kappa_j , l_j )}{g(\kappa_j ,l_j)}
\frac{1}{1 - \frac{u}{J}}
\left(
-(l_i - l_{i-1}) e^{- \frac{uK}{J} l_{i-1}}
+
\frac{J}{uK}
\left(
e^{- \frac{uK}{J} l_{i-1}} - e^{- \frac{uK}{J} l_{i}}
\right)
\right)
\Bigg]
\\
=: 
\frac{\mathcal{S}_1}{S}  + 
\frac{sN}{S} f(\eta)
\sum_{i}
\sum_{j<i} 
\left(
\frac{\sigma(\kappa_j)}{\sigma(\kappa_i)}
- 1
\right)
\frac{g^{\prime}(\kappa_j , l_j )}{g(\kappa_j ,l_j)}
\frac{1}{1 - \frac{u}{J}}
\left(
-(l_i - l_{i-1}) e^{- \frac{uK}{J} l_{i-1}}\right.
\\
\left.+
\frac{J}{uK}
\left(
e^{- \frac{uK}{J} l_{i-1}} - e^{- \frac{uK}{J} l_{i}}
\right)
\right)
,
\end{multline*}
where we have used integration by parts and $\mathcal{S}_1$ is equal to the terms appearing in \eqref{no space first approximation of level part of generator}. Now we notice that $\mathcal{S}_{1}$ is  multiplied by $1/S$ and since $S \to \infty$ it can be neglected.
Consider then the second term.
We can neglect the factor $1 - u/J$ at the cost of an error of order $NK/J^{3}$, which tends to zero. 
We change the order of summation
to approximate it by
\begin{equation*}
\frac{sN}{S}
f(\eta)
\left(
\frac{\sigma(\kappa_r)}{\sigma(\kappa_c)}
- 1
\right)
\sum_j
\frac{g^{\prime}(\kappa_j , l_j )}{g(\kappa_j ,l_j)}
\sum_{i < j}
\Bigg(
-(l_i - l_{i-1}) e^{- \frac{uK}{J} l_{i-1}}
+
\frac{J}{uK}
\left(
e^{- \frac{uK}{J} l_{i-1}} - e^{- \frac{uK}{J} l_{i}}
\right)
\Bigg)
.
\end{equation*}
Now we  proceed precisely as in the proof of Proposition~\ref{probconvergenceofpaths}. The final approximation of  $A_{sel,3}^{N}$ then takes the form
\begin{equation*}
f(\eta)
\left(
\frac{\sigma(\kappa_r)}{\sigma(\kappa_c)}
- 1
\right)
\sum_{j}
\frac{g^{\prime}(\kappa_j , l_j )}{g(\kappa_j ,l_j)}
\left(
\frac{suN}{JS}l_{j} + \cO \left(\frac{1}{S}\right)
\right)
.
\end{equation*}
\end{proof}

We now turn our attention to $A_{sel,1}^{N} - A_{sel,3}^{N}$.
We must treat the cases when the rare type is favoured or unfavoured separately.
 However, we show in both cases that $\IE[ \left| A_{sel,1}^{N} - A_{sel,3}^{N} \right|] \to 0$ and conclude with Jensen's inequality.

\begin{lemma}
\label{unfavouredselectivegenerator}
Let $l^*_{neu}$ be the lowest level above $v^*$. Let $l^*_{sel}$ be the level above $v^*$ which minimises $\frac{l - v^*}{\sigma}$. Then
\begin{align*}
\IE
\left[
\left|
\frac{sN}{S} f(\eta) \int_0^\infty \frac{uK}{J} e^{-\frac{uK}{J} v^*} g(v^*, \kappa^*) \sum_l \frac{g^\prime(l)}{g(l)}
\left(
\cJ_{sel}(l,l^*_{sel}, v^*) - \cJ_{sel}(l, l^*_{neu}, v^*) 
\right)
\mathrm{d} v^*
\right|
\right]
\to 0
.
\end{align*}
\end{lemma}
\begin{proof}
We note that for $l < v^*$, $\cJ_{sel}(l,l^*_{sel}, v^*) = \cJ_{sel}(l, l^*_{neu}, v^*) $
as
\begin{align*}
\cJ_{sel}(l, l^*, v^*)
=
\begin{cases}
\frac{1}{1 - \frac{u}{J}} \left( l - (l^* - v^*) \frac{\sigma(\kappa)}{\sigma(\kappa^*)} \right) 
&
l > v^* , \; l \neq l^*
,
\\
v^* 
&
l = l^*
,
\\
\frac{1}{1 - \frac{u}{J}} l
&
l< v^*
,
\end{cases}
\end{align*}
and so we need only consider
\begin{align}
\label{selectioncorrectionapproximation}
\frac{sN}{S} f(\eta) \int_0^\infty \frac{uK}{J} e^{-\frac{uK}{J} v^*} g(v^*, \kappa^*) \sum_{l > v^*} \frac{g^\prime(l)}{g(l)}
\left(
l^*_{neu} - l^*_{sel}
\right)
\frac{\sigma(\kappa)}{\sigma(\kappa^*)}
\mathrm{d} v^*
.
\end{align}
We proceed as in Section~\ref{Subsection Neutral model proof} and consider $v^*$ in the interval $[l_{i-1} , l_i]$. 
We note that $l^*_{neu} = l_i$ and if $l_i$ is of the favoured type then $l^*_{sel} = l_i$ also. If $l_i$ is unfavoured we denote the lowest favoured type above $l_i$ by $l^s_i$ then $l^*_{sel} \neq l_i$ if and only if
\begin{align*}
l_{i-1} \leq v^* \leq 
\left(
l_i - \frac{\sigma_w}{\sigma_s - \sigma_w} (l^s_i - l_i) \right)\wedge l_{i-1}
,
\end{align*}
where $\sigma_w = \sigma(\kappa_w)$, $\sigma_s = \sigma(\kappa_s)$, $\kappa_w$ is the unfavoured type and $\kappa_s$ is the favoured type.
Therefore we see that \eqref{selectioncorrectionapproximation} can be written as
\begin{align*}
\frac{sN}{S} f(\eta)
\sum_{(l_i, \kappa_i) \in \eta, \kappa_i = \kappa_w}
\int_{l_{i-1}}^{(
l_i - \frac{\sigma_w}{\sigma_s - \sigma_w} (l^s_i - l_i) )\wedge l_{i-1}}
\frac{uK}{J} e^{-\frac{uK}{J} v^*} g(v^*, \kappa^*) \sum_{l> v^*} \frac{g^\prime(l)}{g(l)}
\left(
l^*_n - l^*_s
\right)
\frac{\sigma(\kappa)}{\sigma(\kappa^*)}
\mathrm{d} v^*
.
\end{align*}
We can then integrate and use a Taylor expansion to see that we need only show
\begin{align*}
\IE
\left[
\frac{sN}{S} f(\eta) \sum_j \frac{g^\prime(l_j)}{g(l_j)}
\sum_{(l_i, \kappa_i) \in \eta, l_i < l_j, \kappa_i = \kappa_w} \frac{uK}{J} (l^s_i - l_i)
\left(
l_i - l_{i-1} - \frac{\sigma_w}{\sigma_s - \sigma_w} (l^s_i - l_i) 
\right) \wedge 0
\right]
\to 0
.
\end{align*}
If the rare type is unfavoured we can then conclude by noting that $\IE \left[ \frac{suNK}{SJ} (l^s_i - l_{i})(l_i - l_{i-1}) \right] = \cO\left(\frac{NK}{SJ} \frac{1}{(K-Z)^2} \right)$.

If the rare type is favoured we may bound it by
\begin{align}
\label{finalselectionapproximation}
\sum_{\hat{l}_{k-1} < l_i < \hat{l}_k} \frac{suNK}{JS} (l_i - l_{i-1})^2 \ind_{\{ l_i - l_{i-1} - \frac{\sigma_w}{\sigma_s - \sigma_w}(\hat{l}_k - l_i) > 0 \}}
,
\end{align}
where we recall that $\hat{l}_k$ is the $k$th highest rare level. 
As before, we condition on there being $n$ individuals of the common type with levels between $\hat{l}_k$ and $\hat{l}_{k-1}$.
%
We can then  consider the $n$ levels as independent uniformly distributed rather than the ordered levels $l_i$.
Suppose that $z, z_1, \dots, z_{n-1}$ are independent uniformly distributed random variables on $[0,1]$, $Y = 1 - z$ and $X = \min \left( \{z\}, \{ (z - z_i)\}_{z_i < z} \right)$.
We then see that \eqref{finalselectionapproximation} is bounded by
$$n\frac{suKN}{JS} \IE\left[
X^2 \ind_{X >cY}
\right] .$$
We now look at
\begin{align}
\label{Funny_expectation_squared_la}
\IE[X^2 \ind_{X >cY} ]
= &
\int_0^1 \IE[X^2 \ind_{X > c (1 - u)} | z = u] \mathrm{d} u
\\
= &
\int_{\frac{c}{1+c}}^1 \IE[X^2 \ind_{X > c (1 - u)} | z = u] \mathrm{d} u
\nonumber
\\
= &
\int_{\frac{c}{1+c}}^1 \int_{0}^u x^2 \ind_{x > c (1 - u)} \mathrm{d}\IP[X=x | z = u] \mathrm{d} u
\nonumber
\\
= &
\int_{\frac{c}{1+c}}^1 \int_{c(1-u)}^u x^2  \mathrm{d}\IP[X=x | z = u] \mathrm{d} u
\nonumber
.
\end{align}
Since
\begin{align*}
\IP[X >x | z = u]
= &
\IP\big[ [ u-x,u] \cap \{z_i\} = \emptyset \big]
= (1-x)^{n-1}
,
\\
\mathrm{d} \IP[X = x | z = u]
= &
(n-1)(1-x)^{n-2}
,
\end{align*}
we may bound \eqref{Funny_expectation_squared_la} by
\begin{align*}
\IE[X^2 \ind_{X >cY} ]
\leq &
\int_{\frac{c}{1+c}}^1 \int_{c(1-u)}^1 x^2 (n-1)(1-x)^{n-2} \mathrm{d}x \;  \mathrm{d} u
\\
= &
\int_{\frac{c}{1+c}}^1
\int_0^{1 - c + cu}
(n-1)(1-v)^2v^{n-2} \mathrm{d} v \; \mathrm{d} u
\\
\leq &
\frac{n-1}{c}
\left[
\frac{v^{n}}{n(n-1)} - 2 \frac{v^{n+1}}{n(n+1)} + \frac{v^{n+2}}{(n+1)(n+2)} 
\right]_0^1
\\
= &
\frac{6}{c n (n+1)(n+2)}
.
\end{align*}
We can then conclude by noting that
\begin{align*}
\sum_{n=1}^\infty \frac{1}{(n+1)(n+2)} \frac{(K-Z^{N})^{n} e^{-(K-Z^{N})}}{n!}   = &\cO\left(\frac{1}{(K-Z^{N})^2} + \exp(-(K-Z^{N}) ) \right)
,
\end{align*}
which again leads to an error of order $\cO\left( \frac{NK}{SJ} \frac{1}{(K-Z)^2} \right)$, completing the proof of the Lemma.
\end{proof}

\begin{proof}[Proof of Proposition~\ref{selectivepropositionA}]
By an application of Lemma~\ref{Lemma enough to bound variance lemma},
 Proposition~\ref{selectivepropositionA} follows from a combination of Lemma~\ref{Lemma selective estimate neutral parent} and Lemma~\ref{unfavouredselectivegenerator}. 
%
\end{proof}

\begin{lemma}
\label{selective intensity lemma}
Define $\beta^{N}$ as in \eqref{0d auxiliary submartingale}. Then $\beta^{N}$ is a dominated by a bounded submartingale.
\end{lemma}
\begin{remark}
Although the statements of Lemma~\ref{neutral intensity lemma} and Lemma~\ref{selective intensity lemma} are the same, the behaviour of the  process $\beta^{N}$ is subtly different, as in  Lemma~\ref{selective intensity lemma} the movement of the levels is affected by selective events. This difference does not effect the proof however we include them as seperate lemmas for completeness.
\end{remark}
\begin{proof}
As in the proof of Lemma~\ref{neutral intensity lemma} (and using the notation introduced there)
\begin{multline*}
\frac{\IE\left[
\sum_{l(t) \in \eta_{t}^{r}} \left(l(t) - \hat{l}(t)\right) h(l(t))
-
\sum_{l(0) \in \eta_{0}^{r}} \left(l(0) - \hat{l}(0)\right) h(l(0))
\right]}
{t}
\\
=
N
\int_0^\infty 
\left(1 + \frac{s}{S} \right)
 \frac{uK}{J} e^{-\frac{uK}{J}v^*}
\Bigg\{
\sum_{l(0) \in \eta_{t}^{r}}
\Bigg\{ \left[\left(\cJ(l(0)) - \cJ(\hat{l}(0))\right) - \left(l(0) - \hat{l}(0)\right)\right] h(l(0))
\\
+
\left(\cJ(\hat{l}(0)) - \hat{l}(t)\right) h(l(0))
\\
+
\left(l(t) - \hat{l}(0) \right)\left(\cJ(l(0)) - l(0) \right) \cO( || {h}^\prime ||_{\infty} )
\Bigg\}
+
\sum_{i = 2}^\infty
\left(v_i - \hat{v}_i\right) h(v_i)\ind_{\{v_i \text{ is rare}\} }
\Bigg\} \mathrm{d} v^*
+
\cO(t)
.
\end{multline*}
Considering combination of Proposition~\ref{probconvergenceofpaths} and Proposition~\ref{selectivepropositionA} we know that the parts involving $\mathcal{J}(l(0)) - l(0)$ converge to $al(0) - bl^{2}(0)$. 
Since $\frac{s}{S} \to 0$ and observing Lemma~\ref{unfavouredselectivegenerator}, we may proceed as in the proof of   Lemma~\ref{neutral intensity lemma} to conclude.
\end{proof}

\subsubsection{Model with selection in fluctuating environment - proof of Theorem~\ref{NS fluctuating selection theorem}}
\label{Subsection fluctuating model proof}


\begin{proof}
We begin by identifying  the limit. 
Recall that the rescaled generator takes  the form
\begin{align}\label{fluctuating generator2}
A^{N}f(\eta,\zeta) = A_{neu}^{N}f(\eta,\zeta) + \hat{S}A_{sel}^{N}f(\eta,\zeta) + \hat{S}^{2}A_{env}f(\eta,\zeta),
\end{align}
where  $A_{neu}^{N}$ is defined as in \eqref{noselecgenerator}, $A_{sel}^{N}$ is defined as in \eqref{selgenerator}, and
\begin{align*}
A_{env}f(\eta,\xi)  = \mathbb{E}_{\pi}[f(\eta,\xi)]- f(\eta,\xi)
.
\end{align*}

We also use \eqref{Asel for f1} and that $ A_{neu}f_{1}$ is of order $1$ by previous calculations.
To identify the correct limit it is therefore enough to evaluate $A_{sel}^{N}\left( 
f_N(\eta, \xi) 
 \right)$, which can be approximated as
\begin{multline*}
A_{sel}^{N}\left( 
f_N(\eta, \xi) 
 \right)
 =
A_{sel}^{N}
\left(
f(\eta)
\frac{suN}{JS}
\left(
\frac{\sigma(\kappa_r,\zeta)}{\sigma(\kappa_c,\zeta)}
- 1
\right)
\sum_{j}l_{j} 
\frac{g^{\prime}(\kappa_j , l_j)}{g(\kappa_j ,l_j)}
\right)
\\
=
\left[
\frac{suN}{JS}
\left(
\frac{\sigma(\kappa_r,\zeta)}{\sigma(\kappa_c,\zeta)}
- 1
\right)
\right]^{2}
f(\eta)
\\
\times
\left\{
\left(
\sum_{i \neq j}l_{i}l_{j} 
\frac{g^{\prime}(\kappa_i , l_i,)g^{\prime}(\kappa_j , l_j)}{g(\kappa_i ,l_i)g(\kappa_j ,l_j)}
\right)
+  \left(
\sum_{j}l_{j} 
\frac{g^{\prime}(\kappa_j , l_j )+ l_{j}g^{\prime\prime}(\kappa_j , l_j)}{g(\kappa_j ,l_j)}
\right)
\right\}
+ \cO \left(\frac{1}{S} + \frac{1}{K^{1/2}}\right)
\end{multline*} 
%
%
%
The bound \eqref{epsilontozero} follows directly from proofs of Propositions~\ref{selectivepropositionA},~\ref{probconvergenceofpaths},~\ref{probconvergenceofbirths}, with Lemma~\ref{neutral intensity lemma} modified in a way analogous to Lemma~\ref{selective intensity lemma}.
\end{proof}

\section{SuperBrownian motion in a random environment}
\label{Section SuperBrownian motion in a random environment}
In this section we present a precise definition of superBrownian motion in a random environment.
We begin by defining Branching Brownian Motion in a random environment
and
recalling the original definition of the corresponding superprocess from \cite{mytnik:1996a}. Then we describe a lookdown construction for both of these models based on the ideas in \cite{kurtz/rodrigues:2011}.

\subsection{Definitions}
Branching Brownian Motion in a random environment (BBMRE) can be described as follows. 
Imagine a collection of particles on $\mathbb{R}^{d}$. 
Each particle moves according to independent standard Brownian motions. 
Each  particle, if alive, gives birth to one new particle at a time, at  rate  $a$. 
The initial  offspring location is the same as that of the parent. 
After the birth, the offspring  moves  and reproduces independently of all other particles.
The particles die at instantaneous rate $a - \zeta_{t}(x)b$, where $\zeta_{t}(x)$, taking values in $\{-1,1\}$ as before,  models the random environment and $x$ is the current location of the particle. 
We assume that $a > b > 0$.
If $\zeta_{t}(x)$ is positive, the particle is less likely to die, if it is negative, it is more likely to do so. 
The evolution of the environment and  particles are independent. 
We give a more formal definition below.

We begin by recalling the description of our environment, which is used for all models in this section. Our environment is modelled through a simple random field.
\begin{definition}
\label{Branching -  Environment definition}
Let $\Pi^{env}$ be a Poisson process 
with intensity $E$, dictating the times of the changes in the environment.
Let $q(x,y)$ be an element of $C_{0}\left(\mathbb{R}^d \times \mathbb{R}^d\right)$ (continuous functions vanishing at infinity) and
let $\{\xi^{(m)}(\cdot)\}_{m\geq 0}$ be a  family of identically distributed 
random fields on $\mathbb{R}^{d}$ such that 
\begin{align*}
\mathbb{P}\left[\xi^{(m)}(x) = -1  \right] = &  \frac{1}{2}=
\mathbb{P}\left[\xi^{(m)}(x) =+1 \right],\\
\mathbb{E}\left[ \xi^{(m)}(x)\xi^{(m)}(y) \right] = & q(x,y).
\end{align*}
Set $\tau_0=0$ and write $\{\tau_m\}_{m\geq 1}$ for the points in $\Pi^{env}$ and define
$$
\zeta(t,\cdot):=\sum_{m=0}^\infty\xi^{(m)}(\cdot)
\mathbf{1}_{[\tau_m,\tau_{m+1})}(t).
$$
\end{definition}

Since the exact labelling of our particles is not important, we identify the particle component of the process with a counting measure, that is for a vector $\overline{x} = (x_{1}, \dots x_{n})$
\begin{align*}
\mu_{\overline{x}} = \sum_{i}\delta_{x_{i}} \quad x_{i} \in \mathbb{R}^{d}.
\end{align*}
We are now ready to state the definition of Branching Brownian motion in a random environment.

\begin{definition}[Branching Brownian motion in a random environment (BBMRE)]
\label{Definition -  BBMRE}
Branching Brownian motion in the  random environment  $\zeta$ is the stochastic process taking values  in purely atomic measures on $\{-1, 1 \} \times \mathbb{R}^{d}$ whose evolution consists of four ingredients.

\begin{enumerate}
\item \textbf{Spatial motion} The location of each particle, $x_{i}$, evolves according to a standard Brownian motion, independently of all other particles.
\item \textbf{Birth events} At exponential rate $a$ (independent for each particle), a particle gives  birth to a new particle (a new particle is added to the system). 
The location of the offspring is the same as the location of their parent. 
The behaviour of the new particle after the birth is independent of all other particles.
\item \textbf{Death events} Each particle dies (is removed from the system) at instantaneous  rate $a - \zeta_{t}(x_{i})b$, where $x_{i}$ is its location.
\item \textbf{Environment changes} The environment evolves as described in Definition~\ref{Branching -  Environment definition}.
\end{enumerate} 

\end{definition}
Alternatively, we may define the BBMRE by the means of the generator. 
For a counting measure $\mu  = \sum_{i}\delta_{x_{i}}$, define
$$f(\mu) = \sum_{i}f(x_{i}).$$
Let $f(\zeta,\mu) = f_{0}(\zeta)f_{1}(\mu) = f_{0}(\zeta)\pi_{i}h(x_{i})$ be a function such that $h \in C_{c}(\mathbb{R}^{d})$, that is $h$ is a continuous function with a  compact support. 
 We define the generator of the BBMRE as  
\begin{multline*}
\mathcal{L}f(\zeta,\mu) \\
=f_{1}(\mu)A^{env}f_{0}(\zeta) + f_{0}(\zeta)\left(\sum_{i = 1}^{n}B f_{1}(\mu) + \sum_{i} a\big(f_{1}(\mu_{b(x|x_{i})}) - f_{1}(\mu)\big)\right.\\
\left.
 + \sum_{i}\big( a - \zeta(x_{i}) b(f_{1}(\mu_{d(x|x_{i})}) - f_{1}(\mu)\big)
\right),
\end{multline*}
where $\mu_{b(x|x_{i})}$ denotes the addition of a particle at location $x_{i}$, and $\mu_{d(x|x_{i})}$ denotes a removal of the particle at location $x_{i}$.

It is well known that  the high density limit for the Branching Brownian Motion gives  rise to a SuperBrownian motion (see e.g. \cite{etheridge:2000}). An analogous result (under certain scaling of the environment) was  established by \cite{mytnik:1996a} for branching random walk in a random environment. The limiting object is 
 SuperBrownian motion in a random environment (SBMRE). 
\begin{definition}[SuperBrownian motion in a random environment (SBMRE)]
\label{Definition SuperBrownian}
Let $q(x,y)$ be a covariance function which belongs to $C_{0}\left(\mathbb{R}^d \times \mathbb{R}^d\right)$. The superBrownian motion in a random environment is the (unique) process for which, for all $\phi \in \mathcal{D}(\Delta)$,

\begin{align}
\label{SBMRE mytnik generator martingale part}
X_{t}(\phi) = X_{t}(\phi) - X_{0}(\phi) - \int_{0}^{t}\frac{1}{2}X_{s}(\Delta\phi)\mathrm{d}s
\end{align}
is a square-integrable martingale with  quadratic variation given by
\begin{align*}
\langle X(\phi)\rangle_{t} = \int_{0}^{t}X_{s}(\phi^{2})\mathrm{d}s 
+ 
\int_{0}^{t}\int_{\mathbb{R}^{d}\times \mathbb{R}^{d}}q(x,y)\phi(x)\phi(y)X_{s}(\mathrm{d}x)X_{s}(\mathrm{d}y)\mathrm{d}s.
\end{align*}
\end{definition}


An equivalent characterisation of the SBMRE can be given in terms of the generator, see Theorem 4.8 in \cite{mytnik:1996a}. Namely,  
for $f\in \bar{C}^{2}(R_{+}), \phi \in\mathcal{D}(\Delta)$ and $q(x,y)$ as in Definition~\ref{Definition SuperBrownian}, 
the generator is given by
\begin{align*}
\mathcal{L}f(\mu(\phi)) = 
f^{\prime}(\mu(\phi))\mu(\Delta\phi) + 
\frac{1}{2}f^{\prime\prime}(\mu(\phi))\left(\mu(\phi^{2}) 
+
\int_{\mathbb{R}^{d}\times \mathbb{R}^{d}}q(x,y)\phi(x)\phi(y)\mu(\mathrm{d}x)\mu(\mathrm{d}y)\mathrm{d}s \right)
.
\end{align*}
We would like to point out that the process of Definition~\ref{Definition SuperBrownian motion drift} which is obtained as a limiting behaviour of the scaled SLFV in a random environment differs from the one from Definition~\ref{Definition SuperBrownian} by a presence of a drift term.  

\begin{remark}
\label{SBMRE uniqueness remark}
Uniqueness of solutions to the martingale problem of Definition~\ref{Definition SuperBrownian} is not immediately clear. 
It was established by \cite{mytnik:1996a} using a novel approximate duality technique, and later re-proved by the means of the log-Laplace transform by \cite{crisan:2004}. 
A uniqueness result for the process of Definition~\ref{Definition SuperBrownian motion drift}  is a simple consequence of Dawson's Girsanov Theorem, see \cite{dawson:1978} and \cite{etheridge:2000}, Chapter~7.
\end{remark}
\begin{remark}
A model similar to that in \cite{mytnik:1996a} was studied in \cite{sturm:2003}. The main difference between the two is in the behaviour of the environment. 
For the limiting model in \cite{sturm:2003}, \eqref{SBMRE mytnik generator martingale part} is again a martingale, but with quadratic variation of the form 
\begin{align*}
\langle X(\phi)\rangle_{t} = 
\int_{0}^{t}\int_{\mathbb{R}^{d}\times \mathbb{R}^{d}}q(x,y)\phi(x)\phi(y)X_{s}(\mathrm{d}x)X_{s}(\mathrm{d}y)\mathrm{d}s.
\end{align*}
In this case, the density of the process can be described as a solution to an SPDE in all dimensions, whereas the SBMRE has a density only in dimension one and the analogous SPDE has no solution in dimensions $d \geq 2$.
\end{remark}
\begin{remark}
\label{Nakashima papers}
An alternative construction of the SBMRE has been suggested by \cite{nakashima:2015}.
The construction in this paper is based on the model introduced by \cite{birkner/geiger/kersting:2005}. 
\end{remark}


\subsection{Lookdown representation for BBMRE and SBMRE}


In this section we describe a new construction of the SBMRE, inspired by the constructions of \cite{kurtz/rodrigues:2011}. 
As a by-product of this construction, we provide a lookdown construction for the BBMRE. 
A precise statement of the results is given in Theorem~\ref{LC BBMRE} and Theorem~\ref{LC SBMRE}.
We will use the lookdown representation of SBMRE in Section~\ref{Scaling limits of the SLFV - dynamics of the rare type} to describe the behaviour of the `rare' type (by which we mean a new mutation establishing in the population) in the Spatial Lambda-Fleming-Viot model with fluctuating selection.


%
%

In order to motivate what follows, let us informally describe the construction of SBMRE from  \cite{mytnik:1996a}. 
The SBMRE is constructed via a series of approximations. 
Consider a sequence of 
mean $0$ random fields $\zeta_{k}$, taking values in $\{-1,1\}$, with  correlation as in Definition~\ref{Branching -  Environment definition} (\cite{mytnik:1996a} considered a more general class of random fields, but this one is sufficient for our purposes). At stage $n$ of the approximation, we start with a population with size of order $n$. Over each time interval $(k/n,(k+1)/n), k \in \mathbb{N}$, each individual, if alive, moves independently of  all the others, according to a standard Brownian motion. At times $k/n$, each individual either splits into two with probability $1/2 + \zeta_{k}(x_{i})/\sqrt{n}$, or dies with probability $1/2 - \zeta_{k}(x_{i})/\sqrt{n}$. 
The state of the environment is resampled with each reproduction event. 
If for each set B we define 
\begin{align*}
X_{t}^{n}(B) = \frac{\text{number of particles in } B \text{ alive at time } t}{n},
\end{align*}
and assume that $X_{0}^{n}$ converges to some $X_{0}$, then passing to the limit as $n$ tends to infinity the process $X^{n}$ converges to a SBMRE with initial condition $X_{0}$. 
Intuitively, this procedure corresponds to increasing the rate of the branching events in BBMRE   by $n$, while scaling down the impact of the environment by $\sqrt{n}$.
Observe that \cite{mytnik:1996a} considers a model with non-overlapping generations. This is for purely technical reasons.

We now move to a description of a lookdown construction for BBMRE (with overlapping generations). 
Fix $\lambda \in \mathbb{R}$.
Consider a system of particles in the geographical space $\mathbb{R}^{d}$. 
Each particle moves independently, according to a standard Brownian motion. 
Each particle $i$ is assigned a level, $l_{i}$. 
The level takes values in $[0,\lambda]$. 

In order to take environmental fluctuations into account we use the process $\zeta_{t}(x) \in \{-1, 1\}$, introduced in Definition~\ref{Branching -  Environment definition}, to model the environment. 
The branching rate of the particles depends on both their level and the state of the environment. 
We assume that every particle gives  birth at instantaneous rate $2a(\lambda - l_{i}(t))$, with the location of offspring being the same as the location of the parent. 
The initial level of the offspring is distributed uniformly on the interval $[l_{i},\lambda]$. 
The levels of the particles evolve according to an ODE with a random coefficient  
\begin{align*}
\frac{\mathrm{d}l_{i}}{\mathrm{d}t} = al_{i}^{2} - \zeta(x_{i})\sqrt{\lambda}bl_{i}.
\end{align*}
The particle dies when its level reaches  $\lambda$. 
We are now ready to define  the process of interest. 


\begin{definition}[Lookdown representation of Branching Brownian motion in a random environment]
\label{Definition - Lookdown for BBMRE}
The lookdown representation of BBMRE is the process taking values in  $(\mathbb{R}^{d}\times\mathbb{R})^{\infty} \times \{-1,1\}$
with  dynamics specified by four components.
\begin{enumerate}
\item \textbf{Spatial motion} The spatial location of each particle, $x_{i} \in \mathbb{R}^{d}$, evolves according to a standard Brownian motion with generator $B  =\frac12 \Delta$.
\item \textbf{Birth events} Each particle gives birth at instantaneous rate $2a(\lambda - l_{i}(t))$, where $l_i$ is the level of the particle. 
The spatial location of the offspring is the same as the location of their parent, $x_i$. 
The level of the offspring is chosen uniformly at random from  $[l_i, \lambda]$.
\item \textbf{Level movement} The level of each of the particles evolves according to the equation
\begin{align*}
\frac{\mathrm{d}l_{i}}{\mathrm{d}t} = al_{i}^{2} - \zeta(x_{i})\sqrt{\lambda}bl_{i}.
\end{align*}
\item \textbf{Environment changes} The  environment evolves as described in Definition~\ref{Branching -  Environment definition}.
\end{enumerate}

\end{definition}
A formal definition, via the generator of the process, is given in \eqref{Mytnik lambda generator}.
It is convenient to identify the process with the counting measure 
$$\mu_{\overline{x},\overline{l}} = \sum_{i}\delta_{x_{i},l_{i}}.$$

%

We now state the first result of this section.

\begin{theorem}[Lookdown construction for Branching Brownian Motion in random environment]
\label{LC BBMRE}
Let $\mu_{0}$ be the measure associated with the initial state of the population. Let $\gamma : \mathcal{N}\big(\mathbb{R}^{d}\times[0,\lambda)\big) \rightarrow \mathcal{M}(\mathbb{R}^{d})$ be given by
\begin{align*}
\gamma\left(\sum_{i}\delta_{x_{i},l_{i}}\right) 
= 
\frac{1}{\lambda}\sum_{l_{i}}\delta_{x_{i}}
.
\end{align*}

The lookdown process of Definition~\ref{Definition - Lookdown for BBMRE} corresponds to BBMRE, in the sense that if $\eta(\overline{x},\overline{l})$ is a solution to the martingale problem for the process defined in Definition~\ref{Definition - Lookdown for BBMRE}, then $\gamma\left(\eta(\overline{x},\overline{l})\right)$ is the solution to the martingale problem for the process of Definition~\ref{Definition -  BBMRE}.
\end{theorem}

Our next aim is to write down  the generator of the process obtained by passing  to the  limit $\lambda \to \infty$, which would correspond to passing to the limit with $n$ passing to infinity  in the sequence of approximations described at the beginning of this section.
In order to formulate the result for the limiting process, we need to consider a special class of test functions of the form
\begin{align}
\label{test_functions_BBMRE}
 f(\zeta,x,l) = f_{0}(\zeta)f_{1}(x,l) =  f_{0}(\zeta)\prod_{i}g(x_{i},l_{i}), 
\end{align} 
with the additional requirement
\begin{align*}
g(x_{i},l_{i}) = 1 \text{ for } l_{i} > \lambda_g,
\end{align*}
which ensures that we ignore all individuals with levels above $\lambda_g$.
We are now in a position to state the main result of this section.

\begin{theorem}[Lookdown construction for superBrownian motion in random environment]
\label{LC SBMRE}
Consider the process with  generator given by
\begin{multline}
\label{SBMRE LIMITING GENERATOR}
A_{\infty}\left(f_{1}\right) = f_{1}(x,l)\sum_{i} \frac{Bg(x_{i},l_{i})}{g(x_{i},l_{i})} 
\\
 + f_{1}(x,l)\sum_{i}2a\int_{l_{i}}^{\lambda}(g(x_{i},v) - 1)\mathrm{d}v 
+ 
f_{1}(x,l)\sum_{i}
(al_{i}^{2}  - b^{2}l_{i})\frac{\partial_{l}g(x_{i},l_{i}) }{g(x_{i},l_{i})}
\\
- f_{1}(x,l)\sum_{i}
b^{2}l_{i}
\left(\sum_{j \ne i}q(x_{i},x_{j})l_{j}\frac{\partial_{l}g(x_{i},l_{i})\partial_{l}g(x_{j},l_{j})}{g(x_{i},l_{j})g(x_{i},l_{i})}   + \frac{l_{i}\partial^{2}_{l}g(x_{i},l_{i})}{g(x_{i},l_{i})}\right). 
\end{multline}
Let $\mu_{0}$ be the measure associated with the initial state of the population in the limit. Define the map $\gamma : \mathcal{N}\big(\mathbb{R}^{d}\times[0,\infty)\big) \rightarrow \mathcal{M}(\mathbb{R}^{d})$ by
\begin{align*}
\gamma\left(\sum_{i}\delta_{x_{i},l_{i}}\right) 
= 
\begin{cases}
\lim_{\lambda \to \infty} \frac{1}{\lambda}\sum_{l_{i} \leq \lambda}\delta_{x_{i}} & \text{if the limit exists}, \\
\mu_{0} & \text{otherwise } . 
\end{cases}
\end{align*}
The process described by the limiting generator \eqref{SBMRE LIMITING GENERATOR} corresponds to the SBMRE of Definition~\ref{Definition SuperBrownian motion drift}, in the sense that if $\eta(\overline{x},\overline{l})$ is a solution to the martingale problem for the process described by the limiting generator \eqref{SBMRE LIMITING GENERATOR}, then $\gamma(\eta(\overline{x},\overline{l}))$ is a solution to the martingale problem for the process in Definition~\ref{Definition SuperBrownian motion drift}.
\end{theorem}







We now wish to specify the generator for the process in  Definition~\ref{Definition - Lookdown for BBMRE}. 
In general, we consider the set of test functions of the form \eqref{test_functions_BBMRE}.
To make the functional setup more precise, we need to borrow the following condition from \cite{kurtz/rodrigues:2011}, which guarantees that the assumptions of the Markov Mapping Theorem are satisfied.

\begin{condition}[Based on \cite{kurtz/rodrigues:2011}, Condition 3.1]
\label{LC Condition 3.1}
We assume that the following conditions on the operator $B$, the coefficients $a,b$ and the test functions $g$ are satisfied.
\begin{enumerate}
\item The operator $B$ is defined on a subset of the space of bounded continuous functions, and its domain is closed under multiplication and separating.
\item The test functions are of the form \eqref{test_functions_BBMRE}, where 
$$g(x,l) = \prod_{j=1}^{m}(1 - g_{1}^{j}(x)g_{2}^{j}(l))$$
and $g^{j}_{1} \in \mathcal{D}(B)$ and $g^{j}_{2}$ are twice differentiable with support in $[0,\lambda]$. Moreover,
$$0 \leq g_{1}^{j}g^{j}_{2}  < \rho_{g} <1.  $$
\item There exists a continuous, non-negative  function $\psi_{B}$ such that for every test function $g$, and for every $x \in \mathbb{R}^{d}$
\begin{align*}
\sup_{l}|Bg(x,l)| \leq c_{g}\psi_{B}(x)
\end{align*}
for some constant $c_{g}$ which depends only on the function $g$.
\item The following bound holds for the test functions:
\begin{align*}
\int_{0}^{\infty}|g(x,l) - 1|\mathrm{d}l + \sup_{l}\lbrace l + l^{2} \rbrace\partial_{l}g(x,l) \leq c_{g}\psi_{B}(x)
\end{align*}
for some constant $c_{g}^{\prime}$ which depends only on the  function $g$.
\item $a >0$, $\lambda a - \sqrt{\lambda}b > 0.$
\end{enumerate}
\end{condition}

\begin{remark}
We notice that the conditions for the generator of the motion $B$ are satisfied by the  Laplacian, the generator of the Brownian motion. This is the only generator that we consider. We provide the general construction for the sake of completeness and to highlight potential extensions of this work.  
\end{remark}


The generator $A_{\lambda}$ of the process of Definition~\ref{Definition - Lookdown for BBMRE}
 can be written as
\begin{multline}
\label{Mytnik lambda generator}
A_{\lambda}f(\zeta,x,l) =  f_{1}(x,l)A^{env}_{\lambda}f_{0}(\zeta) + f(\zeta,x,l)\sum_{i} \frac{Bg(x_{i},l_{i})}{g(x_{i},l_{i})}  \\
 + f(\zeta,x,l)\sum_{i}2a\int_{l_{i}}^{\lambda}(g(x_{i},v) - 1)\mathrm{d}v \\
+ f(\zeta,x,l)\sum_{i} (al_{i}^{2} - \sqrt{\lambda}\zeta(x_{i}) bl_{i})\frac{\partial_{l_{i}}g(x_{i},l_{i})}{g(x_{i},l_{i})}.
\end{multline}


A naive limiting procedure does not lead to a well defined object, since the form of the generator does not take  into account the cancellations coming from fluctuations in the environment. 
In order to identify the correct limit, we once again use the separation of timescales trick. 
Proofs of Theorem~\ref{LC BBMRE} and Theorem~\ref{LC SBMRE} are presented in Appendix~\ref{Proofs SBMRE BBMRE}.


\begin{remark}
Our computations do not lead to any surprising results - we show that  the SBMRE can be obtained as a scaling limit of BBMRE. 
Our arguments, combined with tightness of the sequence of projected processes, guarantees convergence of the sequence of projected models. 
We therefore provide a new construction for the SBMRE. 
However, we still refer  to results described in Remark~\ref{SBMRE uniqueness remark} to guarantee the uniqueness of solutions to the projected martingale problem. The question of uniqueness of solutions  to the martingale problem for the limiting process with levels (which, by the Markov Mapping Theorem would guarantee uniqueness of solutions to the projected model) will be pursued elsewhere. 
\end{remark}

\section{Scaling limits of the SLFV - dynamics of the rare type}
\label{Scaling limits of the SLFV - dynamics of the rare type}
In this section we are interested in a spatial analogue of the results from Section~\ref{Scaling limits of the LFV - dynamics of the rare type}.
We show that under a certain scaling, the dynamics of the subpopulation with a rare mutation, which is a part of a population evolving according to a version of the Spatial-Lambda-Fleming-Viot model with selection in a fluctuating environment, is given by a superBrownian motion in a random environment. 
Since under our scaling the proof of the result does not differ significantly from the one discussed in Section~\ref{Scaling limits of the LFV - dynamics of the rare type}, our discussion will be rather brief and focus on highlighting the differences and required modifications. 

We begin with a description of the model. As before, we consider a population with two genetic types, rare and common, which we denote by $\kappa_{r}$ and $\kappa_{c}$, respectively.
We also consider the random field $\zeta$, specified by Definition~\ref{Branching -  Environment definition}.
Since in our model we consider a population with a countable number of individuals, the state of the population can be represented as 
\begin{align*}
\eta = \sum_{(x,\kappa,l)}\delta_{x,\kappa,l}.
\end{align*}
We assume that $\eta$ is a conditionally Poisson system with Cox measure $\Xi(\mathrm{d} x, \mathrm{d}\kappa) \times m_{leb}(\mathrm{d}l)$.

The evolution of the population is determined by reproduction events of two types - neutral and selective driven by independent Poisson point processes $\Pi^{neu}$ and $\Pi^{sel}$, which specify the time, location, impact and radius of the events. 
They are analogous to those in Section~\ref{Scaling limits of the LFV - dynamics of the rare type}, but now we assume that an event has a location and radius, only individuals within the ball of given radius centred at the location of the event are affected. 
A rigorous definition of the model follows.

\begin{definition}[Lookdown representation of SLFVSRE]
\label{Lookdown representation of SLFVSRE}
Let $\mu$ be a measure on $(0, \infty)$ and for each $r\in (0,\infty)$, 
let $\nu_{r}$ be a probability measure on $[0,1]$, such that the 
mapping $r \mapsto \nu_r$ is measurable and 
\begin{align}
\label{LC integrability}
\int_{(0,\infty)}r^{d}\int_{[0,1]}u\; \nu_{r}(\mathrm{d}u)\mu(\mathrm{d}r) < \infty.
\end{align}

Fix $s \in [0,1]$. Let $\Pi^{neu}, \Pi^{sel}$ be a pair of  independent Poisson processes with intensity measures $(1-s)\mathrm{d}t \otimes \mathrm{d}y \otimes \mu(\mathrm{d}r)\nu_{r}(\mathrm{d}u)$
and 
$s\mathrm{d}t \otimes \mathrm{d}y \otimes \mu(\mathrm{d}r)\nu_{r}(\mathrm{d}u)$  respectively. Let $\Pi^{env}$ be a Poisson process independent of $\Pi^{neu}, \Pi^{sel}$.

The lookdown representation of SLFVSRE is a process taking values in purely atomic measures on $\mathbb{R}^{d}\times\mathbb{R}\times \{\kappa_{r},\kappa_{c}\} \times \{-1,1\}$ with  dynamics described as follows.
\begin{enumerate}
\item If $(t,y,r,u) \in \Pi^{neu}$  
	\begin{enumerate}
	\item a group of new individuals with levels $(v_{1}, v_{2}, \dots,)$ is added to the 			population within the  ball $B_{r}(y)$. 
	Their levels are distributed according to a Poisson process with intensity $u$.
	\item Let $v^{*} = \min\{ v_{1}, v_{2}, \dots,\}$.
	The type of the new individuals is chosen to be the same as the type of the individual with the lowest level above $v^{*}$ within the  ball $B_{r}(y)$.
	\item As a result of an event the individual originally with level, $l$, and position, $x$, has a new level given by
	\begin{align*}
\mathcal{J}_{neu}(l,l^*,v^*, x,  (y, r) )
=
\begin{cases}
l
&
\text{ if } x \notin B_{r}(y)
,
\\
\frac{1}{1 - \frac{u}{J}} (l - (l^* - v^*))
& \text{ if } l > l^*
, x \in B_{r}(y)
,
\\
\frac{1}{1 - \frac{u}{J}} l
& \text{ if } l < l^*
, x \in B_{r}(x)
,
\\
v^*
& \text{ if } l = l^*
, x \in B_{r}(y)
.
\end{cases}
\end{align*}
	\end{enumerate}

\item If $(t,y,r,u) \in \Pi^{sel}$  
	\begin{enumerate}
	\item a group of new individuals with levels $(v_{1}, v_{2}, \dots,)$ is added to the 			population. Their levels distributed according to a Poisson process with intensity $u$.
	\item Let $v^{*} = \min\{ v_{1}, v_{2}, \dots,\}$.
	The type of the new individuals is chosen to be the same as the type of the individual with level above $v*$ minimizing $(l_{i} - v^{*})/\sigma(\kappa_{i},\zeta)$ within the  ball $B(x,r)$.
	\item As a result of an event the individual originally with level, $l$, and position, $x$, has a new level given by
	\begin{align*}
\mathcal{J}_{sel}(l,l^*,v^*, x,  (y, r) )
=
\begin{cases}
l
&
\text{ if } x \notin B_{r}(y)
,
\\
\frac{1}{1 - \frac{u}{J}} (l - (l^* - v^*)\frac{\sigma(\kappa,\zeta)}{\sigma(\kappa^{*},\zeta)})
& \text{ if } l > l^*
, x \in B_{r/M}(y)
,
\\
\frac{1}{1 - \frac{u}{J}} l
& \text{ if } l < l^*
, x \in B_{r}(y)
,
\\
v^*
& \text{ if } l = l^*
, x \in B_{r}(y)
.
\end{cases}
\end{align*}
	\end{enumerate}
\item The dynamics of $\Pi^{env}$ are specified by Definition~\ref{Branching -  Environment definition}.
\end{enumerate}
\end{definition}

Definition~\ref{Lookdown representation of SLFVSRE} is more
general than we require, however, we include it to underline the 
possibility of extending our results. However, in the interest of keeping our notation as simple as possible, from now on we 
shall specialise to fix the radius and impact of reproduction events. 
\begin{assumption}
From now on, fix $R\in (0,\infty)$ and $\bar{u}\in (0,1)$ and take
\begin{align*}
\mu({dr}) = \delta_{R}(\mathrm{d}r), 
\qquad \nu_{r}(\mathrm{d}u) = \delta_{\bar{u}}(\mathrm{d}u)
.
\end{align*}
\end{assumption}
The integrability condition \eqref{LC integrability} is trivially satisfied for our model with fixed radius and impact. 
\subsection{Scaling and statement of main results}


As in Section~\ref{Scaling limits of the LFV - dynamics of the rare type}  we record two theorems which are a by-product of our technique.

\begin{theorem}
\label{Space neutral theorem}
Suppose that
$X_0^N$ is absolutely continuous with respect to Lebesgue measure, that
the support $\mathrm{supp}(X_0^N)\subseteq D$, where $D$ is a 
compact subset of $\IR^d$ (independent of $N$), and that 
$X_0^N$ converges weakly to $X_0$. Furthermore suppose that the intensity of selective events is $0$ and as $N$ tends to infinity,
\begin{align*}
\frac{C_{d}ur^{d+2}N}{JM^2} \to & C_1 ; \quad
J,K,M \to \infty; \quad \frac{K}{JM^{d}} \to 0;  \quad   \frac{N^2}{M^{d}KJ^2} \to 0; \quad  \frac{u^2 V_R NK}{J^2 M^{d}} \to a,
\end{align*} 
In addition, assume that there exists an $n$ such that, as $N$ tends to infinity,
$N(K/J)^{n} \to 0$.
Then the sequence $X_{N}(t)$ converges weakly to superBrownian motion without drift, initial condition $X_{0}$, diffusion parameter $C_1$ and quadratic variation parameter $2a$.
\end{theorem}

\begin{theorem}
\label{Space selection theorem}
Suppose that
$X_0^N$ is absolutely continuous with respect to Lebesgue measure, that
the support $\mathrm{supp}(X_0^N)\subseteq D$, where $D$ is a 
compact subset of $\IR^d$ (independent of $N$), and that 
$X_0^N$ converges weakly to $X_0$. 
Furthermore, suppose that $\widehat{S} = 1$, $\sigma(\kappa, \zeta) = \sigma(\kappa)$ and as $N$ tends to infinity,
\begin{align*}
\frac{N}{JM^2} \to  C_1 \quad
J,K,M \to \infty; \quad \frac{K}{JM^{d}} \to 0;  \quad   \frac{N^2}{M^{d}KJ^2} \to 0; \quad  \frac{u^2 V_R NK}{J^2 M^{d}}  \to a 
;
\\
\frac{suNV_{R}}{JS}\left(\frac{\sigma(\kappa_r)}{\sigma(\kappa_c)} - 1 \right) \to b,
\end{align*} 
In addition, assume that there exists an $n$ such that, as $N$ tends to infinity,
$N(K/J)^{n} \to 0$.
Then the sequence $X_{N}(t)$ converges weakly  to a critical superBrownian motion with initial condition $X_{0}$, diffusion parameter $C_1$, growth rate $b$ and quadratic variation parameter $2a$.
\end{theorem}

The strategy of the proof is analogous to that laid out in Section~\ref{Subsection Convergence of the non-spatial model}.
We therefore focus on describing the differences.
In contrast to the situation described in the Remark~\ref{NS_almost_full_lookdown_proof} we do not claim to show any results on the limits of sequences of lookdown representations. 
We focus on the result for the projected version of the model. 
We divide the generator of the lookdown representation into two separate parts - one part describes the spatial movement of the particles, the other describes the evolution of the levels.
We show that the selective events (and therefore the fluctuations in the direction of selection) do not affect the movement of the particles in the limit.
For the convergence of the neutral model we refer to results of \cite{chetwynd-diggle/etheridge:2017}.
We then use the lookdown representation to deduce the right form of the limiting generator and to justify the separation of timescales trick for the selective part of the generator.
The part of the proof which deals with the evolution of the levels is analogous to that of Section~\ref{Scaling limits of the LFV - dynamics of the rare type}, and we do not repeat it here. 
For the reader’s convenience we include intensity estimate and discuss the splitting of the generator.

Once again we consider a stopped process, see Remark~\ref{rarety remark}.
The domain of the generator is specified by \eqref{Functional setup SLFV 1}, \eqref{Functional setup SLFV 2}.

\subsection{Intensity estimate} \label{Intensityestimatesection}
As in Section~\ref{subsection intensity estimate 0d} we observe that the generator of the projected process is given by

\begin{multline}
\label{SLFV with funny selection genereator}
\mathcal{L}^{N}f(\langle\phi,X^N_0\rangle)
= 
N M^d
\Bigg[
\int_{\mathbb{R}^d}
\int_{B_r^M(x)} 
\frac{1}{|B_r^M|}
w^N_0(z) f \Bigg( K\frac{u}{J} \int_{B_r^M(x)}  \phi(y)  \mathrm{d} y 
\\
+ K\left(1-\frac{u}{J}\right) \int_{B_r^M(x)} \phi(y) w^N_0(y) \mathrm{d} y  
  \Bigg)   -
f \bigg( K\int_{B_r^M(x)} \phi(y) w^N_0(y) \mathrm{d} y \bigg) \mathrm{d} z \mathrm{d} x\\
+
\int_{\mathbb{R}^d}
\int_{B_r^M(x)}
 \big( 1 - w^N_0(z) \big)
f \Bigg( 
K\left(1-\frac{u}{J}\right) \int_{B_r^M(x)} \phi(y) w^N_0(y) \mathrm{d} y  
 \Bigg) \\
-
f \bigg( K\int_{\mathbb{R}^d} \phi(y) w^N_0(y) \mathrm{d} y \bigg) \mathrm{d} z \;  \mathrm{d} x \Bigg]
\\
+
N \frac{s\widehat{S}}{S} M^d
\Bigg[
\int_{\mathbb{R}^d}
\int_{B_r^M(x)} 
\frac{1}{|B_r^M|}
\frac{\sigma(\kappa_{r},\zeta)w^{N}_{0}(z)}{\sigma(\kappa_{r},\zeta)w^{N}_{0}(z) + \sigma(\kappa_{s},\zeta)\left(1-w^{N}_{0}(z)\right)} f \Bigg( K\frac{u}{J} \int_{B_r^M(x)}  \phi(y)  \mathrm{d} y 
\\
+ K\left(1-\frac{u}{J}\right) \int_{B_r^M(x)} \phi(y) w^N_0(y) \mathrm{d} y  
  \Bigg)   -
f \bigg( K\int_{B_r^M(x)} \phi(y) w^N_0(y) \mathrm{d} y \bigg) \mathrm{d} z \mathrm{d} x\\
+
\int_{\mathbb{R}^d}
\int_{B_r^M(x)}
\frac{\sigma(\kappa_{c},\zeta)(1 - w^{N}_{0}(z))}{\sigma(\kappa_{r},\zeta)w^{N}_{0}(z) + \sigma(\kappa_{s},\zeta)\left(1-w^{N}_{0}(z)\right)}
f \Bigg( 
K\left(1-\frac{u}{J}\right) \int_{B_r^M(x)} \phi(y) w^N_0(y) \mathrm{d} y  
 \Bigg) \\
-
f \bigg( K\int_{\mathbb{R}^d} \phi(y) w^N_0(y) \mathrm{d} y \bigg) \mathrm{d} z \;  \mathrm{d} x \Bigg]
.
\end{multline}
The first two terms represent the part of the generator describing the effect of the neutral events and is the same as in \cite{chetwynd-diggle/etheridge:2017}.
The last two terms represent the effects of selective events.
Substituting $Kw^N_0=X^N_0$ this becomes
\begin{multline}
\label{SLFV with funny selection genereator wtih Kw}
\mathcal{L}^{N}f(\langle\phi,X^N_0\rangle)
= 
\mathcal{L}_{neu}^{N}f(\langle\phi,X^N_0\rangle) + \mathcal{L}_{sel}^{N}f(\langle\phi,X^N_0\rangle)
\\
N M^d
\Bigg[
\int_{\mathbb{R}^d}
\int_{B_r^M(x)} 
\frac{1}{|B_r^M|}
\frac{X^{N}_{0}(z)}{K} f \Bigg( K\frac{u}{J} \int_{B_r^M(x)}  \phi(y)  \mathrm{d} y 
\\
+ \left(1-\frac{u}{J}\right) \int_{B_r^M(x)} \phi(y) X^N_0(y) \mathrm{d} y  
  \Bigg)   -
f \bigg( \int_{B_r^M(x)} \phi(y) X^N_0(y) \mathrm{d} y \bigg) \mathrm{d} z \mathrm{d} x\\
+
\int_{\mathbb{R}^d}
\int_{B_r^M(x)}
 \left( 1 - \frac{X^{N}_{0}(z)}{K} \right)
f \Bigg( 
\left(1-\frac{u}{J}\right) \int_{B_r^M(x)} \phi(y) X^N_0(y) \mathrm{d} y  
 \Bigg) \\
-
f \bigg( \int_{\mathbb{R}^d} \phi(y)X^N_0(y) \mathrm{d} y \bigg) \mathrm{d} z \;  \mathrm{d} x \Bigg]
\\
+
N \frac{s\widehat{S}}{S} M^d
\Bigg[
\int_{\mathbb{R}^d}
\int_{B_r^M(x)} 
\frac{1}{|B_r^M|}
\frac{\sigma(\kappa_{r},\zeta)\frac{X^{N}_{0}(z)}{K} }{\sigma(\kappa_{r},\zeta)\frac{X^{N}_{0}(z)}{K}  + \sigma(\kappa_{s},\zeta)\left(1-\frac{X^{N}_{0}(z)}{K} \right)} f \Bigg( K\frac{u}{J} \int_{B_r^M(x)}  \phi(y)  \mathrm{d} y 
\\
+ \left(1-\frac{u}{J}\right) \int_{B_r^M(x)} \phi(y) X_{0}^{N}(y) \mathrm{d} y  
  \Bigg)   -
f \bigg( \int_{B_r^M(x)} \phi(y) X_{0}^{N}(y) \mathrm{d} y \bigg) \mathrm{d} z \mathrm{d} x\\
+
\int_{\mathbb{R}^d}
\int_{B_r^M(x)}
\frac{\sigma(\kappa_{c},\zeta)\left(1-\frac{X^{N}_{0}(z)}{K} \right)}{\sigma(\kappa_{r},\zeta)\frac{X^{N}_{0}(z)}{K}  + \sigma(\kappa_{s},\zeta)\left(1-\frac{X^{N}_{0}(z)}{K} \right)}
f \Bigg( 
\left(1-\frac{u}{J}\right) \int_{B_r^M(x)} \phi(y) X_{0}^{N}(y) \mathrm{d} y  
 \Bigg) \\
-
f \bigg( \int_{\mathbb{R}^d} \phi(y) X_{0}^{N}(y) \mathrm{d} y \bigg) \mathrm{d} z \;  \mathrm{d} x \Bigg]
.
\end{multline}

We state useful lemmas from \cite{chetwynd-diggle/etheridge:2017}, which describe the form and estimates for the neutral part of the generator. 
\begin{definition} \label{spatialgenerator}
We denote by $\cA^N$ the operator
\begin{align*}
\cA^N(\phi) := \frac{C(d) Nu r^{d+2}}{JM^{2}}
\Delta \phi
,
\end{align*}
with $C(d) := \int_{\cB_1(0)} x^2 \mathrm{d} x$.
\end{definition}

\begin{lemma}[\cite{chetwynd-diggle/etheridge:2017}, Lemma~4.2] 
\label{roughgenerator}
For $f(x,\zeta) = x$ and  $\phi_s(x):\mathbb{R}\times \mathbb{R}^d\to \mathbb{R}\in C_0^{2,3}$,
\begin{align}
\label{rough generator neutral part}
\int_0^t \mathcal{L}_{neu}^{N}(\langle\phi,X^N(s)\rangle)& \mathrm{d} s =  
\int_0^t  \left\langle X^N_s, \dot{\phi}_s \right\rangle + 
	\left\langle X^N_s, \cA^N(\phi_s) \right\rangle \mathrm{d} s + N_t^N(\phi)
,
\end{align}
where 
\begin{align}
\label{estimate on correction term}
|N_t^N(\phi)| \leq
\mathcal{O}\left(\frac{N\sup_{0\leq s\leq t}\|\phi_s\|_{C^3}}{JM^3}
V_{R}
\right)
\int_0^t \left\langle X^N_s, \ind \right\rangle \mathrm{d} s
.
\end{align}
\end{lemma}

We proceed to the proof of Lemma~\ref{Intensity estimate many d}.
We notice that the proof is a simple combination of our proof of Lemma~\ref{Intensity estimate 0d} and the proof of \cite{chetwynd-diggle/etheridge:2017}, Lemma~5.6.
\begin{lemma}
\label{Intensity estimate many d}
Let $X = Kw$ denote the total intensity of individuals of the rare type.  Assume that $\mathbb{E}[X^{N}(0)] < \infty$. Then for any $T > 0$ 
\begin{align}
\label{First part of intensity estimate many d}
\sup_{t \leq T}\sup_{N}\mathbb{E}[\langle X^{N},1 \rangle]  < &  \infty
,
\\
\label{Second part of intensity estimate many d}
\lim_{H \to \infty}\sup_{N}\mathbb{P}\left[\sup_{t \leq T} \langle X^{N},1 \rangle \geq H \right]  = &  0
.
\end{align} 
\end{lemma}
\begin{proof}
We observe that with the test functions chosen as in Lemma~\ref{roughgenerator}, the part
of the generator describing the change in the population resulting from a selective event can be written as 
\begin{multline}
\label{rough generator sel}
\mathcal{L}_{sel}^{N}(\langle\phi,X^N(s)\rangle)   = 
\\
N\frac{us\widehat{S}}{SJ} M^d \left[\int_{\mathbb{R}^{d}}\frac{1}{|B_{r}^{M}|}
\int_{B_{r}^{M}(x)}
\frac{\sigma(\kappa_{r},\zeta)}{\sigma(\kappa_{r},\zeta)\frac{X}{K} + \sigma(\kappa_{c},\zeta)\left(1 - \frac{X}{K} \right)}
\frac{X(z)}{K}
\int_{B_{r}^{M}(x)}\left(K - X(y)\right)\phi(y)\mathrm{d}y\mathrm{d}z\mathrm{d}x
\right.
\\
\left.
 - \int_{\mathbb{R}^{d}}\frac{1}{|B_{r}^{M}|}
\int_{B_{r}^{M}(x)}
\frac{\sigma(\kappa_{r},\zeta)}{\sigma(\kappa_{r},\zeta)\frac{X}{K} + \sigma(\kappa_{c},\zeta)\left(1 - \frac{X}{K} \right)}
 \left(1 -  \frac{X(z)}{K}\right)
\int_{B_{r}^{M}(x)}X(y)\phi(y)\mathrm{d}y\mathrm{d}z\mathrm{d}x
\right] 
\\
=
N\frac{us\widehat{S}}{SJ} M^d \left[\int_{\mathbb{R}^{d}}\frac{1}{|B_{r}^{M}|}
\int_{B_{r}^{M}(x)}
\frac{\sigma(\kappa_{r},\zeta)}{\sigma(\kappa_{r},\zeta)\frac{X}{K} + \sigma(\kappa_{c},\zeta)\left(1 - \frac{X}{K} \right)}
X(z)
\int_{B_{r}^{M}(x)}\phi(y)\mathrm{d}y\mathrm{d}z\mathrm{d}x
\right.
\\
-
\left.
 \int_{\mathbb{R}^{d}}\frac{1}{|B_{r}^{M}|}
\int_{B_{r}^{M}(x)}
\int_{B_{r}^{M}(x)}X(y)\phi(y)\mathrm{d}y\mathrm{d}z\mathrm{d}x
\right].
\end{multline}
We observe that by Taylor's Theorem $\phi(y)$ can be locally approximated by
\begin{align*}
\phi(y) =  \phi(z) + \nabla\phi(z)(y-z) + \|\phi\|_{C^{2}(\mathbb{R}^{d})}\mathcal{O}\left(|y-z|^{2}\right).
\end{align*} 
Therefore by using a calculation analogous to \eqref{0d intensity estimate fraction expansion} and the fact that $|y-z| < R/M$ within a ball $B_{r}(x)$ we may approximate \eqref{rough generator sel} by
\begin{multline}
\label{rough generator sel app}
\mathcal{L}_{sel}^{N}(\langle\phi,X^N(s)\rangle)   = 
\\
N \frac{us\widehat{S}}{SJM^{2d}} 
\left(
\frac{\sigma(\kappa_{r},\zeta)}{\sigma(\kappa_{c},\zeta)} - 1
\right)
\langle
X,\phi
\rangle
+ 
\frac{1}{K}
\langle
X,\phi
\rangle^{2}
\mathcal{O}\left(  
C_{\sigma}N \frac{us\widehat{S}}{SJM^{2d}}
\right)
+ \mathcal{O}\left(\frac{N\widehat{S}}{SJM}\right),
\end{multline}
where $C_{\sigma}$ is a constant depending on $\sigma$.

Let $\chi^{N}(t) = \mathbb{E}[\langle X_{t}^{N}, 1 \rangle]$.
Let $h_{R}$ denote a sequence of smooth  functions such that $h_{R}$ is supported on the ball of radius $2R$ centred at zero and is  equal to $1$ on the ball of radius $R$ centred at zero.

Assume in addition that the sequence $h_{R}$ satisfies 
\begin{align*}
\Delta h_{R} \leq \epsilon; \quad h_{R}\leq h_{R+1}.
\end{align*}

We combine \eqref{rough generator neutral part} and \eqref{rough generator sel app}
and take expectation in \eqref{SLFV with funny selection genereator wtih Kw} to obtain
\begin{multline}
\mathbb{E}\left[\langle X^{N}_{t},\phi \rangle\right]
=
\mathbb{E}\left[\langle X^{N}_{0},\phi \rangle\right]
+
\mathbb{E}\left[
\int_0^t  \left\langle X^N_s, \dot{\phi}_s \right\rangle + 
	\left\langle X^N_s, \cA^N(\phi_s) \right\rangle \mathrm{d} s\right]
\\
	+  N \frac{us\widehat{S}}{SJM^{2d}} 
\mathbb{E}	
	\left[
\left(
\frac{\sigma(\kappa_{r},\zeta)}{\sigma(\kappa_{c},\zeta)} - 1
\right)
\int_{0}^{t}
\langle
X,\phi
\rangle
+
\frac{C_{\sigma}}{K}
\langle
X,\phi
\rangle^{2}
\mathrm{d}s
\right]+ \mathbb{E} \left[N_t^N(\phi)\right]
\\
\leq 
\chi^{N}(0)
+
C\|\Delta h_{R}\|
+
\mathcal{O}\left(\frac{N\sup_{0\leq s\leq t}\|\phi_s\|_{C^3}}{JM^3}
V_{R}
+
N \frac{us\widehat{S}C_{\sigma}}{SJM^{2d}}
\right)
\chi^{N}(0)
,
\end{multline}
where we have used the properties of $h_{R}$.
Letting $R$ tend to infinity and using the Monotone Convergence Theorem, we arrive at
\begin{align*}
\chi^{N}(t) \leq \chi^{N}(0)
+
C\|\Delta h_{R}\|
+
\mathcal{O}\left(\frac{N\sup_{0\leq s\leq t}\|\phi_s\|_{C^3}}{JM^3}
V_{R}
+
N \frac{us\widehat{S}C_{\sigma}}{SJM^{2d}}
\right)
\chi^{N}(0)
.
\end{align*}
We now apply Gr\"onwall's inequality to conclude.
The second part of the statement follows exactly as the second part of the proof of Lemma~\ref{Intensity estimate 0d}.
\end{proof}

\subsection{Sketch of the proof of Theorem~\ref{Space fluctuating selection theorem}}
The generator of the SLFVRE process is given  by
\begin{multline}
\label{noselecgeneratorwithspace}
A^N f(\eta) =
N M^d
\int_{\IR^d}
\Bigg(
\int_0^\infty
\Bigg[
\frac{uK|B_r|}{JM^d} e^{- \frac{uK|B_r|}{JM^d}v^*} \hat{g}_{y, r/M}(\kappa^*, v^*)
e^{- \frac{uK|B_r|}{JM^d} \int_{v^*}^{\infty} (1 - \hat{g}_{y, r/M}(\kappa^*, v)
) \mathrm{d} v}
\\
\prod_{ (x, \kappa , l) \in \eta , l \neq l^* }
g \Big(x , \kappa , \mathcal{J}_{neu}\big(l,l^*,v^*, x, (y, r/M) \big) \Big)
\Bigg] \mathrm{d} v^*
 - f(\eta) \Bigg)
\mathrm{d} y \\
+
N \frac{s\widehat{S}}{S}M^d
\int_{\IR^d}
\Bigg(
\int_0^\infty
\Bigg[
\frac{uK|B_r|}{JM^d} e^{- \frac{uK|B_r|}{JM^d}v^*} \hat{g}_{y, r/M}(\kappa^*, v^*)
e^{- \frac{uK|B_r|}{JM^d} \int_{v^*}^{\infty} (1 - \hat{g}_{y, r/M}(\kappa^*, v)
) \mathrm{d} v}
\\
\prod_{ (x, \kappa , l) \in \eta , l \neq l^* }
g \Big(x , \kappa , \mathcal{J}_{sel}\big(l,l^*,v^*, x, (y, r/M) \big) \Big)
\Bigg] \mathrm{d} v^*
 - f(\eta) \Bigg)
\mathrm{d} y 
+ A^{env}
,
\end{multline}
where $A^{env}$ is specified as in \eqref{chapter4 environmnet generator} and $\hat{g}_{y,r/M}(\kappa, l) := \frac{1}{|B_{r/M}|} \int_{B_{r/M}(y)} g(z, \kappa, l) \mathrm{d}z$.



We split the generator into three parts, by adding and subtracting $g(y,\kappa^{*},v)$ inside the integral.
The first two parts describes the movement of the levels. 
\begin{multline}
\label{A1 definition}
A^N_{full, neu} f(\eta)
\\
=
N M^d
\int_{\IR^d}
\Bigg(
\int_0^\infty
\Bigg[
\frac{uK|B_r|}{JM^d} e^{- \frac{uK|B_r|}{JM^d}v^*} g(x^{*},\kappa^*, v^*)
e^{- \frac{uK|B_r|}{JM^d} \int_{v^*}^{\infty} (1 - \hat{g}_{y, r/M}(\kappa^*, v)
) \mathrm{d} v}
\\
\prod_{ (x, \kappa , l) \in \eta_{B_{r/M}(y)} , l \neq l^* }
g \Big(x^{*},\kappa , \mathcal{J}_{neu}\big(l,l^*,v^*, x, (y, r/M) \big) \Big)
\prod_{ (x, \kappa , l) \notin \eta_{B_{r/M}(y)} , l \neq l^* }
g \Big(x , \kappa , l\Big)
\Bigg] \mathrm{d} v^*
\\ 
 - f(\eta) \Bigg)
\mathrm{d} y
,
\end{multline}
where $ \eta_{B_{r/M}(y)} := \{(x,\kappa, l) \in \eta : x \in B_{r/M}(y)    \} $ and 
\begin{multline}
\label{A3 definition}
A^N_{full, sel} f(\eta)
\\
=
N \frac{s\widehat{S}}{S}M^d
\int_{\IR^d}
\Bigg(
\int_0^\infty
\Bigg[
\frac{uK|B_r|}{JM^d} e^{- \frac{uK|B_r|}{JM^d}v^*}  g(x^{*},\kappa^*, v^*)
e^{- \frac{uK|B_r|}{JM^d} \int_{v^*}^{\infty} (1 - \hat{g}_{y, r/M}(\kappa^*, v)
) \mathrm{d} v}
\\
\prod_{ (x, \kappa , l) \in \eta_{B_{r/M}(y)} , l \neq l^* }
g\Big(x^{*},\kappa , \mathcal{J}_{sel}\big(l,l^*,v^*, x, (y, r/M) \big) \Big)
\prod_{ (x, \kappa , l) \notin \eta_{B_{r/M}(y)} , l \neq l^* }
g \Big(x , \kappa , l\Big)
\Bigg] \mathrm{d} v^*
\\ 
 -f(\eta) \Bigg)
\mathrm{d} y
,
\end{multline}
Observe that $\widehat{g})_{y,R/M}$ is symmetrical with respect to $y$.
Therefore
\eqref{A1 definition}, \eqref{A3 definition} 
can be treated using arguments which are analogous to those which we applied to $A_{neu,1}, A_{neu,2}, A_{sel,1}$ and $A_{sel,2}$ in Section~\ref{Scaling limits of the LFV - dynamics of the rare type}.
The second two parts are given by
\begin{multline}
\label{A2 definition}
A^N_{full,neu,2} f(\eta)
=
N M^d
\int_{\IR^d}
\Bigg(
\int_0^\infty
\Bigg[
\frac{uK|B_r|}{JM^d} e^{- \frac{uK|B_r|}{JM^d}v^*} 
\bigg(
\hat{g}_{y,r/M}(\kappa^*, v^*)
-
g(y, \kappa^*, v^*)
\bigg)
\\
\times
e^{- \frac{uK|B_r|}{JM^d} \int_{v^*}^{\infty} (1 - \hat{g}_{y, r/M}(\kappa^*, v)
) \mathrm{d} v}
\\
\times
\prod_{ (x, \kappa , l) \in \eta , l \neq l^* }
g \Big(x , \kappa , \mathcal{J}_{neu}\big(l,l^*,v^*, x, (y, r/M) \big) \Big)
\Bigg] \mathrm{d} v^*
\Bigg)
\mathrm{d} y
\end{multline}
and
\begin{multline}
\label{A4 definition}
A^N_{full,sel, 2} f(\eta)
=
N \frac{s\widehat{S}}{S}M^d
\int_{\IR^d}
\Bigg(
\int_0^\infty
\Bigg[
\frac{uK|B_r|}{JM^d} e^{- \frac{uK|B_r|}{JM^d}v^*} 
\bigg(
\hat{g}_{y,r/M}(\kappa^*, v^*)
-
g(y, \kappa^*, v^*)
\bigg)
\\
\times
e^{- \frac{uK|B_r|}{JM^d} \int_{v^*}^{\infty} (1 - \hat{g}_{y, r/M}(\kappa^*, v)
) \mathrm{d} v}
\\
\times
\prod_{ (x, \kappa , l) \in \eta , l \neq l^* }
g \Big(x , \kappa , \mathcal{J}_{sel}\big(l,l^*,v^*, x, (y, r/M) \big) \Big)
\Bigg] \mathrm{d} v^*
\Bigg)
\mathrm{d} y.
\end{multline}
We observe that our scaling implies that if the contribution from \eqref{A2 definition} is non-negligible, the contribution from \eqref{A4 definition} vanishes in the limit.
For that reason the results of \cite{chetwynd-diggle/etheridge:2017} are sufficient to deduce the behaviour of the spatial movement of the generator of all cases of interest. 
From
\begin{multline*}
A^N_{full,sel, 2} f(\eta) = NM^d \int_{\IR^d} \int_0^\infty \frac{uK|B_r|}{JM^d} e^{-\frac{uK|B_r|}{JM^d} v^*}
\left(
\hat{g}_{y,r/M}(\kappa^*, v^*) - g(x^*,\kappa^*,v^*)
\right)
\\
\times
e^{-\frac{uK|B_r|}{JM^d}\int_{v^*}^\infty (1 - \hat{g}_{y,r/m}(\kappa^*,v)) \mathrm{d}v}
\prod_{(x,\kappa,l) \in \eta, l \neq l^*}
g(x,\kappa,\cJ)
\mathrm{d} v^* \mathrm{d}y
,
\end{multline*}
we can see as in the calculations in Section~\ref{Scaling limits of the LFV - dynamics of the rare type}
that this leads to
\begin{multline*}
\alpha A^N_{full,sel, 2} h (\Xi)
\\
=
NM^d \int_{\IR^d \times \cK} 
\frac{M^d}{|B_r|}
\int_{B_{r/M}(x^*)}
e^{-\int_{\IR/B_{r/M}(y)} h(z,\kappa) \Xi(\mathrm{d}z, \mathrm{d}\kappa)}
e^{-\left(1 - \frac{u}{J}\right) \int_{B_{r/M}(y)} h(z, \kappa) \Xi(\mathrm{d}z, \mathrm{d} \kappa) }
\\
\times
\left[
e^{- \frac{uK|B_r|}{JM^d} \hat{h}_{y,r/M}(\kappa^*)}
-
e^{- \frac{uK|B_r|}{JM^d} h(x^*,\kappa^*)}
\right]
\mathrm{d}y \;
\Xi(\mathrm{d}x^*, \mathrm{d}\kappa^*)
.
\end{multline*}
Performing a simple Taylor expansion and noting this is identical to calculations appearing in \cite{chetwynd-diggle/etheridge:2017}, Section~4.3 we see that this will give us
\begin{align*}
\alpha A^N_{full,sel, 2} f (\Xi)
= \exp( - \langle h, \Xi \rangle)
\left\langle 
\frac{C(d) u r^d N}{JM^2} \Delta h + \cO(1/M)
,
\Xi
\right\rangle
,
\end{align*}
where we recall that $C(d):= \int_{B_1} x^2 \mathrm{d} x$ and $\langle h, \Xi \rangle := \int_{\IR^d \times \cK} h(x, \kappa) \Xi(\mathrm{d}x,\mathrm{d} \kappa)$.


For the sake of completeness we state the propositions describing the evolution of the levels.
We observe that both \eqref{A1 definition} and \eqref{A3 definition} are  integrals over compact sets (recall that $g(x,l)$ is equal to $1$ outside of a compact set) of their non-spatial counterparts  
studied in Section~\ref{Subsection Convergence of the non-spatial model}. The task of  analysing this generator is then a simple expansion of arguments in that section. 
In particular,  the sparsity condition ${K}/{M^d} \to \infty$ leads to following analogous of  Proposition~\ref{Proposition Neutral Convergence} and Proposition~\ref{selectivepropositionA}.
\begin{proposition}
\label{Proposition Neutral Convergence In Space}
Under the conditions of Theorem~\ref{Space neutral theorem}, 
\begin{multline*}
\IE \Bigg[
\sup_{t \leq T}
\Bigg|
\int_0^t
A_{full, neu}^{N}f(\eta_s)
\\
-
\left(
f(\eta_s^r) 
\sum_{l_i(s) \in \eta_s^N} al_{i}^{2} \frac{\partial_{l}g(x_i,l_{i}(s))}{g(x_i,l_{i}(s))}
+
2a
f(\eta_s^r) \sum_{l_i(s) \in \eta_s^r}
\int_{l_i(s)}^{\infty} \big( 1 - g(x_{i} , v) \big) \mathrm{d}v
\right)
\mathrm{d} s
\Bigg|
\Bigg]
\to 0
.
\end{multline*}
\end{proposition}
\begin{proposition} \label{selectivepropositionAspace}
Under the conditions of Theorem~\ref{Space selection theorem}
for any $T \in \IR$
\begin{align*}
\IE\Bigg[ \sup_{t \leq T} \Bigg| \int_0^t A^N_{full, sel}f(\eta^N_s) - 
f(\eta^N_s) \left(
\sum_{l_i(t) \in \eta_t^N} - bl_{i} \frac{g'(x_i,l_{i})}{g(x_i,l_{i})}
\right)
\mathrm{d} s
\Bigg| \Bigg]
\to 0
.
\end{align*}
\end{proposition}

The proofs of Theorem~\ref{Space fluctuating selection theorem}, Theorem~\ref{Space neutral theorem} and Theorem~\ref{Space selection theorem} now follow by applying Theorem~\ref{kurtzaveraging} to the process characterised by the generator of the projected version of $A_{full, neu}^{N}$, $A_{full, sel}^{N}$ and $A_{env}$.

\begin{appendix}
\section{Poisson random measures}
\label{Section Poisson random measures}

In this section we present some facts about Poisson random measures. 
Most of the facts presented in this section have been stated in the papers on lookdown constructions, see, for example \cite{kurtz/rodrigues:2011}, \cite{etheridge/kurtz:2014}. We present them again as they are useful for many calculations involving lookdown constructions.

%


\begin{lemma}
Let $\xi$ be a Poisson random measure with  mean measure $\nu$. Let $f \in L^{1}(\mathbb{R}^{d}, \nu)$.
Then
\begin{align*}
\mathbb{E} 
\left[
\exp \left( 
\int f(z)\xi(\mathrm{d}z)
 \right)
\right]
=
\exp\left(
\int (e^{f(x)} - 1)\nu(\mathrm{d}x)
\right)
\end{align*}
Similarly, the expected value and variance of   the integral with respect to a Poisson random measure is given by
\begin{align*}
\mathbb{E} 
\left[
\int f(z)\xi(\mathrm{d}z)
\right]
=
\int f(x) \nu(\mathrm{d}x)
\quad
\text{Var} 
\left[
\int f(z)\xi(\mathrm{d}z)
\right]
=
\int f^{2}(x)\nu(\mathrm{d}x)
\end{align*}
\end{lemma}

\begin{definition}[Conditionally Poisson system]
Consider a counting measure $\xi$ on $\mathbb{R}^{d}$. 
Let $\Xi$ be a locally finite random measure on $\mathbb{R}^{d}$. 
We say that $\xi$ is conditionally Poisson with Cox measure $\Xi$ if, conditioned on $\Xi$, $\xi$ is a Poisson random measure with mean measure $\Xi$.
\end{definition}
Conditionally Poisson systems are sometimes referred to as Cox processes. 
We notice that to check that a Poisson random measure $\xi$ is actually a Cox process, it is enough to check that 
\begin{align*}
\mathbb{E}
\left[
\exp
\left(
-\int_{\mathbb{R}^{d}}f \mathrm{d}\xi
\right)
\right]
=
\mathbb{E}
\left[
\exp
\left(
-\int_{\mathbb{R}^{d}}(1 - e^{f})\mathrm{d}\Xi
\right)
\right]
\end{align*}
for all positive Borel-measurable functions $f$.
Our main application of the presented theory is to show convergence of the particles systems to their high intensity limits. 
We need some more definitions, as the convergence of sequences of conditionally Poisson systems requires a rather exotic topology. 

Consider a family of continuous functions $h_{k}:\mathbb{R}^{d} \rightarrow [0,1]$ such that 
\begin{align*}
\bigcup_{k} S_{h_{k}} = \mathbb{R}^{d},
\end{align*}
where $S_{f}$ denotes the support of $f$.
Let $\mathcal{M}_{h_{k}}(\mathbb{R}^{d})$ be the collection of Borel measures on $\mathbb{R}^{d}$ such that
\begin{align*}
\int_{\mathbb{R}^{d}} h_{k} f \mathrm{d}\nu < \infty.
\end{align*} 
Let $\mathrm{d}\nu^{k} = h_{k}\mathrm{d}\nu$. 
The space $\mathcal{M}_{h_{k}}(\mathbb{R}^{d})$  endowed with the topology of weak convergence of $\mathrm{d}\nu^{k}$ is metrizable.
We observe that checking convergence in $\mathcal{M}_{h_{k}}(\mathbb{R}^{d})$ is equivalent to checking convergence of $\int_{\mathbb{R}^{d}}f\mathrm{d}\nu_{n}$ for all bounded and continuous functions which satisfy 
\begin{align*}
\int_{\mathbb{R}^{d}}f \mathrm{d}\nu^{k} < \infty \text{ for } f \leq ch_{k}
\end{align*}
for some constant $c >0$.
The space $\mathcal{M}_{h_{k}}(\mathbb{R}^{d} \times [0,\infty))$ can be defined in a similar way.  

\begin{theorem}[\cite{kurtz/rodrigues:2011}, Theorem A.9]
\label{KR convergence of random measures theorem}
Let $\xi_{n}$ be a sequence of conditionally Poisson random measures on $\mathbb{R}^{d}\times [0,\infty)$ with Cox measures $\lbrace \Xi_{n} \times \Lambda \rbrace$. Then $\xi_{n} \Rightarrow \xi$ in $\mathcal{M}_{h_{k}}(\mathbb{R}^{d} \times [0,\infty))$ if and only if $\Xi_{n} \Rightarrow \Xi$ in $\mathcal{M}_{h_{k}}( \IR^d)$ . If the limit exists, $\xi$ is a conditionally Poisson random measure with Cox measure $\Xi \times \Lambda$.
\end{theorem}
\end{appendix}
\section{Markov Mapping Theorem}
\label{Markov Mapping Theorem}

We recall some basic definitions and introduce the necessary notation. 
For a detailed account of this introductory material we refer to \cite{ethier/kurtz:1986}, Chapter~1, and  \cite{lunardi:2012}.

Let $(E,d)$, $(E_{0},d_{0})$ be a pair of complete, separable metric spaces (with metrics $d$ and  $d_{0}$, respectively).
Let $B(E)$ be the space of bounded measurable functions on $E$. We notice that equipped with the usual supremum norm $\| \cdot \|_{\infty}$ $B(E)$ forms a Banach space. 
Let $C(E) \subset B(E)$ denote the subspace of continuous functions on $E$.
A subspace $A$ of $B(E) \times B(E)$ is a multivalued linear operator.
It domain is given by $\mathcal{D} = \{f : (f,g) \in A \}$ and its range by $\mathcal{R} = \{g: (f,g) \in A\}$.
\begin{definition}[Dissipative operator]
We say that the operator $A$ is dissipative if for each $(f,g) \in A$ and $\lambda > 0$
\begin{align*}
\|\lambda f - g \| \geq \lambda \|f\|.
\end{align*}
\end{definition}
\begin{definition}[Graph separable pre-generator]
\label{definition pre-generator}
We say that an operator $A \subset B(E) \times B(E)$ is a pre-generator if it is dissipative and there exists a sequence of functions $\mu_{n}$ mapping $E$ to the set of probability measures over $\mathcal{P}(E)$, and a sequence of $\lambda_{n} \in E$, such that for each $(f,g) \in A$
\begin{align*}
g(x) = \lim_{n \to \infty} \lambda_{n}\int_{E}\left(f(y) - f(x) \right)\mu_{n}(x,\mathrm{d}y).
\end{align*}
If in addition there exists a countable subset $\{f_{k}\} \subset \mathcal{D}(A) \bigcup C(E)$ such that the graph of $A$ is contained in the closure of the linear span of $(f_{n},Af_{n})$, we say that it is graph-separable.
\end{definition}
We notice that the generators of Markov process are graph-separable pre-generators.
Let $D_{E}[0, \infty)$ denote the space if c\`adl\`ag functions and $M_{E}[0,\infty)$ denote the space of Borel measurable functions from $[0,\infty)$ taking values in $E
$.

\begin{theorem}[\cite{kurtz/rodrigues:2011}, Theorem A.15]
Let $A \subset \overline{C}(E) \times C(E)$  and let $\psi$ be a continuous function taking values in $\mathbb{R}$ such that $\psi \geq 1$.
Suppose that for each $f \in \mathcal{D}(A)$ there exists a constant $c_{f} > 0$ such that 
\begin{align*}
|Af(x)| \leq c_{f}\psi(x).
\end{align*}
Let $A_{0}$ be defined as
\begin{align*}
A_{0}f(x) = \frac{Af(x)}{\psi(x)}.
\end{align*}
Suppose that $A_{0}$ is graph-separable pre-generator  and suppose that $\mathcal{D}(A) = \mathcal{D}(A_{0})$ is closed under multiplication and separating. 
Let $\gamma : E \rightarrow E_{0}$ be  Borel measurable, and let $\alpha$ be a transition function from $E_{0}$ into $E$ satisfying $\alpha(y, \gamma^{-1}(y)) = 1$.
Assume that for each $y \in E_{0}$
\begin{align*}
\widetilde{\psi} = \int_{E}\psi(y,z)\alpha(y,\mathrm{d}z) < \infty.
\end{align*}
and define
\begin{align*}
 C = 
\left\{
\left(
\int_{E}
f(z)\alpha(\cdot, \mathrm{d}z),
\int_{E}
Bf(z) \alpha(\cdot, \mathrm{d}z) 
\right) 
: f \in \mathcal{D}(B)
\right\}.
\end{align*}
Let $\mu_{0} \in \mathcal{P}(E_{0})$, and define $\nu_{0}(y) = \int_{\mathbb{R}}\alpha(y, \cdot) \mu_{0}(\mathrm{d}y)$. 
\begin{enumerate}
\item Let $\tilde{Y}$ be a solution of a martingale problem for $(C,\mu_{0})$. Assume that it satisfies  the moment condition
\begin{align}
\label{MMT moment condition}
\int_{0}^{t}\mathbb{E}\left[ \widetilde{\psi}(\tilde{Y}(s))  \right] \mathrm{d}s < \infty \quad \forall t \geq 0.
\end{align}
Then there exists a solution $X$ of the martingale problem for  $(A,\nu_{0})$ such that $\tilde{Y}$ has the same distribution on $\mathcal{M}_{E_{0}}[0,\infty)$ as $\widetilde{Y} = \gamma \circ Y$.
\item If, in addition, uniqueness holds for the martingale problem for $(A, \nu_{0})$, then uniqueness holds for the martingale problem for  $(C,\mu_{0})$.
\end{enumerate}
\end{theorem}
The original proof of the theorem was inspired by the proofs of generalisations of Burke's Output Theorem appearing in \cite{kliemann/koch/marchetti:1990} and the proof of equivalence of martingale problems for the  Moran model and its lookdown representation  in \cite{donnelly/kurtz:1996}.

Since then the Markov Mapping Theorem has been a useful tool in mathematical population genetics (\cite{etheridge/kurtz:2014}, \cite{kurtz/rodrigues:2011}), mathematical biology (\cite{gupta:2012}), mathematical finance (\cite{stockbridge:2002}) and analysis of infinite dimensional stochastic differential equations (\cite{kurtz:2010}).

The main power of the Markov Mapping Theorem comes in simplifications of proofs of equivalence of seemingly different martingale problems. 
The main source of the power is in exploitation of properties of exchangeable process and conditionally Poisson systems.

\section{Kurtz-Rodrigues' Martingale Lemma}
\label{kurtzrodmartingalelemma}
The following Lemma plays an important role in our applications of the Markov Mapping Theorem. 
Intuitively, it clarifies why the averaged process is a solution to a projected martingale problem.
\begin{lemma}[\cite{kurtz/rodrigues:2011}, Lemma~A.13]
Let $\{ \mathcal{F}_{t}\}$ and $\{\mathcal{G}_{t}\}$ be filtrations with $\mathcal{G}_{t} \subset \mathcal{F}_{t}$.
Suppose that for each $t\geq 0$ 
\begin{align*}
\mathbb{E}\big[ |X_{t}| + \int_{0}^{t}|Y_{s}|\mathrm{d}s \big] < \infty.
\end{align*}
and that 
\begin{align*}
M_{t} = X_{t} - \int_{0}^{t}Y_{s}\mathrm{d}s
\end{align*}
is an $\mathcal{F}_{t}$-martingale.
Then
\begin{align*}
\widehat{M}_{t} = \mathbb{E}
\big[
X_{t}|\mathcal{G}_{t}
\big]
-
\int_{0}^{t} 
 \mathbb{E}
\big[
Y_{s}|\mathcal{G}_{s}
\big]
\mathrm{d}s
\end{align*}
is a $\{\mathcal{G}_{t}\}$martingale.
\end{lemma}

\section{Proofs of Theorem~\ref{LC BBMRE} and Theorem~\ref{LC SBMRE}}
\label{Proofs SBMRE BBMRE}
\begin{proof}[Proof of Theorem~\ref{LC BBMRE}]
Before we can proceed,  additional objects need to be defined. 
Let $\alpha_{\lambda}(n,l)$ be the joint distribution of $n$ i.i.d.~uniformly distributed random variables on $[0,\lambda]$. 
Recall that $\mu$ denotes a point measure representing positions of individuals.
For a test function $f_{1}(x_{i},l_{i})$, we define the projection onto type space $\hat{f}$ as 
\begin{align*}
\hat{f}(\mu) = \prod_{i}\hat{g}(x_{i})=e^{-\sum_{i}\mathcal{I}(g(x_{i}))},
\end{align*}
where the average for a single level is defined as 
\begin{align*}
e^{-\mathcal{I}(g(x_{i}))} = \hat{g}(x_{i}) = \frac{1}{\lambda}\int_{0}^{\lambda}g(x_{i},z)\mathrm{d}z.
\end{align*}
To calculate the generator of the projected model (the generator averaged over the distribution of the levels), we need to evaluate 
\begin{align*}
\int A_{\lambda}f(\zeta,x,l)\alpha_{\lambda}(\mathrm{d}l).
\end{align*}
Let us integrate the four terms appearing in $A_{\lambda}$ separately. 
We begin with the two terms which are least involved - the movement of particles and the environment. 
Since both of those terms do not depend on the levels, integrals with respect to them do not alter our projections, namely
\begin{align}\label{BBMRE movement}
\int f(\zeta,x,l)\sum_{i} \frac{Bg(x_{i},l_{i})}{g(x_{i},l_{i})}\alpha_{\lambda}(\mathrm{d}l) = \sum_{i=1}^{n}B\hat{f}(\zeta,\mu)
= n B\hat{f}(\zeta,\mu)
.
\end{align}
Analogously,
\begin{align}\label{BBMRE environment}
\int\lambda f_{1}(x,l)A^{env}_{\lambda}f_{0}(\zeta,x)\alpha_{\lambda}(n,\mathrm{d}l) = \lambda\hat{f}_{1}(\mu)A^{env}_{\lambda}f_{0}(\zeta,\mu).
\end{align}
In order to evaluate terms describing births and movement of the levels, which both do depend on the exact value of the level, it is convenient to note that (here we follow the calculation on p.~492 in \cite{kurtz/rodrigues:2011})
\begin{align}
\label{BBMRE helpful identity 1}
\lambda^{-1}2a
\int_{0}^{\lambda}g(x,z)
\int_{z}^{\lambda}(g(x,v) - 1)
\mathrm{d}v\mathrm{d}z 
= 
a\lambda e^{-\mathcal{I}_{g}} - 2a\lambda^{-1}\int_{0}^{\lambda}g(x,z)(\lambda - z)\mathrm{d}z,
\end{align}
where we have used Fubini's Theorem, and
\begin{align}
\label{BBMRE helpful identity 2}
\lambda^{-1}
\int_{0}^{\lambda}(az^{2} - \zeta bz)g^{\prime}(x,z)  
\mathrm{d}z 
&= 
-\lambda^{-1}
\int_{0}^{\lambda}
(2az  - \zeta b)(g(z) - 1)
\mathrm{d}z
\nonumber
\\
&=
\lambda^{-1}2a
\int_{0}^{\lambda}zg(x,z)
\mathrm{d}z + a\lambda + b(e^{-\mathcal{I}_{g}} - 1),
\end{align}
where we have integrated by parts. 
It will also be useful to describe the changes in our system due to births and deaths. 
Whenever a birth event occurs, the new individual is located at the same place as the parent.
 If a death occurs, the individual is just removed from the system. 
 Therefore, if we denote the new collection of particles after a birth at location $y$ by $(b(\overline{x}|y))$ and the new collection of particles after a death at location $x_{j}$ by $d(\overline{x}|x_{j})$, we see that 
\begin{align*}
\mu_{b(\overline{x}|y)} =  \delta_{y} + \sum_{i=1}^{n}\delta_{x_{i}}, \quad \mu_{d(\overline{x}|x_{j})} = -\delta_{x_{j}} + \sum_{i=1}^{n}\delta_{x_{i}}. 
\end{align*}

 Armed with these observations and identities \eqref{BBMRE helpful identity 1}, \eqref{BBMRE helpful identity 2} we proceed to evaluate the remaining terms. A simple calculation shows that

\begin{multline}
\label{BBMRE al^2 - bl computation}
\int  f(\zeta,x,l)\left\lbrace\sum_{i}2a\int_{l_{i}}^{\lambda}(g(x_{i},v) - 1)\mathrm{d}v 
+ \sum_{i} (al_{i}^{2} - \sqrt{\lambda}\zeta(x_{i}) bl_{i})\frac{\partial_{l_{i}}g(x_{i},l_{i})}{g(x_{i},l_{i})}\right\rbrace \alpha_{\lambda}(\mathrm{d}l)
\\
= \sum_{j}
e^{\sum_{i\ne j}\mathcal{I}_g(x_{i})}   
\left\lbrace
2a\lambda e^{-\mathcal{I}_{g(x_{j})}}
+2a\lambda^{-1}\int_{0}^{\lambda}g(x_{j},z)\mathrm{d}z 
\right.
\\
\left.
-2a\lambda^{-1}\int_{0}^{\lambda}g(x_{j},z)\mathrm{d}z
+ a\lambda
+ \sqrt{\lambda} b\zeta(x_{j})(e^{\mathcal{I}_{g(x_j)}} - 1)
\right\rbrace
\\
=
a\lambda e^{\sum_{i}\mathcal{I}_g(x_{i})}
\sum_{j}
(e^{-\mathcal{I}_{g(x_{j})}} - 1) 
+ 
\sum_{j}
(\lambda a - \sqrt{\lambda}b\zeta(x_{j}))
e^{\sum_{i\ne j}\mathcal{I}_g(x_{i})} (1- e^{\mathcal{I}_{g(x_j)}})      
\\
= \sum_{i}\lambda a(\hat{f}(\mu_{b(x|x_{i})}) - \hat{f}(\mu))
 + \sum_{i}(\lambda a - \sqrt{\lambda}\zeta(x_{i}) b)(\hat{f}(\mu_{d(x|x_{i})}) - \hat{f}(\mu)).
\end{multline}
Combining \eqref{BBMRE movement}, \eqref{BBMRE environment} and \eqref{BBMRE al^2 - bl computation} we have established that the projected generator can be written as 	
\begin{align*}
\mathcal{L}_{\lambda}\hat{f}(\zeta,\mu) =\lambda \hat{f}_{1}(\mu)A^{env}_{\lambda}f_{0}(\xi) + \sum_{i} B_{x_{i}}\hat{f}(\mu) + \sum_{i}\lambda a\big(\hat{f}(\mu_{b(x|x_{i})}) - \hat{f}(\mu)\big)\\
 + \sum_{i}\big(\lambda a - \zeta(x_{i})\sqrt{\lambda} b\big)\big(\hat{f}(\mu_{d(x|x_{i})}) - \hat{f}(\mu)\big), 
\end{align*}
which is the generator of the BBMRE with  birth rate $\lambda a$ and  death rate $(\lambda a - \sqrt{\lambda}\zeta b)$, as claimed.

Now we only need to check that all assumptions of the  Markov Mapping Theorem are satisfied.  Fortunately our Condition~\ref{LC Condition 3.1} has been imposed  to guarantee just that.
The map  $\gamma : \mathcal{N}(\mathbb{R}^{d}\times[0,\lambda)) \rightarrow \mathcal{M}(\mathbb{R}^{d})$ (mapping counting measures to measures on $\mathbb{R}^{d}$) is given by
\begin{align*}
\gamma\left(\sum_{i}\delta_{x_{i},l_{i}}\right) 
= 
\frac{1}{\lambda}\sum_{l_{i}}\delta_{x_{i}} 
\end{align*}  
The moment condition \eqref{MMT moment condition} is satisfied if we consider $\psi$ of the form
\begin{align}
\label{BBMRE psi}
\psi(x,l) = 1 + \sum_{i} \psi_{B}(x_{i})\left(1 + a + b\right)e^{-l_{i}}
,
\end{align}
so that  the averaged $\tilde{\psi}$ is of the form
\begin{align}
\label{BBMRE psi tilde}
\tilde{\psi}(x) = 1 + \sum_{i} \psi_{B}(x_{i})\left(1 + a + b\right)(1 -e^{-\lambda}).
\end{align}
We note that the  $1$ appearing in the definitions of $\psi$ and $\tilde{\psi}$ has been added only to ensure that both of these functions are greater than or equal to $1$.
\end{proof}

\begin{proof}[Proof of Theorem~\ref{LC SBMRE}]
We define a test function of the form 
\begin{align*}
h_{1}(\zeta,x,l) =  -f_{1}(x,l)b\sum_{i}\zeta(x_{i}) l_{i}\frac{\partial_{l}g(x_{i},l_{i})}{g(x_{i},l_{i})}
\end{align*}
and apply the generator \eqref{Mytnik lambda generator} to a test function of the form $G = f_{1} + \frac{1}{\sqrt{\lambda}}h_{1}$.
This leads to 
\begin{multline}\label{First separation}
A_{\lambda}\left(f_{1}(x,l) +\frac{1}{\sqrt{\lambda}}h_{1}\right) = f_{1}(x,l)\sum_{i} \frac{Bg(x_{i},l_{i})}{g(x_{i},l_{i})} 
\\
 + f_{1}(x,l)\sum_{i}
 2a\int_{l_{i}}^{\lambda}(g(x_{i},v) - 1)\mathrm{d}v   
 + f(x,l)\sum_{i} 
 (al_{i}^{2} - \sqrt{\lambda}\zeta(x_{i}) bl_{i})
 \frac{\partial_{l}g(x_{i},l_{i})}{g(x_{i},l_{i})} 
 \\
+ \sqrt{\lambda}f_{1}(x,l)\sum_{i}\zeta(x_{i}) l_{i}\frac{\partial_{l}g(x_{i},l_{i})}{g(x_{i},l_{i})} 
\\
- \frac{1}{\sqrt{\lambda}} 
\left\lbrace
f_{1}(x,l)b\zeta(x_{i})\sum_{i} \left[\frac{Bg(x_{i},l_{i})}{g(x_{i},l_{i})} + \frac{B\partial_{l}g(x_{i},l_{i})}{g(x_{i},l_{i})}\right]
 \right.
\\ 
+ \left[ \sum_{i}\zeta(x_{i}) l_{i}\frac{\partial_{l}g(x_{i},l_{i})}{g(x_{i},l_{i})}\right] f_{1}(x,l)\sum_{i}
 2a\int_{l_{i}}^{\lambda}(g(x_{i},v) - 1)
\\ 
\left.
+ f_{1}(x,l)\sum_{i}
[al_{i}^{2} - \sqrt{\lambda}b\zeta(x_{i})l_{i}]
\left(
\sum_{j \ne i}l_{j}\zeta(x_{j})\frac{\partial_{l}g(x_{i},l_{i})\partial_{l}g(x_{j},l_{j})}{g(x_{i},l_{j})g(x_{i},l_{i})}\right.\right.
\\
\left.\left.
   + \frac{\zeta(x_{i})\partial_{l}g(x_{i},l_{i}) + \zeta(x_{i})l_{i}\partial^{2}_{l^{2}}g(x_{i},l_{i})}{g(x_{i},l_{i})}
\right)
\right\rbrace
,
\end{multline}
where we have used the fact that $f_{1}$ does not depend on the environment and that $\mathbb{E}_\pi[h_{1}]=0$,
(where $\mathbb{E}_\pi$ is the expected value over the stationary distribution for the environment)
since
$\mathbb{E}_{\pi}[\zeta] = 0$.
Passing to the limit in \eqref{First separation} as  $\lambda$  tends to infinity we obtain
\begin{multline}
A_{\infty}\left(f_{1}\right) = f_{1}(x,l)\sum_{i} \frac{Bg(x_{i},l_{i})}{g(x_{i},l_{i})} 
\nonumber \\
 + f_{1}(x,l)\sum_{i}2a\int_{l_{i}}^{\infty}(g(x_{i},v) - 1)\mathrm{d}v 
+ 
f_{1}(x,l)\sum_{i}
(al_{i}^{2}  - b^{2}l_{i})\frac{\partial_{l_{i}}g(x_{i},l_{i}) }{g(x_{i},l_{i})}
\nonumber \\
- f_{1}(x,l,n)\sum_{i}
b^{2}\zeta(x_{i})l_{i}
\left(\sum_{j \ne i}\zeta(x_{j})l_{j}\frac{\partial_{l_{i}}g(x_{i},l_{i})\partial_{l_{j}}g(x_{j},l_{j})}{g(x_{i},l_{j})g(x_{i},l_{i})}   + \frac{\zeta(x_{i})l_{i}\partial^{2}_{l_{i}^{2}}g(x_{i},l_{i})}{g(x_{i},l_{i})}\right). 
\end{multline}
%
%
%

%

Conditions of Theorem~\ref{kurtzaveraging} are satisfied if we consider $A = A_{\infty}\left(f_{1}\right)$ which would lead too $\varepsilon^{f}_{n} = \mathcal{O}\left(\frac{1}{\lambda}\right)$.
We now show that if we average the levels of the limiting generator, we obtain the generator of the SBMRE of Definition~\ref{Definition SuperBrownian}.  
The general principle is the same as for the proof of Theorem~\ref{LC BBMRE} - we average out the levels and refer to the Markov Mapping Theorem to show that the distribution of the projected process is the distribution of the SBMRE. 

We consider a Poisson random measure with distribution  $\alpha(\mu,\mathrm{d}x\times\mathrm{d}l)$ on $\mathbb{R}^{d}\times \mathbb{R}_{+}$ with mean measure $\mu \times m_{leb}$, where $m_{leb}$ is Lebesgue measure. Just as in the Branching Brownian motion case, we consider a special set of test functions of the form
\begin{align*}
h(x) = \int_{0}^{\infty}(1 - g(x,l))\mathrm{d}l
.
\end{align*}

In this setup, for a test function $f$, the projected (averaged) test function, $\hat{f}$, takes the form
\begin{align*}
\hat{f}(\mu) = \alpha f(\mu) = \int f(x,v)\alpha(\mu,\mathrm{d}x\times\mathrm{d}v) = e^{\int_{\mathbb{R}^{d}}\int_{0}^{\infty}(1 - g(x,v))\mathrm{d}v\mu(\mathrm{d}x)} = e^{-\langle h,\mu\rangle},
\end{align*}
which is a  simple consequence of properties of  Poisson random measures. 
Once again, we integrate the groups of terms that behave similarly separately. Also, to make the calculations easier to read, we write the averaging `level by level' - performing the computation for a single level wherever possible.

Since the part of the generator which describes the movement of the particles does not depend on the value of the level $l$ the averaging is simply
\begin{align}
\label{SBMRE movement}
\alpha\left(f_{1}(x,l)\sum_{i} \frac{Bg(x_{i},l_{i})}{g(x_{i},l_{i})}\right) 
=
  \int_{\mathbb{R}^{d}}-Bh(y)\mu(\mathrm{d}y) e^{-\langle h,\mu\rangle},
\end{align}
We now turn our attention to the terms which behave in a very similar fashion to those in \eqref{BBMRE al^2 - bl computation}. The computation is analogous.
 \begin{align*}
 \int_{\mathbb{R}^{d}} & \int_{0}^{\infty} avg(y,z) \int_{v}^{\infty}(1 - g(y,z))\mathrm{d}z\mathrm{d}v\mu(\mathrm{d}y)e^{-\langle h,\mu\rangle}
 \\
& +\int_{\mathbb{R}^{d}}\int_{0}^{\infty} (av^{2} - b^{2}v)\partial_{v}g(y,v)\mathrm{d}v\mu(\mathrm{d}y)e^{-\langle h,\mu\rangle}
\\
=  &
 \int_{\mathbb{R}^{d}}\int_{0}^{\infty} avg(y,z) \int_{v}^{\infty}(1 - g(y,z))\mathrm{d}z\mathrm{d}v\mu(\mathrm{d}y)e^{-\langle h,\mu\rangle}
\\
& -
  \int_{\mathbb{R}^{d}}\int_{0}^{\infty} (2av - b^{2})g(y,v)\mathrm{d}v\mu(\mathrm{d}y)e^{-\langle h,\mu\rangle} 
\\
  = &
 \int_{\mathbb{R}^{d}}\int_{0}^{\infty} avg(y,z) \int_{v}^{\infty}(1 - g(y,z))\mathrm{d}z\mathrm{d}v\mu(\mathrm{d}y)e^{-\langle h,\mu\rangle}
\\
& -
 \int_{\mathbb{R}^{d}}\int_{0}^{\infty} 2a \int_{v}^{\infty}g(z,v)\mathrm{d}z\mathrm{d}v\mu(\mathrm{d}y)e^{-\langle h,\mu\rangle}
 +
   \int_{\mathbb{R}^{d}}\int_{0}^{\infty}  b^{2}g(y,v)\mathrm{d}v\mu(\mathrm{d}y)e^{-\langle h,\mu\rangle}
\\
= &
 \int_{\mathbb{R}^{d}}a\left(\int_{0}^{\infty}[g(y,v) - 1] \mathrm{d}v\right)^{2}\mu(\mathrm{d}y)e^{-\langle h,\mu\rangle}
  +
   \int_{\mathbb{R}^{d}}\int_{0}^{\infty}  b^{2}g(y,v)\mathrm{d}v\mu(\mathrm{d}y)e^{-\langle h,\mu\rangle}
\\
 = &
 \int_{\mathbb{R}^{d}}\lbrace ah^{2}(y) + b^{2}h(y)\rbrace \mu(\mathrm{d}y) e^{-\langle h,\mu\rangle}, 
\numberthis
 \label{SBMRE 3th term}
 \end{align*}
where we have integrated by parts and used the analogues of identities \eqref{BBMRE helpful identity 1}, \eqref{BBMRE helpful identity 2}.

Finally, the projections of the terms which are a direct consequence of separation of timescales lead to 
\begin{align*}
 \int_{\mathbb{R}^{d}\times\mathbb{R^{d}}} & b^{2}\int_{0}^{\infty}\int_{0}^{\infty}\zeta(y_{1})\zeta(y_{2})v\partial_{v}g(y,v)z\partial_{z}g(y,z)\mathrm{d}v\mathrm{d}z\mu(\mathrm{d}y_{1}) 
\mu(\mathrm{d}y_{2}) e^{-\langle h,\mu\rangle} 
\\
= &
- \int_{\mathbb{R}^{d}\times\mathbb{R}^{d}}b^{2}\zeta(y_{1})\zeta(y_{2})\int_{0}^{\infty}(g(y,v) - 1)\mathrm{d}v\int_{0}^{\infty}(g(y,z)-1)\mathrm{d}v\mathrm{d}z\mu(\mathrm{d}y_{1})\mu(\mathrm{d}y_{2})   e^{-\langle h,\mu\rangle}
\\
= & \int_{\mathbb{R}^{d}\times\mathbb{R}^{d}}b^{2}\zeta(y_{1})\zeta(y_{2})h(y_{1})h(y_{2})\mu(\mathrm{d}y_{1})\mu(\mathrm{d}y_{2}) e^{-\langle h,\mu\rangle},
\numberthis
\label{SBMRE 4th term}
\end{align*} 
where we have integrated by parts, and
\begin{multline}
\label{SBMRE 5th term}
 \int_{\mathbb{R}^{d}}b^{2}\int_{0}^{\infty}v^{2}\partial^{2}_{v}g(y,v)\mathrm{d}v\mu(\mathrm{d}y) e^{-\langle h,\mu\rangle}
 = 
  -\int_{\mathbb{R}^{d}}2b^{2}\int_{0}^{\infty}v\partial_{v}g(y,v)\mathrm{d}v\mu(\mathrm{d}y) e^{-\langle h,\mu\rangle}
\\ 
 =
  -\int_{\mathbb{R}^{d}}2b^{2}\int_{0}^{\infty}(1-g(y,v))\mathrm{d}v\mu(\mathrm{d}y) e^{-\langle h,\mu\rangle}
  = -  \int_{\mathbb{R}^{d}}2b^{2}h(y)\mu(\mathrm{d}y) e^{-\langle h,\mu\rangle},
\end{multline} 
where we have integrated by parts twice.

Combining the calculations \eqref{SBMRE movement}, \eqref{SBMRE 3th term}, \eqref{SBMRE 4th term}, \eqref{SBMRE 5th term} and  appealing to Theorem~\ref{kurtzaveraging} to average over the environment we arrive at
\begin{multline*}
\mathcal{L}\hat{f}(\mu) =  \mathcal{L}e^{\langle f,\mu \rangle} \int_{\mathbb{R}^{d}}\left\lbrace
 -Bh(y) + ah^{2}(y) - b^{2}h(y)
 \right\rbrace
\mu(\mathrm{d}y) e^{-\langle h,\mu\rangle}
\\
+
\int_{\mathbb{R}^{d}\times\mathbb{R}^{d}}b^{2}q(y_{1},y_{2})h(y_{1})h(y_{2})\mu(\mathrm{d}y_{1})\mu(\mathrm{d}y_{2}) e^{-\langle h,\mu\rangle},
\end{multline*}
which is the generator of the SBMRE. 

As before, to  ensure that the solution of the martingale problem for the lookdown process gives us  information about the solution of the martingale problem for the projected process, we need to specify the Markov map $\gamma$ and check that the conditions of the Markov Mapping Theorem are satisfied. 
Once again we appeal to Condition~\ref{LC Condition 3.1}.

The Markov map $\gamma$ is given by
\begin{align*}
\gamma\left(\sum_{i}\delta_{x_{i},l_{i}}\right) 
= 
\begin{cases}
\lim_{\lambda \to \infty} \frac{1}{\lambda}\sum_{l_{i} \leq \lambda}\delta_{x_{i}} & \text{if measures converge}, \\
\mu_{0} & \text{otherwise }. 
\end{cases}
\end{align*}
Our class of test functions is separating over the counting measures, and closed under multiplication.
The moment condition~\ref{MMT moment condition}  is satisfied if we consider $\psi$ of the form \eqref{BBMRE psi}  
and the averaged $\tilde{\psi}$ of the form \eqref{BBMRE psi tilde}.
\end{proof}

\section{Lookdown construction of the Spatial Lambda-Fleming-Viot model}
 \label{Section Spatial Lambda-Fleming-Viot model}
In this section we describe a construction of SLFV model, which is a
 special case of the construction  developed in \cite{etheridge/kurtz:2014}, Section~4.1.3. 
This construction forms the basis for the construction of the SLFV with selection in a fluctuating environment, whose scaling limits we investigate in Section~\ref{Scaling limits of the SLFV - dynamics of the rare type}.
We restrict our attention to the neutral model and do not intend to present any proofs or details.

Let us recall the key elements of the SLFV process. 
We consider a population living in a geographical space,
which, for simplicity, we choose to be $\mathbb{R}^{d}$.
Each individual is assigned a type from a typespace $\mathcal{K}$. 
Let $\mu  =  m_{leb} \times \nu^{1}(w,\mathrm{d}u) \times \nu^{2}(\mathrm{d}w)$ be a measure on $\mathbb{R}^{d} \times [0,1] \times [0,\infty)$, where $ m_{leb}$ is $d$-dimensional Lebesgue measure,  $\nu^{1}$ is a measure which determines \emph{impacts} of the events and $\nu^{2}$ is a $\sigma$-finite measure of event radii which satisfy conditions which are specified in \eqref{VW existence condition}.
Evolution of the population is driven by a Poisson point process $\Pi$ on $ [0,\infty) \times \mathbb{R}^{d} \times [0,1] \times [0,\infty)$ with  mean measure  $m_{leb}\times \mu$. 
Whenever $(t,x,u,r) \in \Pi$, a reproduction event occurs at time $t$ in the closed ball $B_{r}(x)$ (a ball of radius $r$ centred at $x$) with  \emph{impact} $u$. 
The impact of the event determines the proportion of the individuals within the ball $B_{r}(x)$ that are replaced during the event by the offspring of a parent chosen from the ball $B_{r}(x)$ just before the event. 
The locations of new individuals are distributed uniformly over  $B_{r}(x)$.
For the construction to be valid, we assume that
\begin{align}
\label{VW existence condition}
\int_{[0,1]\times (0,\infty)}uw^{d}\nu^{1}(w,\mathrm{d}u)\nu^{2}(\mathrm{d}w) < \infty.
\end{align}

For simplicity, we only describe the construction for a fixed impact $u$ and assume that the radius of the reproduction events is always fixed and equal to $r$.

In  the spirit of lookdown constructions, in addition to a location in geographical space and a type in the typespace, $\mathcal{K}$, each individual is equipped with a level $l \in \mathbb{R}_{+} \cup \{0\}$.
The value of the level impacts the choice of the parent during a reproduction event.

As was the case for the models of Section~\ref{Section SuperBrownian motion in a random environment}, it is convenient to consider our model as a counting measure on $\mathbb{R}^{d}\times\mathcal{K}\times [0,\infty)$, where the first component encodes the geographical space, the second encodes the type of the individual and the third encodes the level of the individual.
The state of the population is given by 
\begin{align*}
\eta = \sum_{i} \delta_{x_{i},\kappa_{i},l_{i}},
\end{align*}
and the single individual $i$ is described by a triple $(x_{i},\kappa_{i},l_{i})$, where $x_{i}$ is the location of the individual, $\kappa_{i}$ is their type and $l_{i}$ is their level. 
In our particular case the levels will always be a conditionally Poisson system with  Cox measure $\Xi \times m_{leb}$, where $m_{leb}$ is the Lebesgue measure on $\mathbb{R}$. 
The measure $\Xi$ is then nothing else but the distribution of locations and types of individuals.

We specify the model in terms of generator.  Let us describe the domains on which our generator are defined,  which turns out to be useful not only for formalizing the constructions in this section but also will serve as a functional setup for the  considerations in Section~\ref{Scaling limits of the SLFV - dynamics of the rare type}. 
Define
\begin{multline}
\label{Functional setup SLFV 1}
\mathcal{D}_{\lambda} = 
\left\lbrace  
f(\eta) = \prod_{x,\kappa,l \in \eta} g(x,\kappa,l) :
\right.
\\
 0 \leq g(x,\kappa,l) \leq 1, g (\cdot,\kappa,l) \in C^{2}(\mathbb{R}^{d}), \|\partial_{l} g(x,\kappa,l) \| < \infty
 \\   
\exists \text{ compact } K_{g} \in \mathbb{R}^{d}, 0 < l_{g}\leq \lambda 
\\
\left. \phantom{\prod_{x,\kappa,l \in \eta}}
g(x,\kappa,l) = 1 \text{ for } (x,l) \notin K_{g}\times[0,l_{g}] 
\right\rbrace,
\end{multline}
and 
\begin{align}
\label{Functional setup SLFV 2}
\mathcal{D}_{\infty} =  \bigcup_{\lambda} \mathcal{D}_{\lambda}.
\end{align}
Our test functions are specified by \eqref{Functional setup SLFV 2}.

\begin{remark}
Notice that our restrictions on domains of the generators are very similar to those in Condition~\ref{LC Condition 3.1}.
This is due to the fact that once again we will apply the Markov Mapping Theorem. 
\end{remark}

The evolution of the population is based on events which are composed of two elements - discrete births and so-called thinning of the population. 
Whenever $(t,x) \in \Pi$, if the number of individuals within $B_{r}(x)$ is greater than zero, the birth event produces offspring, with levels distributed on $[0,\infty)$ according to independent Poisson point processes with intensity $\alpha_{z} =  u V_{r}$ (recall that $V_{r}$ denotes the volume of ball of radius $r$).
The levels of new particles are  denote by $(v_{1}, v_{2},  \dots)$. 
Their locations are distributed uniformly over $B_{r}(x)$.
Let $v^{*}$ be the minimum of $(v_{1}, v_{2},  \dots)$.

Let $(x^{*},\kappa^{*},l^{*})$ denote the element in $\eta$ such that $x^{*} \in B_{r}(y)$  with the smallest level greater than $v^{*}$.
The individual $(x^{*},\kappa^{*},l^{*})$ is chosen as the parent of the event and removed from the population.
All new individuals are assigned a type which is same as the type of the parent.
The levels of old individuals in the population are changed. 
If the level of the individual was smaller than $v^{*}$, it remains unaffected by the birth part of the event.
If the level of the individual was larger than  $v^{*}$, it is moved to $l - l^{*} + v^{*}$. 
The thinning occurs after the movement of the levels due to birth event has been accounted for.
The new individuals are not affected by thinning.
The thinning  takes the new level of each individual present within the ball $B_{r}(x)$ just before the event (apart from the parent), and multiplies it by $1/(1-u)$.

We note that instead of removing the parent from the population we can identify the parent with the lowest offspring (with level $v^*$).
The choice between those two options is a matter of convenience and does not affect the model. 
In our considerations in Section~\ref{Scaling limits of the LFV - dynamics of the rare type} and Section~\ref{Scaling limits of the SLFV - dynamics of the rare type} we find it more convenient to identify the parent with the lowest offspring.

Let $v_{y,r}$ denote the density of the uniform distribution on $B_{r}(y)$.
Let $\mathcal{J}_{EK}$ denote the expected value of the test function evaluated immediately after an event centred at $y$.
It is  given by
\begin{multline*}
\mathcal{J}_{EK}(g,\eta) = \prod_{(x,\kappa,l) \in \eta, x \notin B_{r}(y)}g(x,\kappa,l)
\\
\times 
\int_{0}^{\infty}\left[    
\alpha_{z}
e^{- \alpha_{z} v^{*}}
\int g(x^{\prime},\kappa^{*},v^{*}) v_{y,r}(\mathrm{d}x^{\prime})\right.
\\
\times
\exp\left( -\alpha_{z}
\int_{v^{*}}^{\infty} \left(1 - \int g(x^{\prime},\kappa^{*},v^{*}) v_{y,r}(\mathrm{d}x^{\prime}) \right)\mathrm{d}v^{*}
\right)
\\
\times \prod_{(x,\kappa,l) \in \eta, x \in B_{r}(y), l > l^{*}} g(x,\kappa,\frac{1}{1-u}(l -l^{*} + v^{*}))
\\
\left.
\times \prod_{(x,\kappa,l) \in \eta, x \in B_{r}(y), l < l^{*}}
g\left(x,\kappa,\frac{1}{1-u}l\right)
\right]\mathrm{d}v^{*}
.
\end{multline*}
The generator of the lookdown representation of the SLFV can be now written as 
\begin{align}
\label{Generator EK levels bla}
A_{EK}f(\eta) = \int_{\mathbb{R}^{d}}\ind_{\eta(B_r(y) \times [0,\infty)) > 0} \left\lbrace \mathcal{J}_{EK}(g,\eta) - f(\eta)\right\rbrace\mathrm{d}y.
\end{align}
Recall that $\eta$ is a conditionally Poisson process with Cox measure
$(\Xi(s) \times m_{leb})$. 
To average the generator over the distribution of the levels, we 
define
\begin{align}
\label{definition of h}
h(x,\kappa) = \int_{0}^{\infty}(1 - g(x,\kappa,l))\mathrm{d}l
\end{align} 
and 
\begin{align}
\label{definition of h^*}
h^{*}_{y,r}(\kappa) =\int\left( 1 - \int g(x^{\prime},\kappa,l) v_{y,r}(\mathrm{d}x^{\prime})\right)\mathrm{d}l.
\end{align}

Observe that, by integration by parts, 
\begin{multline*}
\int_{0}^{\infty}
\left\lbrace
\alpha_{z}
e^{-\alpha_{z} v^{*}}
\int g(x^{\prime},\kappa^{*},v^{*}) v_{y,r}(\mathrm{d}x^{\prime})
\right.
\\
\left.
\times
\exp\left( -\alpha_{z}
\int_{v^{*}}^{\infty} \left(1 - \int g(x^{\prime},\kappa^{*},v^{*}) v_{y,r}(\mathrm{d}x^{\prime}) \right)\mathrm{d}v
\right)\right\rbrace\mathrm{d}v^{*}
\\ 
= e^{-\alpha_{z}h^{*}_{y,r}(\kappa^{*})}
\end{multline*}
Therefore if we average out the levels in generator~\eqref{Generator EK levels bla} we obtain
\begin{multline*}
\alpha A_{EK}f(\Xi) = \exp \left( \int_{\mathbb{R}^{d}}h(x,\kappa)\Xi(\mathrm{d}x,\mathrm{d}\kappa) \right)
\\
\times
\int_{\mathrm{R}^{d}} \left\lbrace\mathbb{H}_{2}(h_{y,r}^{*},\Xi)
\exp \left(
u\int_{B_{r}(x)\times \mathcal{K}} h(x,\kappa)\Xi(\mathrm{d}x,\mathrm{d}\kappa)
\right) -1\right\rbrace,
\end{multline*}
where 
\begin{align*}
\mathbb{H}_{2}(h_{y,r}^{*},\Xi) = \frac{1}{\Xi(B_{r}(x)\times \mathcal{K})} \int_{B_{r}(x)\times \mathcal{K}} \exp \left( 
-uV_{r}h^{*}_{y,r}(\kappa)
\right)
\Xi(\mathrm{d}x,\mathrm{d}\kappa).
\end{align*}
The averaged model is simply the usual Spatial Lambda-Fleming-Viot model, see \cite{etheridge/kurtz:2014}.
We also observe that if $\Xi(0,\mathrm{d}x \times \mathcal{K})$ is Lebesgue measure then $\Xi(t,\mathrm{d}x \times \mathcal{K})$ is Lebesgue measure, for arbitrary $t$.
If this is the case, $\mathbb{H}_{2}$ can be written as 
\begin{align*}
 \mathbb{H}_{2}(h_{y,r}^{*}) 
=  \frac{1}{\Xi(B_{r}(x)\times \mathcal{K})} \int_{B_{r}(x)\times \mathcal{K}} \exp \left( 
-uh^{*}_{y,r}(\kappa)\Xi(B_{r}(x)\times \mathcal{K}))
\right)
\Xi(\mathrm{d}x,\mathrm{d}\kappa).
\end{align*}

\addcontentsline{toc}{section}{References}
\bibliographystyle{plainnat}
\DeclareRobustCommand{\VAN}[3]{#3}

\paragraph{Acknowledgements}
{
The authors would like to thank Alison Etheridge and Tom Kurtz for useful discussions and insights.
}

\end{document}